\documentclass[reqno]{amsart}
\usepackage{amsfonts, graphicx, amsmath, amssymb, color, verbatim}
\usepackage[all, 2cell]{xypic}
\usepackage[margin = 1.5in]{geometry}

\newtheorem{theorem}{Theorem}
\newtheorem{lemma}[theorem]{Lemma}
\newtheorem{proposition}[theorem]{Proposition}
\newtheorem{corollary}[theorem]{Corollary}
\newtheorem{remark}[theorem]{Remark}

\numberwithin{equation}{section}

\newcommand \p {\partial}
\newcommand \lra {\longrightarrow}
\newcommand \wt [1]{\widetilde{#1}}
\newcommand \ov [1]{\overline{#1}}

\newcommand{\la}{\langle}
\newcommand{\ra}{\rangle}
\newcommand{\lp}{\left(}
\newcommand{\rp}{\right)}
\newcommand{\set}[1]{\left\{ #1 \right\} }

\newcommand{\norm}[2][]{\left \| #2 \right \|_{#1} }
\newcommand{\Mtan}{\Xi}

\DeclareMathOperator{\cl}{cl}
\DeclareMathOperator{\dist}{dist}
\DeclareMathOperator{\Vol}{Vol}
\DeclareMathOperator{\dVol}{dVol}
\DeclareMathOperator{\End}{End}
\DeclareMathOperator{\eend}{end}

\DeclareMathOperator{\diag}{diag}
\DeclareMathOperator{\id}{Id}
\DeclareMathOperator{\tf}{tf}

\DeclareMathOperator{\ff}{ff}
\DeclareMathOperator{\nff}{nf}
\DeclareMathOperator{\lf}{lf}
\DeclareMathOperator{\rf}{rf}
\DeclareMathOperator{\sff}{sf}
\DeclareMathOperator{\tb}{tb}
\DeclareMathOperator{\phg}{\mathrm {phg}}

\DeclareMathOperator \Rfacet {R}
\DeclareMathOperator \Cfacet {C}
\DeclareMathOperator \Lfacet {L}
\DeclareMathOperator \Tfacet {tb'}
\DeclareMathOperator{\fff}{ff_1}

\DeclareMathOperator{\Tr}{Tr}
\DeclareMathOperator{\codim}{codim}
\newcommand \fffC {\ff_1^C}
\newcommand \fffL {\ff_1^L}
\newcommand \fffR {\ff_1^R}
\newcommand \fffb {\ff_1^\bullet}
\newcommand{\Mtrip}{M^3_{\operatorname{heat}}}
\newcommand{\Mtripa}{M^3_{\operatorname{heat}, 0}}
\newcommand{\Mtripb}{M^3_{\operatorname{heat}, 1}}

\newcommand\bl{{\mathrm b}}
\newcommand\Diff{\mathrm{Diff}}
\newcommand\Diffb{\mathrm{Diff}_{\bl}}

\newcommand{\Mheat}{M^2_{\operatorname{heat}}}
\newcommand{\Xh}{M^2_{\operatorname{heat}}}
\newcommand{\Mheata}{M^2_{\operatorname{heat}, 1}}
\newcommand{\Mheatb}{M^2_{\operatorname{heat}, 2}}

\newcommand{\mm}{\mathfrak{m}}
\newcommand{\pp}{\mathfrak{p}}
\newcommand\re{\operatorname{Re}}

\newcommand{\XX}{\mathcal{X}}
\newcommand{\YY}{\mathcal{Y}}
\newcommand{\TT}{\mathcal{T}}
\newcommand{\PP}{\mathcal{P}}
\newcommand{\PS}{\Psi}

\newcommand{\NPP}{\overline{\mathcal{P}}}
\newcommand{\NPS}{\overline{\Psi}}
\newcommand{\dpi}{\ov{\pi}}
\newcommand{\km}{\textbf{k}}
\newcommand{\Tb}{{}^bT}
\newcommand{\Tbc}{{}^bT^*}
\newcommand\xlra[1]{\xrightarrow{\phantom{x} #1 \phantom{x}}}

\newcommand\fib{\textrm{fib}}
\newcommand\Sym{\operatorname{Sym}}
\newcommand \Ticec {{}^{ice}T^{*}}
\newcommand {\ice}{\mathrm{ice}}
\newcommand \dimY {b}

\newcommand \Lice {{}^{\textrm{ice}}\Lambda^{*}}

\newcommand \Tice {{}^{ice} T}

\DeclareMathOperator \NN {\mathbf{N}}

\newcommand \Endbd {\mbox{End}(\Lice)}

\newcommand \Harmm {\mathcal{H}}
\newcommand \Nff {N_{\ff}(t(\p_{t} + \Delta^{g}))}
\newcommand \Nfff {N_{\fff}(t(\p_{t} + \Delta^{g}))}
\newcommand \NHfff {N_{\fff}(H)}
\newcommand \NHff {N_{\ff}(H)}

\newcommand \push {\sigma}
\newcommand{\wtf}{\mathfrak{t}}

\title{Spectral and Hodge theory of `Witt' incomplete cusp edge spaces}
\author{Jesse Gell-Redman}
\address{Department of Mathematics, Johns Hopkins University}
\email{jgell@math.jhu.edu}
\author{Jan Swoboda}
\address{Mathematisches Institut der LMU M\"unchen}
\email{swoboda@math.lmu.de}
\thanks{The authors are happy to thank Rafe Mazzeo for helpful
  conversations, and to Richard Melrose for explaining to us his work with
Xuwen  Zhu on the Riemann moduli space.  The authors would also like
to thank the Max Planck Institute for Mathematics in Bonn for its support during the
early stages of this project.}

\begin{document}

\begin{abstract}Incomplete cusp edges model the behavior of the
  Weil-Petersson metric on the compactified Riemann moduli space near the interior
  of a divisor.  Assuming such a space is Witt, we construct a fundamental solution to the heat
  equation, and using a precise description of its asymptotic behavior
  at the singular set, we prove that the Hodge-Laplacian on
  differential forms is essentially self-adjoint, with discrete spectrum satisfying Weyl
  asymptotics.  We go on to prove bounds on the growth of
  $L^2$-harmonic forms at the singular set and to prove a Hodge
  theorem, namely that the space of $L^2$-harmonic forms is naturally
  isomorphic to the middle-perversity intersection cohomology.
  Moreover, we develop an asymptotic expansion for the heat trace near $t = 0$.
\end{abstract}

\maketitle

\section{Introduction}

On a compact manifold $M$ with boundary $\p M$ which is the total
space of a fiber bundle
\begin{equation}\label{eq:boundary-fibration}
Z \hookrightarrow \p M \xlra{\pi} Y,
\end{equation}
with $Z, Y$ closed manifolds, an
\emph{incomplete cusp edge metric} $g_{ice}$ is, roughly speaking, a
smooth Riemannian metric on the interior of $M$ which near the
boundary takes the form
\begin{equation}
  \label{eq:fake-incomplete-cusp-edge}
  g_{ice} = dx^2 + x^{2k} g_Z + \pi^* g_Y + \wt{g}, \qquad k > 1,
\end{equation}
where $g_Y$ is a metric on the base $Y$, $g_Z$ is positive definite
restricted to the fibers, $x$ is the distance to the boundary (to
first order), and $\wt{g}$ is a higher order term.    Thus near the boundary $(M, g_{ice})$ is a
bundle of geometric horns over a smooth Riemannian manifold $Y$.  When $k = 3$, such metrics model
the singular behavior of the Weil-Petersson metric on the moduli space
of Riemann surfaces, as we discuss below.

In this paper, we study the Hodge-Laplacian 
\begin{equation}
  \label{eq:Hodge-Laplacian}
  \Delta := \Delta^{g_{ice}} =  d \delta + \delta d
\end{equation}
acting on differential forms.  Our first result shows that under
conditions which contain the main examples of interest, one need not
impose `ideal boundary  conditions' at $\p M$ in order to obtain a
self-adjoint operator.
\begin{theorem}\label{thm:essentiallyselfadjoint}
  Let $(M, g_{ice})$ be an incomplete cusp edge manifold that is `Witt',
  meaning that either $\dim Z = f$ is odd or
  \begin{equation}
    \label{eq:witt}
    H^{f/2}(Z) = \{0\}.
  \end{equation}
  Assume furthermore that $g_{ice}$ satisfies \eqref{eq:cuspedgemetric}
  below and that the parameter $k$ in \eqref{eq:fake-incomplete-cusp-edge} satisfies
  \begin{equation}
    \label{eq:k_3}
    k \ge 3.
  \end{equation}
Then the Hodge-Laplacian $\Delta^{g_{ice}}$ acting on differential forms is essentially   
  self-adjoint and has discrete spectrum.
\end{theorem}
Thus, by the spectral theorem \cite{taylor:vol2}, there exists an orthonormal basis of $L^2(\Omega^p(M))$ of
eigenforms $\Delta^{g_{ice}} \alpha_{j,p} = \lambda^2_{j,p} \alpha_{j,p}$.  We also prove that
the distribution of eigenvalues satisfies ``Weyl asymptotics,''
concretely, for fixed degree $p$
\begin{equation}
  \label{eq:weyl-asymptotics}
  \# \{ j \mid \lambda^2_{j, p} < \lambda^2 \} = c_n \Vol(M, g_{ice}) \lambda^n +
  o(\lambda^n) \mbox{ as } \lambda \to \infty.
\end{equation}
See \textsection \ref{sec:spectral-theory} for the proofs of Theorem \ref{thm:essentiallyselfadjoint} and of the
asymptotic formula in \eqref{eq:weyl-asymptotics}.

Having established these fundamental properties of the Hodge-Laplacian
on such spaces, we turn to the next natural topic: Hodge Theory.
Here the object of study is ``Hodge cohomology'', or the space of $L^2$ harmonic forms, 
\begin{equation}
  \label{eq:hodge-cohomology}
  \mathcal{H}^p_{L^2}(M, g_{ice}) = \set{ \alpha \in L^2(\Omega^p(M), g_{ice}) \mid
    d\alpha = 0 = \delta \alpha},
\end{equation}
and one phrasing of the Hodge theory problem is to find a
parametrization for $\mathcal{H}^*_{L^2}(M, g_{ice})$ in terms of a
topological invariant.   As
described in \cite{HHM2004}, in analogous settings the
relevant topological space for Hodge theoretic statements is not the manifold $M$, but the
stratified space $X$ obtained by collapsing the fibration at the
boundary over the base,
\begin{equation}
  \label{eq:collapsed-manifold}
  X := M / \{ p \sim q \mid p, q \in \p M \mbox{ and } \pi(p) = \pi(q)\}.
\end{equation}
In \textsection\ref{sec:hodge-theory} we will prove the following. 
\begin{theorem}\label{thm:hodge}
For a cusp edge space $(M, g_{ice})$ whose link $Z$ satisfies the Witt
condition \eqref{eq:witt}, there is a natural isomorphism
  \begin{equation}
    \label{eq:Hodge-isomorphism-general}
    \mathcal{H}^*_{L^2}(M, g_{ice}) \simeq I\!H_{\ov{\mm}}(X),
  \end{equation}
where $I\!H_{\ov{\mm}}$ is the middle perversity intersection
cohomology of $X$.
 Furthermore,
differential forms $\gamma \in \mathcal{H}^*_{L^2}(M, g_{ice})$ admit
asymptotic expansions at the boundary of $M.$ 

Moreover, if $Z \simeq \mathbb{S}^{f}$, the sphere of
  dimension $f$, then $X$ is homeomorphic to a differentiable manifold and the
  isomorphism \eqref{eq:Hodge-isomorphism-general} becomes
  \begin{equation}
    \label{eq:Hodge-isomorphism}
    \mathcal{H}^*_{L^2}(M, g_{ice}) \simeq H^*_{dR}(X),
  \end{equation}
where the latter is the de Rham cohomology of $X$.
\end{theorem}

We recall the relevant facts about intersection cohomology, originally
defined by Goresky and MacPherson in \cite{GM1980, GM1983}, in Section
\ref{sec:hodge-theory} below.  The equivalence in
\eqref{eq:Hodge-isomorphism-general} will follow using the arguments
from Hunsicker and Rochon's recent work \cite{HR2012} on iterated
fibered cusp edge metrics (which are \emph{complete}, non-compact
Riemannian manifolds).  To elaborate on the asymptotic expansion for
$L^2$-harmonic forms $\gamma$, we will show in Lemma
\ref{thm:phg-kernel} below that in fact 
$$
\mathcal{H}^*_{L^2}(M, g_{ice}) = \set{  \alpha \in L^2(\Omega^p(M), g_{ice}) \mid
    \Delta^{g_{ice}} \alpha = 0 },
$$
(that the former is included into
the latter is obvious), and we show that elements in the $L^2$ kernel of
$\Delta^{g_{ice}}$ have
expansions at $\p M$ analogous to Taylor expansions but with
non-integer powers, a statement which can be be interpreted as a sort
of elliptic regularity at the boundary of $M.$

One application of these results, and to putative further work we
describe below, is to the analysis on the Riemann moduli spaces $\mathcal M_{\gamma,\ell}$ of
Riemann surfaces of genus $\gamma \geq 0$ with $\ell\geq0$ marked
points. These spaces carry a natural $L^2$ metric, the
Weil-Petersson metric $g_{WP}$, which near the interior of a divisor is an
incomplete cusp edge metric with $k=3$. In general divisors
may intersect with normal crossings, but in at least two cases only
one divisor is present.  
\begin{theorem}\label{thm:moduli-space}
  Let $\mathcal{M}_{1, 1}$ (also known as the moduli space of elliptic
  curves) and $\mathcal{M}_{0, 4}$ be the spaces of, respectively, once
  punctured Riemann surfaces of genus $1$ and $4$ times punctured
  Riemann surfaces of genus zero, modulo conformal diffeomorphism.
  Then the Hodge-Laplacian $\Delta^{g_{WP}}$ on differential forms is essentially self-adjoint on
  $L^2$ with core domain $C^{\infty}_{c, orb}$ (see Section
  \ref{thm:heatkernel}) with discrete spectrum and Weyl asymptotics, and if
  $\ov{\mathcal{M}}_{1,1}$ and $\ov{\mathcal{M}}_{0,4}$
  denote the Deligne-Mumford compactifications (see e.g.\
  \cite{Harris-Morrison, Wolpert-WP-geometry}).  Then the de Rham cohomology
  spaces $H_{dR}(\ov{\mathcal{M}}_{1,1})$ are naturally isomorphic to
  $\mathcal{H}^*_{L^2}(\ov{\mathcal{M}}_{1,1}, g_{WP})$, and the same holds for
  $\ov{\mathcal{M}}_{0, 4}$.
\end{theorem}
We discuss the proof at the end of Section \ref{sec:hodge-theory},
though this is really a direct application
of our results together with the recent work on the structure of the  
 Weil-Petersson metric near a divisor in 
 \cite{MS2015} and \cite{MZ2015}.

 This article is partly motivated by Ji, Mazzeo, M\"uller, and Vasy's work \cite{JMMV2014} on the
 spectral theory of the (scalar) Laplace-Beltrami operator on  the
 Riemann moduli spaces $\mathcal M_{g}$, for which it
 was shown by methods different from ours that it is essentially
 self-adjoint and its eigenvalues satisfy a Weyl asymptotic formula.
 Here they analyze incomplete cusp edge spaces with normal crossings,
 and find in particular that the value $k = 3$ in \eqref{eq:k_3} is
 critical; indeed \emph{for values $k < 3$ one does not expect
   self-adjointness}.  It would be
 interesting (though more complex) to find a parametrization of the space of closed
 extensions of incomplete cusp edge Laplacians with $k < 3$, which is
 expected to be infinite dimensional, e.g.\ by \cite{ALMPII}.  

In contrast with \cite{JMMV2014}, since our eventual goal is Hodge and
index theory on moduli space, our main technical contribution is the construction and detailed
description of the heat kernel $H = \exp(- t \Delta^{g_{\ice}})$.
Indeed, our approach to establishing Theorem
\ref{thm:essentiallyselfadjoint} (which justifies the use of the word
`the' in the previous sentence) and Theorem \ref{thm:hodge}, is
to develop in Theorem \ref{thm:heatkernel} below a precise
understanding of the behavior of a fundamental solution to the heat
equation, which we only conclude is \emph{the} heat kernel after using
it to prove Theorem \ref{thm:essentiallyselfadjoint}; we establish
asymptotic expressions for it at the singular set, uniformly down to time
$t = 0$, obtaining in particular in Corollary \ref{thm:heat-trace}, an asymptotic formula for its trace
(which has potential applications to index theory, since our method
for analyzing $\Delta^{g_{\ice}}$ may be used for other natural elliptic
differential operators on these spaces as well) and fine mapping
properties of $\Delta^{g_{\ice}}$ which allow us to analyze its kernel, i.e.\
harmonic forms.  This is all described in detail in Section \ref{sec:proofs}.

Essential self-adjointness of a differential operator $P$ is typically a
statement about the decay of $L^2$ sections $u$ for which $P u \in
L^2$.  (Here the derivative is taken in the distributional sense.)  The set of such sections is denoted
\begin{equation}
  \label{eq:Dmax}
  \mathcal{D}_{\max} := \mathcal{D}_{\max}(\Delta^{g_{ice}}) = \set{ u \in L^2 \mid P u \in
    L^2}.
\end{equation}
This is the largest subset of $L^2$ which is a closed subspace in the
graph norm $\| u \|_\Gamma = \| u \|_{L^2} + \| P u \|_{L^2}$.
On the other hand, the smallest such closed extension from the domain
$C^\infty_c(M)$ is the closure, i.e. the minimal domain
\begin{equation}
  \label{eq:Dmin}
  \mathcal{D}_{\min} := \mathcal{D}_{\min}(\Delta^{g_{ice}}) = \set{ u \in L^2 \mid
    \exists u_k \in C^\infty_c(M) \mbox{ with }\lim_{k \to \infty} \|
    u_k - u \|_\Gamma = 0}.
\end{equation}
The essential self-adjointness statement in Theorem
\ref{thm:essentiallyselfadjoint} says that the smallest closed
extension is equal to the largest, i.e.\ that
\begin{equation}
  \label{eq:Dmax=Dmin}
  \mathcal{D}_{\max} = \mathcal{D}_{\min},
\end{equation}
and therefore there is exactly one closed extension.  On the other
hand, $\mathcal{D}_{\max}$ is dual to $\mathcal{D}_{\min}$ with respect
to $L^2$ and thus if \eqref{eq:Dmax=Dmin} holds then $P$ with core
domain $C^\infty_c(M)$ has exactly one closed extension, which we
denote by $\mathcal{D} = \mathcal{D}_{\min} = \mathcal{D}_{\max}$ and
$(P, \mathcal{D})$ is a self-adjoint, unbounded operator on $L^2$.
Equation \eqref{eq:Dmax=Dmin} is a statement about decay in the sense
that to prove it we will show
that a differential form $\alpha \in \mathcal{D}_{\max}$ decays fast
enough near $\p M$ that it can be approximated in the graph norm by
compactly supported smooth forms.  This we do using the heat kernel.

Recall that the heat kernel $H$ is a section of the form bundle $\Pi \colon \End(\Lambda) \lra M^\circ \times
M^\circ \times [0, \infty)$, where $M^\circ$ is the interior of $M$
and $\End(\Lambda)$ is the vector bundle whose fiber over $(p, q,
t)$ is $\End(\Lambda^*_q(M) ; \Lambda^*_p(M))$, smooth on the interior
$M^\circ \times M^\circ \times [0, \infty)_t$, which solves   
\begin{equation}
  \label{eq:fundamental-solution}
  (\p_t + \Delta^{g_{ice}}) H = 0 \mbox{ and } H_t \lra \operatorname{Id}, \mbox{strongly as } t
  \downarrow 0.
\end{equation}
For a compactly supported smooth differential form $\alpha$, the
differential form
$$
\beta(\omega, t) := \int_M
H(w, \wt{w}, t)\alpha(\wt{w}) \dVol_{g_{ice}}(\wt{w})
$$ 
solves the heat equation
$(\p_t + \Delta^{g_{ice}}) \beta = 0$ with initial data $\beta
\rvert_{t = 0} = \alpha$.  One
consequence of our precise description of $H$ in Theorem
\ref{thm:heatkernel} below will be the following.
\begin{theorem}\label{thm:heatkernelmap}
On a Witt incomplete cusp edge space $(M, g_{ice})$ with metric
satisfying the assumptions in \eqref{eq:cuspedgemetric} below together
with \eqref{eq:k_3}, there exists a
fundamental solution to the heat equation $H_{t} = H(w, \wt{w}, t)$ in
the sense of \eqref{eq:fundamental-solution} such that for $t>0$
\begin{equation}
  \label{eq:heat-mapsto-Dmin}
  H_{t} \colon L^{2}(M; \Omega^{*}(M)) \lra \mathcal{D}_{\min},
\end{equation}
and such that $H_{t}$ and $\p_{t} H_{t}$ are bounded, self-adjoint
operators on $L^{2}$.
\end{theorem}
Theorem \ref{thm:heatkernelmap} implies the essential
self-adjointness statement; indeed the fundamental solution $H_t$
directly gives a sequence (indeed a path) of sections on
$\mathcal{D}_{\min}$ which approaches a given form in
$\mathcal{D}_{\max}$.  Namely, 
\begin{equation}
  \label{eq:obvious-approximation}
  \alpha \in \mathcal{D}_{\max} \implies H_t \alpha \to \alpha \mbox{
    in } \mathcal{D}_{\min} \mbox{ as } t \downarrow 0.
\end{equation}
As we see now, the proof of this is straightforward functional analysis given the conclusions
of Theorem \ref{thm:heatkernelmap}.
\begin{proof}[Proof of essential self-adjointness using Theorem
  \ref{thm:heatkernelmap}]
The proof has nothing to do with the fine structure of incomplete cusp
edge spaces, it depends only on the soft properties of the
fundamental solution $H$ in Theorem \ref{thm:heatkernelmap}.  To
emphasize this, let $(M,
g)$ be any Riemannian manifold and $P$ a differential
operator of order $2$ acting on sections of a vector bundle $E$ with hermitian
metric $G$, such that $P$ is symmetric on $L^{2}(M; E)$.  For
$t > 0$, let $H_{t}$ be a smooth section of  $\End(E) \to M
\times M$ which depends smoothly on $t$ and satisfies
\begin{equation}
  \label{eq:1}
  (\p_{t} + P) H_{t} = 0 \quad \mbox{ and } \quad \lim_{t \to 0} H_{t} = \id,
\end{equation}
where the latter limit holds in the strong topology on $L^{2}$, and
furthermore such that $H_t$ and $\p_t H_t$ are self-adjoint on $L^2$.

  Let $u \in \mathcal{D}_{\max}(P)$, i.e.~$u \in L^{2}, Pu \in
  L^{2}$.  We will show that $u \in \mathcal{D}_{\min}(P)$ as well, and
  thus $\mathcal{D}_{\min} = \mathcal{D}_{\max}$.  Indeed, we will show
  that
  \begin{equation}
    \label{eq:realdeal}
    H_{t} u \to u \mbox{ in } \mathcal{D}_{\max}, \quad \mbox{i.e.\
      that } H_{t}u \to u
    \mbox{ and } P H_{t} u \to P u \mbox{ in } L^{2}.
  \end{equation}
This suffices to prove that $u \in \mathcal{D}_{\min}$ since $H_{t}u
\in \mathcal{D}_{\min}$ by assumption and $\mathcal{D}_{\min}$ is a
closed subspace of $\mathcal{D}_{\max}$ in the graph norm. To prove \eqref{eq:realdeal}, we note first that $H_{t}u \to u$ in
$L^{2}$ trivially since $H_{t} \to \id$ in the strong topology on
$L^{2}$.  Also note that since $u \in \mathcal{D}_{\max}$, $Pu \in
L^{2}$, so $H_{t} Pu \to Pu $ in $L^{2}$ also.  Of course, this is not
what we want; we want $P H_{t} u \to Pu$, but in fact we claim that
\begin{equation}\label{eq:heat-kernel-commutes}
u \in \mathcal{D}_{\max} \implies P H_{t} u = H_{t} P u,
\end{equation}
which will establish \eqref{eq:Dmax=Dmin}.

It remains to prove \eqref{eq:heat-kernel-commutes}.  Note that for $u \in \mathcal{D}_{\max}$ and $v \in L^{2}$,
then $\la H_{t} P u , v \ra_{L^2} = \la P u , H_{t} v
\ra_{L^2}$ by self-adjointness of $H_t$ on $L^2$, while $\la P u , H_{t} v
\ra_{L^2} = \la u , P H_{t} v
\ra_{L^2}$.  Indeed, the adjoint domain of $\mathcal{D}_{\min}$
is $\mathcal{D}_{\max}$, so for any $f \in \mathcal{D}_{\min}, g \in
\mathcal{D}_{\max}$, $\la P f, g \ra_{L^2} = \la f, Pg \ra_{L^2}$.  But, then since
$PH_t = - \p_t H_t$ we see that $\la H_{t} P u , v
\ra_{L^2} = - \la u , \p_t H_t v \ra_{L^2}$.  But
$\p_t H_t$ is self-adjoint on $L^2$ so $\la u , \p_t H_t v
\ra_{L^2} = \la \p_t H_t u , v \ra_{L^2} = - \la P H_t
u , v \ra_{L^2}$, and thus $\la H_{t} P u , v
\ra_{L^2} = \la P H_{t} u , v
\ra_{L^2}$ for all $u \in \mathcal{D}_{\max}, v \in L^2$, i.e.\
\eqref{eq:heat-kernel-commutes} holds.
\end{proof}

The central vehicle for the construction of the heat kernel is the construction of a manifold with
corners $\Mheat$ via iterated radial blow up of the natural domain of the
heat kernel, namely the space $M \times M \times [0, \infty)_t$; thus
the interiors of these two spaces are diffeomorphic, and the blowup
process furnishes a `blowdown' map
\begin{equation}\label{eq:singularmodelblowdown}
\beta \colon \Mheat \lra M \times M \times [0, \infty)_{t},
\end{equation}
which encodes deeper information about the
relationship between the various boundary hypersurfaces (codimension
one boundary faces) of $\Mheat$
and those of $M \times M \times [0, \infty)_t$.  The upshot is that
the heat kernel $H$, which lives a priori on the latter space, pulls
back via $\beta$ to be ``nice'' (precisely
to be polyhomogeneous, see Appendix \ref{sec:mwc}) on $\Mheat$.  In fact, in
Section \ref{sec:heat-kernel}
\emph{we will construct a parametrix $K$ for the heat equation
  directly on $\Mheat$}.  To obtain the actual heat kernel $H$ we use a Neumann
series argument to iterate away the error.

The latter process builds on what is now a substantial body of work on
analysis (in particular the structure of heat kernels) on singular and
non-compact Riemannian spaces, going back at least to the work of
Cheeger on manifolds with
conical singularities \cite{C1979,C1980,C1983}.  Our approach here is
more closely related to Melrose's geometric microlocal analysis on
asymptotically cylindrical
manifolds \cite{tapsit} (a non-compact example) and Mooers' paper
\cite{Mooers-Edith} on manifolds with conical singularities (an
incomplete, singular example).  The general procedure, which one
sees in both the parabolic and elliptic settings, is to express the
relevant differential operator as an element in the universal
enveloping algebra of a Lie algebra of vector fields, and to `resolve'
this Lie algebra via radial blowup of the underlying space.  

Two papers closely related to our work are, first, Mazzeo-Vertman
\cite{MV2012}, in which the authors study analytic torsion on
incomplete edge spaces, which are
the $k = 1$ case of incomplete cusp edges.  (They analyze the behavior
of elliptic operators which have the same type of degeneracy in their
principal symbols near the singular
set as the Laplacian, so a weider class than considered here.) Their analysis involves a
heat kernel construction using blowup analysis, which is slightly
simpler in their context as the resolved double space has
one less blow up (and thus the triple space is much simpler).  One
substantial difference between the incomplete edge and incomplete cusp
edge cases is that in the incomplete edge case the space of
self-adjoint extensions is substantially more subtle.  In particular,
a Witt space (this is a topological condition and has nothing to do
with the value of $k$) that is incomplete edge may have infinitely
many self-adjoint extensions if the family of induced operators on the
fibers have small non-zero eigenvalues \cite{ALMPII}; on incomplete cusp edge
spaces the small non-zero fiber eigenvalues do not contribute.  On the
other hand, one expects that the zero mode in the fiber (the space of fiber
harmonic forms) makes a simlar contribution in both the cusp and cone
cases, in particular that an incomplete cusp edge space which is not
Witt will have an infinite dimensional space of closed extensions on
which `Cheeger ideal boundary conditions' must be imposed to make the
operator self-adjoint, as is the case in \cite{ALMPII}.

A second closely related
work is Grieser-Hunsicker \cite{Grieser-Hunsicker}, which uses also inhomogeneous radial
blowups, in this case to construct a Green's function for elliptic
operators on a certain class of complete Riemannian manifolds (called
`$\phi$-manifolds') which require similar analysis. There are
many other related works in a similar vein
including, just to name a few, Albin-Rochon
\cite{Albin-Rochon}, Br\"uning-Seeley \cite{Brunin-Seeley},
Gil-Krainer-Mendoza \cite{Gil-Krainer-Mendoza}, Lesch \cite{L1997},
Schultze \cite{Schulze}, and Grieser's notes on parametrix
constructions for heat kernels \cite{Grieser-notes}.  
For analysis of moduli space, to give just a sample recent work, we
refer the reader to the papers of Liu-Sun-Yao, for example
\cite{Liu-Sun-Yau-goodness,Liu-Sun-Yao-New-results}.

\tableofcontents

\section{Incomplete cusp edge differential geometry}\label{sec:differential-geometry}

As described above we work on a smooth manifold with
boundary $M$ whose boundary is the total space of a fiber bundle with
base $Y$ and typical fiber $Z$.  (Again, we assume that $Y, Z$ are
both compact and without boundary.)  We fix once and for all a boundary
defining function $x$, meaning a function $x \in C^{\infty}(M)$ with
$\{ x = 0 \} = \p M$ and $dx$ non-vanishing on $\p M$, and thereby fix
(after possibly scaling $x$ by a constant) a tubular
neighborhood of the boundary, 
\begin{equation}
  \label{eq:tubular}
  \mathcal{U} \simeq \p M \times [0, 1)_x.
\end{equation}
Note that near each point $q$ on the boundary, there is a neighborhood
$V$ in $\partial M$ of $y = \pi(q)$ ($\pi$ is the projection onto the base) such
that $\pi^{-1}(V) \simeq V \times Z$, where $\simeq$ denotes
diffeomorphism.  Below, we will say that we work `locally over the
base' when we fix our attention on a neighborhood of $q$ in $M$ of the
form $\pi^{-1}(V) \times [0, 1)_x$.  If one chooses local coordinates
$y$ on the base and $z$ on $\pi^{-1}(q) \simeq Z$, then
\begin{equation}\label{eq:local-coords-near-boundary}
(x, y, z) \mbox{ form a coordinate chart on $M$ in a neighborhood of $q$.}
\end{equation}
Let
\begin{equation}
  \label{eq:f-and-b}
  f:= \dim Z, \quad b:= \dim Y.
\end{equation}

In the space $X$ defined by
collapsing the fibers $Z$ over the boundary (see \eqref{eq:collapsed-manifold}), the image of $\p M$ via
the projection $\pi \colon M \lra X$ defines a subset $\pi(\p M)$,
which can be identified with $Y$, and $\pi(\p M)$ forms the unique
singular stratum of the stratified space
$X$. (See, for instance \cite{ALMPI} or \cite{DLR2011}, among many
others, for a discussion of stratified spaces.) In the case $Z \simeq \mathbb{S}^f$ a
sphere, $X$ is homeomorphic to a manifold.  In general, given $p \in
\pi(\p M)$, there is a basis of open sets at $p$ consisting of
\textit{canonical neighborhoods} homeomorphic to
\begin{equation}
  \label{eq:canon-neighb}
  V \times C_\delta(Z),
\end{equation}
where $V$ is an open subset of $Y$ containing $p$ and 
$$
C_\delta(Z) = [0, \delta) \times Z / \set{0} \times Z.
$$

\subsection{Differential form bundles and the Hodge-Laplacian}

We will consider differential forms and vector fields which are of
approximately unit
size with respect to Riemannian
metrics of the type in \eqref{eq:fake-incomplete-cusp-edge}.  These
are the incomplete cusp edge forms, which are sections of the
incomplete cusp edge form bundle, $\Lice(M)$, whose smooth sections
are generated locally over the base by the forms
\begin{equation}
  \label{eq:ice-forms}
  dx, \qquad dy_{i}\quad (i=1,\ldots,b=\dim Y), \qquad x^{k} dz_{\alpha}\quad (\alpha=1,\ldots,f=\dim Z).
\end{equation}
Locally over the base, we have the isomorphism
\begin{equation}
  \label{eq:ice-decomp}
  \Lice(M) = \Lambda^{*}(Y) \wedge x^{k \NN}\Lambda^{*}(Z)  \oplus  dx \wedge
  \Lambda^{*}(Y) \wedge x^{k \NN}\Lambda^{*}(Z),
\end{equation}
where $\NN$ is a multiplication  operator 
 on $\Lambda^{*}(Z)$ which takes a
form $\alpha$ of degree $q$ to $q  \alpha$.  Note that with respect to
this decomposition, if $\NN$ is the operator acting on
$\Lambda^{\NN}(Y)$ by multipliction by $\NN$, the exterior derivative
$d$ acts as
\begin{equation}\label{eq:exterior-derivative}
d=\begin{pmatrix}d^Y+ (-1)^{q_Y} x^{-k}d^Z&0\\
\partial_x+kx^{-1}\NN&-d^Y-(-1)^{q_Y} x^{-k}d^Z\end{pmatrix}.
\end{equation}

Moreover, we will use the space of vector fields which are locally $C^\infty(M)$ linear combinations of
the vector fields
\begin{equation}
  \label{eq:ice-fields}
  \p_x, \quad \p_{y_{i}}, \quad x^{-k} \p_{z_{\alpha}}.
\end{equation}
These vector fields are local sections of a bundle $\Tice(M)$ which
is dual to $\Ticec(M) = \Lambda^1_{ice}(M)$.  We denote sections of $\Tice(M)$ by $\mathcal{V}_{ice}$.

We consider metrics $g$ on $M$ which are sections of
$\Sym^{0,2}(\Ticec(M))$, and which locally over the boundary take the form
\begin{equation}
  \label{eq:cuspedgemetric}
  g =  
\lp
\begin{array}{ccc}
  dx & dy^i & x^{k} dz^\alpha
\end{array}
\rp
\lp
\lp
\begin{array}{ccc}
1 & 0 & 0 \\
0 & (h_{ij}) & 0 \\
0 & 0 & (\km_{\alpha \beta})
\end{array}
\rp 
+ O(x^k, g_0)
\rp
\lp
\begin{array}{c}
  dx \\ dy^i \\ x^{k} dz^\alpha
\end{array}
\rp,
\end{equation}
with $h_{ij}$ and  and $\km_{\alpha \beta}$ independent of $x$.
Here $O(x^k, g_0)$ refers to a $O(x^k)$ norm bound with respect to an
exact incomplete cusp edge metric $g_0$ as in
\eqref{eq:exact-cuspedgemetric} below, and furthermore we assume that
the $O(x^k, g_0)$ term is \textit{polyhomogeneous conormal}, a
regularity assumption defined precisely in Section
\ref{sec:mwc} below, which roughly speaking means that the
coefficients have an asymptotic expansion at $x = 0$ analogous to a
Taylor expansion but with non-integer powers and with precise
derivative bounds on the error terms.  Metrics satisfying these
assumptions are what we refer to henceforth as \textbf{incomplete cusp
edge metrics}.  (Note that the assumptions on $g$ are stronger than
merely assuming that $g \in \Sym^{0,2}(\Ticec(M))$, as the latter
space contains e.g.\ $x (x^k dz \otimes_{sym} dx)$, which does not
obey the error bound.)

\begin{remark}
As is shown in \cite{MS2015} (see the introduction for
further discussion) with previous results for example in
\cite{Wolpert-WP-geometry, Yamada} the Weil-Petersson metric on moduli
space takes the form \eqref{eq:cuspedgemetric} near the interior of a
divisor and satisfies the polyhomogeneity assumption.
\end{remark}

As an example of such a metric, one can take an \textit{exact}
incomplete cusp edge metric, constructed as follows.  Let $h_{1} \in
\Sym^{0, 2}(y)$ be a metric on the base, and $h = \pi^{*}h_{1}$
its $x$-independent pullback via the projection $\pi \colon \p M
\lra Y$ extended to our tubular neighborhood $\mathcal{U} \simeq \p M
\times [0, 1)_x$ in the obvious way.  The
tensor $\km$ we take to be the pullback to $\mathcal{U}$ of a tensor on $\p M$ which
by abuse of notation we write as
$\km \in Sym^{0,2}(\p M)$, having the property that $\km$
restricted to any fiber is positive definite.  The metric $g_0 = dx^2 + h +
x^{2k} \km$ is an exact incomplete cusp edge metric. Note that in
coordinates $(x, y, z)$ such a metric takes the form
\begin{equation}
  \label{eq:exact-cuspedgemetric}
  g_0 =  
\lp
\begin{array}{ccc}
  dx & dy^i & x^{k} dz^\alpha
\end{array}
\rp
\lp
\begin{array}{ccc}
1 & 0 & 0 \\
0 & (h_{ij}) & x^k (\km_{i\alpha}) \\
0 & x^k (\km_{\alpha i}) & \km_{\alpha \beta}
\end{array}
\rp
\lp
\begin{array}{c}
  dx \\ dy^i \\ x^{k} dz^\alpha
\end{array}
\rp.
\end{equation}
Thus an incomplete cusp edge metric diffes from an exact one by a
polyhomogeneous error of order $O(x^k)$ in norm and can therefore be
taken to correspond uniquely to an exact cusp edge metric.

Note that an exact incomplete cusp edge metric (and thus an incomplete
cusp metric) gives rise, locally over the boundary to well defined notions of
the operators $d_Y$, $d_Z$, the Hodge star operators $\star_Y$ and
$\star_Z$, obtained by allowing $\star_{(Y,h)}$ the Hodge star associated with $(Y,
h)$, the Riemannian structure on the base $Y$, to act on the
$\Lambda^*(Y)$ factors in the decomposition \eqref{eq:ice-decomp} and
$\star_{(Z, \km)}$, the $y$-dependent Riemannian structure induced on
the fibers, on the $x^{k\NN}\Lambda^*(Z)$ factors.  (So, for
example, $\star_Z x^k dz_\alpha = x^{(f - 1) k} \star_Z dz_\alpha$.)  Thus, we
may also define unambiguously (still locally over the boundary) the Hodge-de Rham operators $\eth_Y$ and
$\eth_Z$.  Indeed, a homogeneous differential form $\alpha$, in the coordinates
$(x, y, z)$ can be written $\alpha = (f(x, y, z) dy^I) \wedge
x^{kp}dz^A(\wedge dx) = \pm(f(x, y, z) x^{kp}dz^A) \wedge dy^I
(\wedge dx)$, where $I$ and $A$ are multi-indices and $|I|$, resp.\
$|A|$,   is their order.  One can act via $\eth_Y = d_Y + (-1)^{\dimY(|I|-1)+1} \star_Y d_Y \star_Y$ on $f(x, y, z) dy^I$, where $\star_Y$ is the
Hodge star operator of $(Y, h)$, and thus define $\eth_Y \alpha$, or
act via $\eth_Z = d_Z + (-1)^{f (|A| -
  1) + 1} \star_Z d_Z \star_Z$ on $f(x, y, z) x^{kp}dz^A$ where $\star_Z$
is the Hodge star operator of $(Z, \km)$, and thus define
$\eth_Z \alpha$.  Note that $\eth_Z$ is thus a \textit{family} of
operators on the fiber $Z$ parametrized by the base point $y$.  
The Hodge-de Rham operator can be decomposed locally over the boundary
in terms of the base and fiber Hodge-de Rham operators according to
the following proposition, which we prove in Section
\ref{sec:proof-of-operator-decomps} below.

\begin{proposition}\label{thm:normaloperatordeRham}
Locally over the boundary (see below \eqref{eq:tubular}), the Hodge-de
Rham operator decomposes as
\begin{equation}
    \begin{split}
      \label{eq:fulldeRham}
      \eth &= \eth_0 + \delta_Z P_1 + P_2 \delta_Z 
      + x^{k -1} E,
    \end{split}
  \end{equation}
where $\eth_0$ acts on sections of $\Lice$ decomposed as in
\eqref{eq:ice-decomp} by
\begin{equation*}
  \eth_0 = \lp
      \begin{array}{cc}
        x^{-k} \eth^{Z} + \eth^{Y} & - \p_{x} - k x^{-1} (f - \NN) \\
        \p_{x} + k x^{-1} \NN & - x^{-k}\eth^{Z} - \eth^{Y}
      \end{array}
      \rp,
\end{equation*}
  $\eth^{Z}$ depends on the base $Y$ parametrically, acts in
  only the fiber direction and only on the $\Lambda^{*}(Z)$ factor, and is
  equal to the Hodge-de Rham operator for the Riemannian manifold 
  $\km \rvert_{y}$, and where the $P_i$ are polyhomogeneous,
  bounded endomorphisms on $\Lice$, and $E$ is as in
  \eqref{eq:errors-in-operator1}-\eqref{eq:errors-in-operator3}.
\end{proposition}

To describe the term $E$, we say first   it is a  
  differential operator of order one which does not increase the order of blow
up of polyhomogeneous distributions; in particular   
\begin{equation}\label{eq:errors-in-operator1}
E( x^k \gamma) =
O(x^k),
\end{equation}
for $\gamma$ a smooth, bounded section of $\Lice$.  In fact, it is a 
$\bl$-differential operators on $\ice$-forms with polyhomogeneous coefficients
\begin{equation}
  \label{eq:errors-in-operator2}
  E  \in \Diff^1_{\bl,\phg}(M; \Lice).
\end{equation}
Letting $\mathcal{V}_{\bl,\phg}$ denote the polyhomogeneous vector
fields tangent to the boundary $\p M$, \eqref{eq:errors-in-operator2}
means that $E$ lies in the algebra of differential operators
generated $\mathcal{V}_{\bl,\phg}$.  Concretely, it satisfies
\begin{equation}
  \label{eq:errors-in-operator3}
  E = a x \p_x + b^i \p_{y^i} + c^{\alpha} \p_{z^\alpha} + d,
\end{equation}
for polyhomogeneous, bounded endomorphisms $a, b^i, c^\alpha, d$, and  
where repeated indices are summed over.  In general, an element $Q \in
\Diff^m_{\bl,\phg}(M; \Lice)$ also satisfies \eqref{eq:errors-in-operator1}, and is
given locally by polyhomogeneous linear combinations of $x \p_x, \p_y,
\p_z$, i.e.\
$$
Q = \sum_{i + |\alpha| + |\beta| \le m} a_{i,\alpha, \beta} (x \p_x)^i
\p_y^\alpha \p_z^\beta,
$$
where $a_{i, \alpha, \beta}$ is a polyhomogeneous bounded endomorphism
of $\Lice$.

The Hodge-Laplacian $\Delta = \eth^2 = d \delta + \delta d$ can now be
decomposed along the same lines.

\begin{proposition}\label{thm:Hodge-Laplacian}
  Locally over the base, $\Delta$ can be
  decomposed as follows
  \begin{equation}
    \label{eq:Hodge-Laplacian-decomp}
    \Delta = \Delta_0 + x^{-k}\wt{P} + x^{-1} \wt{E},
  \end{equation}
where $\Delta_0$ acts on forms decomposed as in \eqref{eq:ice-decomp} by
\begin{align}\label{eq:Hodge-Laplacian-decomp1}
\begin{split}
\Delta_0 &= \operatorname{Id}_{2 \times 2} \lp - \p_{x}^{2} - \frac{kf}{x}
\p_{x} + \frac{1}{x^{2k}} \Delta^{Z,y} + \Delta^{Y} \rp \\ & \qquad
+ \lp
\begin{array}{cc}
  k\NN(1 - k(f - \NN)) x^{-2} & - 2k x^{-k - 1} d^{Z} \\
  - 2 k x^{-k - 1} \delta^{Z} & k(f - \NN)(1 - k \NN) x^{-2}
\end{array} \rp,
\end{split}
\end{align}
where
$$
\wt{P} = \delta_Z Q_1  + d_Z Q_2  + Q_3 \delta_Z + Q_4 d_Z
$$
and where the $Q_i\in \Diff^1_{{\bl,\phg}}(M)$ and $\wt{E} \in \Diff^2_{{\bl,\phg}}(M)$, as
 described below Proposition \ref{thm:normaloperatordeRham}.
\end{proposition}

\subsection{Proof of Propositions \ref{thm:normaloperatordeRham} and
  \ref{thm:Hodge-Laplacian}}\label{sec:proof-of-operator-decomps}

To prove Proposition \ref{thm:normaloperatordeRham}, we will compare
the Hodge--de Rham operator for $g$ to that of the locally defined
warped product metric
\begin{equation}
  \label{eq:local-product-metric}
    g_p =  
\lp
\begin{array}{ccc}
  dx & dy^i & x^{k} dz^\alpha
\end{array}
\rp
\lp
\begin{array}{ccc}
1 & 0 & 0 \\
0 & (h_{ij}) & 0 \\
0 & 0 & (\km_{\alpha \beta})
\end{array}
\rp 
\lp
\begin{array}{c}
  dx \\ dy^i \\ x^{k} dz^\alpha
\end{array}
\rp.
\end{equation}
We compute the Hodge-de Rham operator for $g_p$ now.  First of all,
the Hodge star operator of $g_p$, $\star_p$, acts on forms decomposed
as in \eqref{eq:ice-decomp} as
\begin{equation}
  \label{eq:star-warped-produect}
  \star_p = 
\lp
\begin{array}{cc}
  0 & (-1)^{\NN(b - q_Y)} \star_Y \star_Z \\
 (-1)^{q_Y + \NN(b - q_Y + 1)} \star_Y \star_Z & 0
\end{array}
\rp,
\end{equation}
where $q_Y$ denotes multiplication by the degree in $Y$ and $\NN$
denotes multiplication by the degree in $Z$, and $b = \dim Y$.    
Using \eqref{eq:exterior-derivative}, we see that the Hodge-de Rham
operator for $g_p$ acting on $k$-forms,
$\eth_p = d + (-1)^{k(n - k)}\star_p d \star_p$ satisfies
\eqref{eq:fulldeRham}.  We now use this to prove the proposition.

\begin{proof}[Proof of Proposition \ref{thm:normaloperatordeRham}]
A brief calculation shows that \eqref{eq:fulldeRham} holds with $P_1=P_2=E=0$ for the warped product metric $g_{p}$, defined locally on $M$ in \eqref{eq:local-product-metric}. Next, we claim that  the general case follows from that of a warped product metric $g_{p}$.  To see this, it suffices to show that if
$\star_{p}$ continues to denote the Hodge star operator for $g_{p}$, then the
pointwise norm
\begin{equation}
  \label{eq:19}
  \norm[g]{\star_{p} - \star} \le C x^{k}.
\end{equation}
Indeed, if this holds, then writing 
\begin{equation}
  \label{eq:20}
  \pm \delta = \star \lp dx \wedge \p_{x} \oplus d_{Y} \oplus d_{Z} \rp \star,
\end{equation}
we consider $\delta - \delta_{g_{p}}$, the latter being the dual of
$d$ with respect to $d_{g_p}$, which is given by $\pm 1$ times
\begin{equation}
  \label{eq:21}
  (\star dx \wedge \p_{x} \star - \star_{p} dx \wedge \p_{x} \star_{p})
  \oplus  (\star d_{Y} \star - \star_{p} d_{Y} \star_{p})\oplus( \star
  d_{Z} \star - \star_{p} d_{Z} \star_{p}).
\end{equation}
The operator $d_{Z}$, as an operator on sections of $\Lice$ satisfies
$d_{Z} = \sum a_{i} \p_{z_{i}}$ where $a_{i} $ are endomorphisms
satisfying $a_{i} = O(x^{-k})$.   (Proof: $d_{Z} \alpha = \sum_{\alpha}(x^{k}dz_{\alpha}) \wedge
x^{-k}\p_{z_{\alpha}}\alpha$ and $x^{k} dz_{i}$ is a bounded endomorphism
since $x^{k}dz_{i}$ has unit length.)  By the same rationale, $dx
\wedge \p_{x}$ and $d_{Y}$ are differential operators with bounded
coefficients.  By  \eqref{eq:19}  it follows  that in the difference  
\begin{eqnarray*}
\delta-\delta_p&=&\pm(\ast d \ast- \ast_p d \ast_p)\\
&=&\pm \delta d(\ast-\ast_p) \pm(\ast-\ast_p) d(\ast-\ast_p)\pm (\ast-\ast_p) d\ast\\
&=&\pm\delta\ast (\ast-\ast_p)\pm (\ast-\ast_p)d (\ast-\ast_p)\pm (\ast-\ast_p)\ast\delta
\end{eqnarray*}
the terms coming from $d_Y$ and $\p_x$ derivatives have coefficients
bounded to order $x^k$, while those coming from $d_Z$ have
coefficients with bounded derivatives. Hence \eqref{eq:fulldeRham}  follows.

It remains to prove \eqref{eq:19}.  First note that, under the assumptions above,
if we let $g^{\otimes k,- \otimes l} = g \otimes \dots \otimes g
\otimes g^{-1} \otimes \dots \otimes g^{-1}$ denote the metric induced
by $g$ on $TM^{\otimes k} \otimes T^{*}M^{\otimes l}$, and similarly
for $g_{p}$, then
\begin{equation}
  \label{eq:23}
  \norm[g]{g^{\otimes k,- \otimes l} - g_{p}^{\otimes k,- \otimes l}}
  < C x^{k},
\end{equation}
where the norm is the one induced by $g$ on the tensor powers.
Indeed this is obvious for $l = 0$.  Moreover, $g = g_{p} + O(x^{k})$ 
implies $g^{-1} = g_{p}^{-1} + O(x^{k})$, and \eqref{eq:23} follows.
Now let $e_{i}$ be an oriented orthonormal basis of $1$-forms for $g$, and for
$I = (i_{1}, \dots, i_{k})$ denoting an ordered subset of $\set{1, \dots, n}$, let $e_{I} =
e_{i_{1}} \wedge \dots \wedge e_{i_{k}}$ be the corresponding $k$-form.  Then $\star e_{I} = e_{I^{c}}$ where $I^{c}$ denotes the
complement of $I$ in $\set{1, \dots, n}$ such that $e_{I} \wedge
e_{I^{c}} $ is positive.  Thus, we want to show that $\norm[g]{\star e_{I} -
\star_{p} e_{I} } < c x^{k}$.  For two ordered multi-indices $I, J$, letting
$\delta_{I,J} = 1$ if $I = J$ and zero otherwise, consider
\begin{equation}
  \label{eq:22}
  \begin{split}
    \la \star_{p} e_{I}, e_{J} \ra_{g_{p}} &= \la \la \star_{p} e_{I},
    e_{J} \ra_{g_{p}} \dVol_{g_{p}}, \dVol_{g_{p}} \ra_{g_{p}} \\
    &= \la e_{I} \wedge e_{J}, \dVol_{g_{p}} \ra_{g_{p}} \\
    &=  \pm \delta_{I^{c}, J}\la \dVol_{g}, \dVol_{g_{p}} \ra_{g_{p}} \\
    &= \pm \delta_{I^{c}, J} (1 + O(x^{k})),
  \end{split}
\end{equation}
so since $\la \star_{p} e_{I}, e_{J} \ra_{g} = \la \star_{p}
e_{I}, e_{J} \ra_{g_{p}} + O(x^{k}) = \pm \delta_{I^{c}, J} (1 +
O(x^{k})) = \la \star
e_{I}, e_{J} \ra_{g} + O(x^{k}),$ the bound in \eqref{eq:19} holds.\end{proof}

To see that Proposition \ref{thm:Hodge-Laplacian} holds, simply use
$\Delta = \eth^* \eth$ together with
Proposition \ref{thm:normaloperatordeRham}.

\subsection{Fiber harmonic forms}\label{sec:fibharmforms}  
We now discuss forms which are
approximately harmonic with respect to the family of metrics $\km$ induced
on the fibers $Z$ by an $\ice$-metric as in \eqref{eq:cuspedgemetric}.
We begin by working directly on the fiber bundle $\pi \colon \p M \lra Y$ with a submersion metric, i.e. a
metric $g^\p = \pi^* h_1 + \km$ (so $h_1$ is a Riemannian metric on
$Y$ and $\km$ is assumed only to be
non-degenerate when restricted to the fibers).  Then $\alpha \in
\Omega^q(\p M)$ is \textbf{fiber
  harmonic} if $\alpha \rvert_{\pi^{-1}(p)}$ is harmonic with respect
to $\km \rvert_{\pi^{-1}(p)}$ for every $p \in Y$, i.e.\ if it lies in
the kernel of $\Delta^{Z_y, \km_y}$.  One can obtain fiber harmonic
forms as follow.  Let $\wt{\alpha} \in \Gamma(Y; \Omega^q_{fib})$,
where $\Omega^q_{fib}$ is the bundle whose fiber over a point $p
\in Y$ is the space of $q$-forms on $\pi^{-1}(p)$.  Since a metric
$g^\p$ is present, there is a natural way to extend such a
$\wt{\alpha}$ to a differential form on $\p M$; consider the space
$H^q := L^2(\pi^*\Omega^q(Y)) \subset
L^2(\p M ; \Lambda^q)$ which is the $L^2$ closure of the space of
differential forms pulled back from $Y$, and note that there is a
unique differential form $\wt{\alpha}_V$ on $\p M$ which is at
every point orthogonal to $H^q$ and which pulls back to $\wt{\alpha}$ on
each fiber.  (This is analogous to being a
horizontal vector field; vertical vector fields are defined in
differential terms and the horizontal vector fields are defined with
the metric.  Indeed, being a vertical form here means being dual via
the metric to a vertical vector field).  Moreover, if $\wt{\alpha}$ is
harmonic on each fiber, we claim that $\alpha := \wt{\alpha}_V$ is fiber harmonic.
Namely,  choosing a local
trivialization of $\p M$ near $p \in Y$ and coordinates $(y, z)$, for
a family of $q$-forms $\wt{\alpha}$ on the fibers, the vertical extension satisfies   
\begin{equation}\label{eq:vertical-local}
\wt{\alpha}_V = \sum_{|A| = q} a^A dz_A + \sum_{|I| + |B| = q, |I| > 1}
b^{I, B} dy_I \wedge dz_B,
\end{equation}
with $\wt{\alpha} = \sum_{|A| = q} a^A dz_A$.  Choosing another local
trivialization$(\wt{y}, \wt{z})$ we have $z = z(\wt{y},
  \wt{z}), y = \wt{y}$, $dz_i = (\p
  z^i/\p \wt{z}^j )d\wt{z}^j + f^i dy_i$.  Thus the
  pullback of $\wt{\alpha}_V$ is the change of
  variables for differential forms on the fiber, on which
  $\Delta^{Z,y}$ acts invariantly, proving the claim.

Let $\ker(\Delta^Z)
\lra Y$ denote the bundle whose fiber above each $y$ is the kernel of $\Delta^{Z_y,
  \km_y}$.  This bundle has a grading by the form degree, 
$$
\mathcal{H}^q_{\p} := \ker(\Delta^Z \colon \Omega^q(Z) \lra \Omega^q(Z)),
$$
and by the above paragraph the vertical extensions of these forms also
forms a vector bundle
$\mathcal{H}^q_{\p ,V} \subset \Lambda^q(\p M)$, the sections of which
are fiber harmonic.  We now define the
\textbf{approximately fiber harmonic} differential forms
$\mathcal{H}$ to be the direct sum of the spaces
\begin{equation}
  \label{eq:fiber-harmonic-space}
\mathcal{H} = \oplus_{q = 0}^f \mathcal{H}^q, \mbox{ where }
\mathcal{H}^q :=   \pi^*\Lambda(Y) \wedge x^{kq} \mathcal{H}^q_{\p, V},
\end{equation}
or, in words, the sections of $\mathcal{H}^q$ form the space of differential forms on
$\mathcal{U} \simeq \p M \times [0, 1)_x$ pulled back from $Y$ wedged
with the vertical extentions of harmonic $q$-forms on the fibers,
thought of as living on $\mathcal{U}$ and weighted by $x^{kq}$ (so as
to make them unit size.)

The main point to note about sections of $\mathcal{H}$ in
\eqref{eq:fiber-harmonic-space} is that a form $\gamma \in \mathcal{H}$
satisfies that, locally over the base, we have
\begin{equation}
  \label{eq:8}
  \Delta^{Z, y} \gamma = O(x^k),
\end{equation}
where $\Delta^{Z, y}$ is the operator defined (again only locally over    
the base) in Proposition \ref{thm:Hodge-Laplacian}, and as usual the $O(x^k)$ bound is
a pointwise norm bound.  Indeed, this follows since by
\eqref{eq:vertical-local}, under a change of local
trivialization of the boundary fibration, we have
\begin{equation}\label{eq:local-trivialization-invariance}
x^k dz_i =
x^k (\p z^i/\p \wt{z}^j )d\wt{z}^j + O(x^k),
\end{equation}
so $x^{kq}$ times $\gamma$ is $O(x^k)$
times the $q$-form part of the pullback to the fiber, together with the fact that the
$q$-form part of the pullback is harmonic.

Below we will often work with forms which are merely polyhomogeneous
(as opposed to smooth).  These are sections $\gamma$ of
$\mathcal{A}_{\phg}(\Lice)$, defined rigorously in \eqref{eq:simbound}
below.  As described above, these have an expansion at $x = 0$
analogous to a Taylor expansion, but with non-interger and possibly
also $\log^p(x)$ factors.  In particular we will be forced by possibly
only polyhomogeneous regularity of the metric to work in the 
larger space $x^{s_0} \mathcal{A}_{\phg}(\mathcal{H}^q)$ defined to
be the subset of $\mathcal{A}_{\phg}(\Lice)$ such that
\begin{equation}\label{eq:phg-appro-fib-harm}
x^{s_0} \mathcal{A}_{\phg}(\mathcal{H}^q) \subset x^{s_0} C^\infty(M ;
\mathcal{H}^q) \oplus O(x^{s_0 + k}).
\end{equation}
This space contains, in particular, sections of $x^{s_0}
\mathcal{H}^q$, and also $\ice$-forms
$\gamma$ which are polyhomogeneous and can be written as $\wt{\gamma} = x^{s_0} \gamma_0 +
O(x^{s_0 + k})$, where $\gamma_0 \in \mathcal{H}^q$, but moreover
contains $x^{s_0} \gamma_1 + x^{s_0 + 1/2} \gamma_2 + O(x^{s_0 + k})$ for $\gamma_{i} \in
\mathcal{H}^q$ and indeed any polyhomogeneous $\ice$-form whose
expansion terms not bounded in norm by $x^{s_0 + k}$ have fiber
harmonic coefficients.  Thus, $\gamma
\in x^{s_0} \mathcal{A}_{\phg}(\mathcal{H}^q)$ implies that locally over the base,
\begin{equation}
\Delta^{Z, y} \wt{\gamma} = O(x^{s_0 + k}).\label{eq:fiber-harmonic-error-1}
\end{equation}

It is in fact possible to define an operator $\Delta_{fib}$ on certain forms which is
an invariant version of the locally defined $\Delta^{Z, y}$, for example on smooth sections of
$\Lice(M)$, if one notices that a section $\alpha \in C^\infty(M ;
\Lice(M))$ defines, by restriction to the boundary, a smooth section
of $\Lambda^* Y \otimes \Omega^*_{fib} \lra Y$.  Indeed, this follows
since \eqref{eq:local-trivialization-invariance} and local computation produces a form on $Y$
with values in $\Omega^*(Z)$.  Letting $\wt{\alpha}$ denote this
section, one can then define $\alpha$ to
be $\beta_1 = x^{k\NN}(\Delta^{Z, y} \wt{\alpha})_V$, the vertical extension of
the fiber-wise Laplacian applied to the $\Omega^*(Z)$ factor of
$\wt{\alpha}$, and weighted by $x^{k\NN}$ to make it a smooth section
of $\Lice(M)$.  Then $\beta_1 - \Delta^{Z, y} \alpha = O(x)$ locally
over the base.  One can in fact iterate this to find a section
$\Delta_{fib} \alpha := \beta$ such that $\beta - \Delta^{Z, y} \alpha = O(x^k)$, since if $x
\alpha'$ is a smooth section then letting $\wt{\alpha}'$ denote its
value (as an $\Omega^*(Z)$-valued form on $Y$) the $\ice$-form
$\beta_2 = x \cdot x^{k\NN}(\Delta^{Z, y} \wt{\alpha}')_V$, then
$(\beta_1 + \beta_2) - \Delta^{Z, y} \alpha = O(x^2)$, and this can be iterated up to the
error $O(x^k)$ (and since all steps introduce $O(x^k)$ errors, no
better.)  We will let $\Delta_{fib} \alpha = \beta$.

We will also need to ask when a, say, smooth $\ice$-form can be written
as $\Delta_{fib}$ of another form, or equivalently, when, given a
$\ice$-form $\alpha$, there exists a $\Lice$-form $\gamma$ such that
locally over the base $\alpha - \Delta^{Z, y} \gamma = O(x^k)$.  A
related question is whether there is an invariant definition of
projection onto the space of fiber harmonic forms.  Indeed, given
$\alpha$ a smooth section of $\Lice(M)$, again considering the
$\Omega^*_{fib}$ valued form on $Y$, $\wt{\alpha}$, over the boundary,
if $\Pi^y$ denotes the projection onto the kernel of $\Delta^{Z, y}$,
we can consider $x^{k\NN}(\Pi^{y} \wt{\alpha})$, and again can iterate
this $k$-steps down the Taylor expansion of $\alpha$ to obtain a
projection $\Pi^\mathcal{H}$ onto the fiber harmonic forms.  In
particular, if $\alpha \in x^{s_0} \mathcal{H}_{phg}$, the definition
of $\Pi^{\mathcal{H}}$ can be extended to
\begin{equation}\label{eq:real-fiber-harmonic-projection}
\Pi_{\mathcal{H}} \colon x^{s_0}L^2(\Lice) \lra x^{s_0}\mathcal{H}, \quad
\Pi_{\mathcal{H}} = \Pi_{\mathcal{H}}^{loc} + O(x^k),
\end{equation} 
where $\Pi^{loc}_{\mathcal{H}}$ is the projection onto $\ker(\Delta^{Z, y})$
defined locally over the base,
so that, again locally over the base we have
\begin{equation}
  \label{eq:harmonic-projection}
  \Pi^{\mathcal{H}} d^Z, \Pi^{\mathcal{H}} \delta^Z,
  d^Z \Pi^{\mathcal{H}},  \delta^Z \Pi^{\mathcal{H}}  \in x^k \Diffb^1(M),
\end{equation}
and thus since $\Delta^{Z, y} = d^Z \delta^Z + \delta^Z d^Z$,
\begin{equation}
  \label{eq:harmonic-projection-2}
  \Pi^{\mathcal{H}} \Delta^{Z, y} \Pi^{\mathcal{H}} \in x^{2k} \Diffb^1(M),
\end{equation}
where $\Diffb^1$ is defined in \eqref{eq:errors-in-operator2}.
Moreover, it is straightforward to show that if $\alpha$ is a smooth section
of $x^{s_0} (\Lice(M))$ then 
\begin{equation}
  \label{eq:solvability}
  \Pi^{\mathcal{H}} \alpha = O(x^{s_0 + k}) \implies \exists \gamma
  \mbox{ such that } \Delta^{Z, y} \gamma - \alpha = O(x^{s_0 + k}),
\end{equation}
this holding locally over the base, with the form
$\gamma$ defined globally.

 \section{The heat kernel}\label{sec:heat-kernel}

In this section we construct a manifold with corners
$\Mheat$ as in \eqref{eq:singularmodelblowdown} together with a
fundamental solution to the heat equation which is a polyhomogeneous conormal distribution on $\Mheat$
with prescribed leading order terms in its asymptotic expansions at
the various faces  (see Theorem \ref{thm:heatkernel}).  To do so,
after the construction of $\Mheat$, we perform a parametrix
construction and then use this parametrix to obtain
the fundamental solution itself via a Neumann series.

\subsection{Heat double space}\label{sec:heat-double-space}

The space $\Mheat$ is obtained by performing
three consecutive inhomogeneous radial blowups of $M \times M \times
[0, \infty)_t$.  Such blowups create,
from a given manifold with corners $X$ and some other data including a
submanifold $N$, another manifold with corners $[X; N]_{inhom}$, which is
diffeomorphic to the complement of $N_+(N)$.  Here $N_+(N)$ denotes the  
inward-pointing normal bundle of $N$ which we think of as an
open neighborhood of $N$ in $X$; thus $[X; N]_{inhom}$ is a manifold with corners
with one more boundary hypersurface (bhs) than $X$.  
It comes together with a blowdown map $\beta \colon [X; N]_{inhom} \lra X$
which is a diffeomorphism of the interiors and satisfies that $C^\infty([X; N]_{inhom})$ contains $\beta^*
C^\infty(X)$ properly, i.e.\ there are functions which are not smooth
on $X$ but nonetheless pull back to smooth functions on the blown up
space.  

In more detail, the submanifold $N$ is a
`p-submanifold,'  meaning that there are local coordinates $(x, x', y,
y') \in (\mathbb{R}^+)^k_{(x, x')} \times \mathbb{R}^{n - k}_{(y,
  y')}$ with $x \in (\mathbb{R}^+)^{p + 1}, y \in \mathbb{R}^q$ in
which $N$ is defined (locally) by $x = 0, y = 0$.
The space $[X, N]_{inhom}$ with homogeneities $x_0 \sim  x_1^{\alpha_1} \sim
\dots \sim x_p^{\alpha_p} \sim y_1^{\beta_1} \sim \dots \sim
y_q^{\beta_q}$ with $\alpha_j, \beta_k \in \mathbb{N}$, $1 \le \alpha_1 \le
\dots \le \alpha_p$, $\beta_1 \le \dots \le \beta_q$ is a manifold
with corners whose set of boundary hypersurfaces
contains that of $X$ naturally, and $[X, N]_{inhom}$  
has one new boundary hypersurface, which we call $\nff$ for `new face'.
Assuming for the moment that $\alpha_p \ge \beta_q$ and also that $\alpha_p |
\alpha_j, \alpha_p | \beta_k$ for all $j, k$, the
space $\mathcal{M}([X, N]_{inhom})$ is by definition the set $X
\setminus N \cup \Gamma$  where $\Gamma$ is the set of
paths $\gamma(s)$ in  $X$ with $\gamma(0) \in N$ satisfying that $x_p
= a_p(s) s$ for smooth non-vanishing $s$, and $x_j = a_j(s)
s^{\alpha_p/\alpha_j}, y_k = b_k(s) s^{\alpha_p/\beta_k}$ for smooth
$a_j, b_k$, all other coordinates being smooth functions of $s$,
modulo the equivalence relation $\gamma_1  \sim \gamma_2$ if the
coordinate functions agree to order higher then the stated vanishing
order (e.g.\ $x_j(\gamma_1(s)) - x_j(\gamma_2(s)) = O(s^{(\alpha_p/
  \alpha_j) + 1})$.  The space $X
\setminus N \cup \Gamma$ is naturally isomorphic to $X
\setminus N \cup N^+(N)$ and can be given a smooth structure so that
the following polar coordinates are smooth and valid on a collar neighborhood
of the introduced face $\nff$,
\begin{equation*}
  \begin{split}
    \rho_{\nff} &= (x_0 + x_1^{\alpha_1} + \dots + x_p^{\alpha_p} +
    |y_1|^{\beta_1} + \dots + |y_q|^{\beta_q})^{1/\alpha_p}, \\
   \phi_{\nff} &= (\frac{x_0}{\rho_{\nff}^{\alpha_p}},
    \frac{x_1}{\rho_{\nff}^{\alpha_p/\alpha_1}}, \dots, \frac{y_q}{\rho_{\nff}^{\alpha_p/\beta_q}} ).
  \end{split}
\end{equation*}
A full set of (polar) coordinates is then $(\rho_{\nff}, \phi_{\nff},
x', y')$

For a detailed definition of such spaces we refer to Melrose's
work \cite[Chapter
5]{damwc} which contains a more general construction which does not assume that
one has in particular a fixed extension for the manifold $N$ away from the
boundary, (whereas here we fix once and for all a boundary defining
function $x$ as in \eqref{eq:tubular}, which will give all the desired
extensions below).  See also \cite{Grieser-Hunsicker, Kottke-Melrose}.

	\begin{enumerate}
		\item  We first blow up the fiber diagonal in the
                  corner.  This is the subset of $\p M \times \p M
                  \times \{ 0 \} \subset M \times M \times
[0, \infty)_t$ consisting of points $(p, q, 0)$ with $\pi(p) = \pi(q)$
where $\pi$ is the projection of the fibration $\p M$ onto its base.  If
local coordinates $(x, y, z)$ are chosen as above, this set can be written
                  \begin{equation}
                    \label{eq:first-blowdown}
                    \mathcal{B}_0 := \set{x = \wt{x} = t =
  \dist_h(y,\wt{y}) = 0},
                  \end{equation}
where $\dist_h(\bullet, \bullet')$ is the distance function on the
base $(Y, h)$.  In fact, $\mathcal{B}_0$  is naturally isomorphic to $\diag_{\fib}(\p M)
\times \{ 0 \}$ where $\diag_{\mbox{fib}}(\p M)$ is the ``fiber
diagonal,'' i.e.\ the fiber product of $\p M \times_{\fib} \p M$.
We let
\begin{equation}\label{eq:firstmodel}
\Mheata := [M \times M \times [0, \infty)_{t} ; \mathcal{B}_0]_{inhom},
\end{equation}
with $t \sim x^{2} \sim \wt{x}^{2} \sim |y - \wt{y}|^{2}$.        
To be precise, $\Mheata$ is the parabolic blowup in time of the set
$\mathcal{B}_0$ as defined in \cite[Chapter 7]{tapsit}. In particular there is a
blowdown map $\beta_1 \colon \Mheata \lra M^2 \times [0, \infty)_t$,
and polar coordinates on $\Mheata$ near $\beta_1^{-1}(\mathcal{B}_0)$ (once
coordinates $y, z$ are chosen on $\p M$) are given by   
\begin{equation} 
  \label{eq:polarfirstmodel}
  \begin{split}
    \rho &= \lp t + x^{2} + \wt{x}^{2} + |y - \wt{y}|^{2}\rp^{1/2}, \\
    \phi &= \lp \frac{t}{\rho^{2}} , \frac{x}{\rho},
    \frac{\wt{x}}{\rho}, \frac{y - \wt{y}}{\rho}\rp \\
    &= (\phi_{t}, \phi_{x}, \phi_{\wt{x}}, \phi_{y}),
    \mbox{ along with } \wt{y}, z, \wt{z}.
  \end{split}
\end{equation}
The set $\set{\rho = 0}$ is a
boundary hypersurface on $\Mheata$ introduced by the blow up; we call it
$\fff$; we will see that only the
projection of the heat kernel onto the zero mode in $Z$
is relevant at the face $\fff$.  Letting $s = x/\wt{x}$, the interior of $\fff$ is the total
space of a fiber bundle over $Y \times (0, \infty)_s$,  which is the fiber product $\p M
\times_{\fib} \p M \times_{\fib} TY \times
\mathbb{R}_{t'} $ where $t'$ is a rescaled time variable (see
\eqref{eq:frontfront} below).  Indeed, the map from $\fff$ to the base
$Y$ is simply $\beta_1 \rvert_{\fff}$

		\item  The preceding blow up does not resolve the
                  term $\frac{t}{x^{2k}} \Delta^{Z,y}$ arising from
                  \eqref{eq:Hodge-Laplacian-decomp1}.  To
                  accomplish this, we blow up the subset of $\fff$
                  defined in polar coordinates by
                  \begin{equation}\label{eq:second-blowdown}
                  \mathcal{B}_1 := \set{ \rho = 0, \phi_t =
                    \phi_y = 0, \phi_x = \phi_{\wt{x}} },
                \end{equation}
i.e.\ by $\rho = 0, \phi = (0, 1/\sqrt{2},
1/\sqrt{2}, 0)$,
inhomogeneously so that 
                  near the new face, $\ff$, the function $t/x^{2k}$ is
                  smooth, and furthermore so that $t \p_x^2$ is
                  non-degenerate, the latter condition being satisfied if $(x -
                  \wt{x}) / \sqrt{t}$ is smooth up to the interior of
                  $\ff$.  Near $\mathcal{B}_1$ we can use projective coordinates 
                  \begin{equation}\label{eq:frontfront}
                  \wt{x}, \quad s = x/\wt{x} , \quad \eta = \frac{y -
                    \wt{y}}{\wt{x}}, \quad t' = t /
                  \wt{x}^{2},
                \end{equation}
along with $\wt{y}, z, \wt{z}$.  Let
                  \begin{equation}
                    \label{eq:secondmodel}
                    \Mheatb := [\Mheata ; \fff \cap \ 
                   \mathcal{B}_1 ]_{inhom},
                  \end{equation}
with $t' \sim |\eta|^2   \sim (s - 1)^2 \sim \wt{x}^{2(k - 1)}$,
so we have polar coordinates near
$\ff$ given by 
\begin{equation}
  \label{eq:polarsecondnmodel}
  \begin{split}
    \overline{\rho} &= \lp  (t/\wt{x}^2) + \wt{x}^{2(k - 1)}  + (s - 1)^2
    + (|y - \wt{y}|/\wt{x})^{2} \rp^{1/2(k-1)}, \\
    \overline{\phi} &:= (\overline{\phi}_{t},\overline{\phi}_{\wt{x}}, 
    \overline{\psi}_{x}, \overline{\psi}_{y}) = \lp 
    \frac{t}{\wt{x}^2\overline{\rho}^{2(k-1)}}, \frac{\wt{x}}{\overline{\rho}} ,
    \frac{x - \wt{x}}{\wt{x}\overline{\rho}^{(k-1)}},
     \frac{y - \wt{y}}{\wt{x}\overline{\rho}^{(k-1)}}
    \rp  \mbox{ along with }
    \wt{y}, z, \wt{z}.
  \end{split}
\end{equation}
Let
            \begin{equation}\label{eq:intermediate-blowdown-double}
\beta_{2} \colon
            \Mheatb \lra M \times M \times [0,
\infty)_{t}
\end{equation}
denote the blowdown map.  Then, similar to the setup at $\fff$, if we
define $\sigma = (x - \wt{x})/\wt{x}$, the interior of $\ff$ is a bundle
over $Y \times \mathbb{R}_\sigma$ whose fiber over $p \in Y$ is isomorphic to $T_p Y \times Z^2
\times \mathbb{R}_{\wt{T}}$ for $\wt{T}$ the
rescaled time variable below.

See Remark \ref{thm:explanation} below for further discussion of the
need for the two distinct blown up faces $\ff$ and $\fff$.
	
	\item Finally, we blow up the time equals zero diagonal,
          $\mathcal{B}_2 := \cl(\beta_{2}(\diag(M^\circ) \times \set{t = 0}))$, where
          $\cl$ denotes the closure, parabolically in time.
Note that $\mathcal{B}_2$  intersects the face $\ff$ at $\overline{\phi}= (1, 0 , 0, 0)$, so
            near the intersection, defining the functions
            \begin{equation}
              \label{eq:projectivesecondmodle}
              \wt{x}, \quad \sigma = \frac{s - 1}{\wt{x}^{k-1}} = \frac{x -
                \wt{x}}{\wt{x}^{k}}, \quad \wt{\eta} = \frac{y -
                \wt{y}}{\wt{x}^{k}}, \quad \wt{T} =
              \frac{t'}{\wt{x}^{2(k-1)}} = \frac{t}{\wt{x}^{2k}},
            \end{equation}
we have the projective coordinates
\begin{equation}
  \label{eq:projectivefff}
  \begin{split}
    \wt{x}, \wt{y}, \sigma, \wt{\eta}, \wt{T}, z, \wt{z}.
  \end{split}
\end{equation}
The full heat space is
\begin{equation}
  \label{eq:4new}
  \Mheat = [\Mheatb ; \mathcal{B}_2 ]_{inhom},
\end{equation}
with $\wt{T} \sim \sigma^{2} \sim (z - \wt{z})^{2}$.
The face $\tf$ introduced by the final blowup satisfies
\begin{equation}
  \label{eq:3}
  \tf^{\circ} \simeq \Tice(M),
\end{equation}
where $\Tice(M)$ is the incomplete cusp edge tangent bundle defined in \eqref{eq:ice-fields}.
Concretely, in coordinates $(x, y, z)$ if we set
\begin{equation}
  \label{eq:coordsfftf}
  \xi = \frac{x - \wt{x}}{\sqrt{t}},\quad \eta_{i} = \frac{y_{i} -
    \wt{y}_{i}}{\sqrt{t}},\quad \zeta_{\alpha} = \frac{z_{\alpha} - \wt{z}_{\alpha}}{\sqrt{t}}
  \ \wt{x}^{k},\quad \tau = \frac{\sqrt{t}}{\wt{x}^{k}},
\end{equation}
then $(x, y, z, \xi, \eta, \zeta, \tau)$ (or $(\wt{x}, \wt{y}, \wt{z}, \xi, \eta, \zeta, \tau)$)  form local coordinates near
the intersection of $\tf$ with $\ff$ and away from $t = 0$, and the
association $\xi \mapsto \p_{x}, \eta_{i} \mapsto \p_{y_{i}},
\zeta_{\alpha} \mapsto x^{-k}\p_{z_{\alpha}}$ induces the map.  
	\end{enumerate}

\begin{figure}[htbp]
\begin{center}
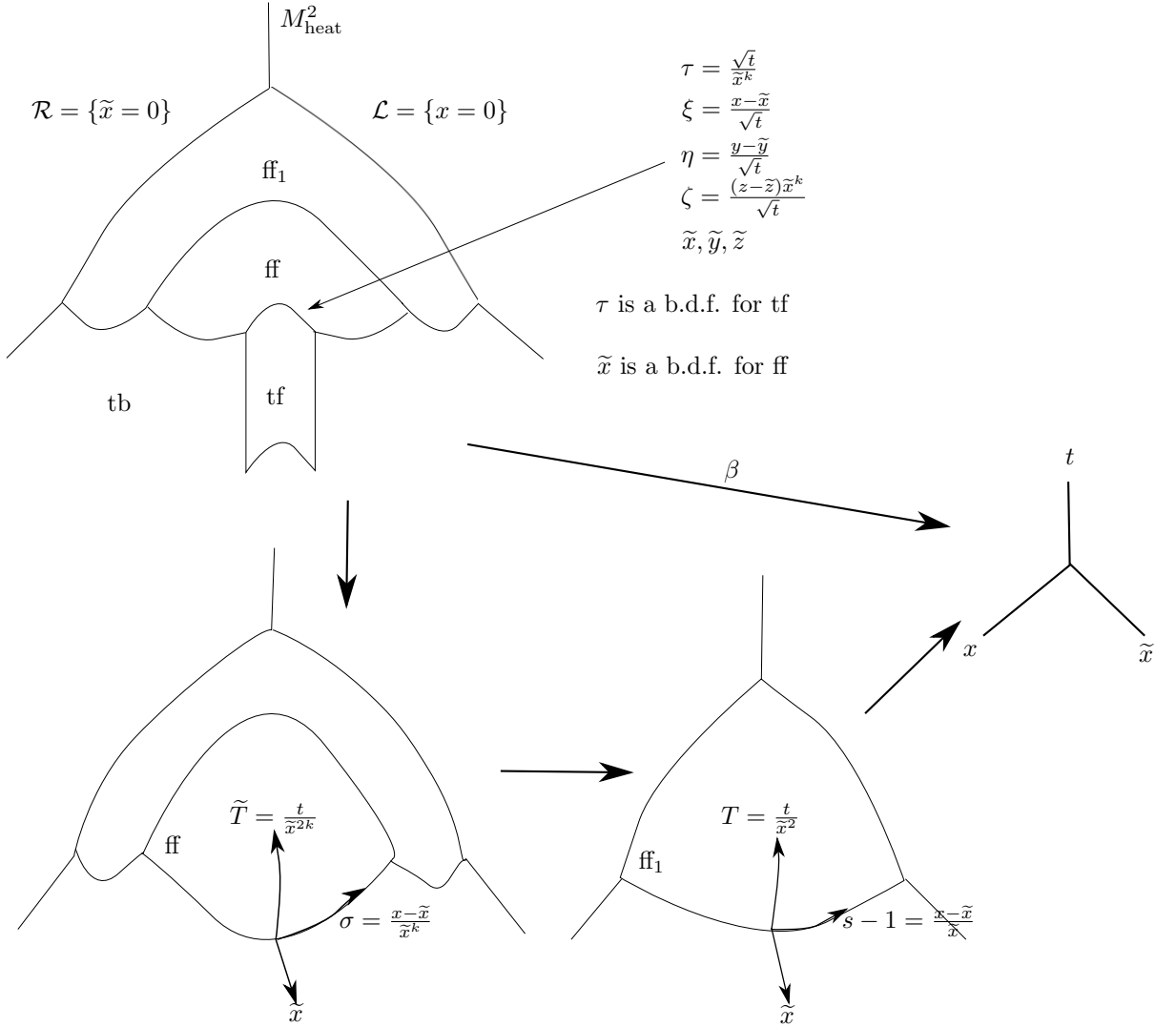
\caption{The heat double space (upper left) and the various
  intermediate blow ups together with their blow down maps.}
\label{fig:heat-double-space}
\end{center}
\end{figure}

\medskip

In summary, we have constructed a manifold with corners $\Mheat$,
depicted in Figure \ref{fig:heat-double-space}, with
a blowdown map $\beta$ as in \eqref{eq:singularmodelblowdown}, such
that $\Mheat$ has six total faces, three of them being the lifts of
the standard faces
\begin{equation}
  \label{eq:names-of-faces}
  \lf :=
  \cl(\beta^{-1}(\{ x = 0 \}^{\circ})), \ \rf := \cl(\beta^{-1}(\{ \wt{x} = 0 \}^{\circ})),  \   \tb := \cl(\beta^{-1}(\{
  t = 0 \}^{\circ})),
\end{equation}
and then the three faces $\fff, \ff,$ and $\tf$ constructed (in that
order) by radial blow up as described above.  Denoting the set of the six boundary
hypersurfaces by $\mathcal{M}(\Mheat) = \{ \lf, \rf, \tb, \fff, \ff, \tf
\}$, and given $\bullet \in \mathcal{M}(\Mheat)$, below we will let
$\rho_\bullet$ denote a boundary defining function for $\bullet$,
 so $\rho_\bullet \in C^\infty(\Mheat)$ satisfies that $\{\rho_\bullet = 0\} = \bullet$ and $d\rho_\bullet$ is
non-vanishing on $\bullet$ and $\rho_\bullet \ge 0$.  We can take
$\rho_{\ff} = \ov{\rho}$ as in \eqref{eq:polarsecondnmodel} and
$\rho_{\fff} = \rho / \overline{\rho}$.  Also note that 
\begin{equation}
  \label{eq:a-bdf}
  x = \rho_{\lf} \rho = \rho_{\lf} \rho_{\fff} \rho_{\ff}.
\end{equation}

In Theorem
\ref{thm:heatkernel} we will show that the heat kernel lifts to be
polyhomogeneous on $\Mheat$.

\subsection{Model operators}

The blown up space $\Mheat$ is useful in the construction of a
parametrix for the heat equation in part because the operator $\p_t +
\Delta$ (more specifically $t(\p_t + \Delta)$) behaves nicely at the
three introduced boundary hypersurfaces $\ff, \fff,$ and $\tf$; in
particular, the first steps in the parametrix construction involve
finding the right asymptotic behavior for the heat kernel so that the
heat equation \eqref{eq:fundamental-solution} is satisfied \emph{at
  least to leading order at} $\ff, \fff, $ and $\tf$.

Thus, we consider the operator $\Delta$ acting on the left spacial factor of $M
\times M \times [0, \infty)_t$, and the pullback $\beta^{*}(t (\p_{t} + \Delta))$
to $\Mheat$, and show that this restricts to an operator at $\tf$.  To be precise, fix a point $p \in M$
and consider the fiber $\tf_{p} = \pi^{-1}(p)$ where $\pi \colon \tf
\lra \diag_{M} = M$ is the projection onto the diagonal (or more
concretely it is $\beta \rvert_{\tf}$).  In the interior of $\tf$, i.e.\ away
from the intersection with $\ff$, this is standard \cite{tapsit}, so we
concern ourselves only with an open neighborhood of the intersection
of $\tf$ with $\ff$.
Indeed, working locally over the base in both spacial factors, consider a
subset of $\tf$ of the form $\{(\wt{x}, \wt{y}, \wt{z}, \xi, \eta,
\zeta, \tau) : (\wt{x}, \wt{y}, \wt{z}) \in \mathcal{O}\}$.
Now note
\begin{equation}
  \label{eq:tflaplace}
  \sqrt{t} \p_{x} = \p_{\xi},\quad \sqrt{t} \p_{y} = \p_{\eta},\quad
  \frac{\sqrt{t}}{x^{k}} \p_{z} = \p_{\zeta} + O(\tau),
\end{equation}
and
\begin{equation}
  \label{eq:tftdt}
  t\p_{t} = \frac 12 \lp \tau \p_{\tau} - R \rp,
\end{equation}
where $R = \xi \p_{\xi} + \eta \cdot \p_{\eta} + \zeta \p_{\zeta}$ is
the radial vector field on the fiber.

Letting $\pi_{L}, \pi_{R} \colon M \times M \times [0, \infty)_{t} \lra M$ denote
the projections onto the left and right $M$ factors, and $\End \lra M \times M$
the endomorphism bundle, whose fiber at $(p, q) \in M^\circ \times M^\circ$
is $\End(\Lambda_{q}^{*}; \Lambda^{*}_{p})$, for $ t > 0$, the heat
kernel \textit{restricted to the interior} will be a smooth section of this
bundle. To study the heat kernel at the boundary we use the incomplete
cusp edge forms and the corresponding endomorphism bundle $\End(\Lice)$ back to $M \times M \times [0, \infty)_{t}$
and then to $\Mheat$ via the blowdown $\beta$.  As usual, restricting
to the spacial diagonal gives the `little endomorphism' bundle 
$$
\End(\Lice) \rvert_{\diag(M)} \simeq \eend(\Lice)
$$
where $\eend(\Lice) \lra M$ is the endomorphism bundle
of the exterior algebra of $M$.   The restriction to the time face,
$\beta^{*}\End \rvert_{\tf}$, is isomorphic to the pullback of
$\eend(\Lambda^{*}_{p})$ to the tangent bundle of $M$ via the
projection map.

Writing $w = (x, y, z), \wt{w} = (\wt{x}, \wt{y}, \wt{z})$, sections of $\beta^{*}\End$ near the fiber
of $\tf$ over $p$ can be written
\begin{equation}
  \label{eq:10}
  \alpha = \sum_{I, J} a_{IJ} dw_{I} \otimes \p\wt{w}_{J},
\end{equation}
where $I, J$ run over all multi-indices and $\p \wt{w}_{J}$ is dual to
$dw_{J}$, and here $a_{IJ} = a_{IJ}(w, \wt{w}, t)$.
We claim that, writing sections
of $\beta^{*}\End$ near $\tf$ as sections of $\beta^{*} \End
\rvert_{\tf} \simeq \Lambda^{*}(M) \otimes \Lambda(M)$, 
\begin{equation}
  \label{eq:tffull}
  t(\p_{t} + \Delta) = (\frac 12 \lp \tau\p_{\tau} - R \rp +
  \sigma(\Delta))\otimes \id + O(\tau),
\end{equation}
where $\sigma(\Delta)$ is a constant coefficient differential
operator in the coordinates $\Mtan = (\xi, \eta, \zeta)$ depending on the
metric $g$ at $p = (\wt{x},\wt{y}, \wt{z})$, namely
\begin{equation}
  \label{eq:asdf}
  \sigma(\Delta) = (d_{\Mtan} + \star_{g(p)}^{-1} d_{\Mtan}  \star_{g(p)})^{2},
\end{equation}
acting on differential forms on the vector space $\Lice_{p}(M)$
with metric $g(p)$.   Indeed, let
$w$ be geodesic normal coordinates.  In the interior of $\tf$ away
from $\ff$ we have coordinates $\Mtan = (w - \wt{w})/\sqrt{t}, \wt{w},
\sqrt{t}$.  
Then $t(\p_{t} + \Delta) \alpha = (t(\p_{t} + \Delta)  \sum_{I,
  J} a_{IJ} dw_{I}) \otimes \p\wt{w}_{J}$ and moreover
\begin{equation}
  \label{eq:interiortmodel}
  \begin{split}
    \star dw_{I} \otimes \p\wt{w}_{J}  &= (\star dw_{I}) \otimes
    \p\wt{w}_{J}  \\
    &=  \pm dw_{I^{c}} \otimes \p \wt{w}_{J} + O(w -
    \wt{w}) \\
    &= \pm (d\wt{w}_{I^{c}} + \sqrt{t}d\Mtan_{I^{c}}) \otimes
    \p \wt{w}_{J} + O(\sqrt{t} \Mtan) \\
    &= (\star_{g(p)} d\wt{w}_{I}) \otimes \p \wt{w}_{J} + O(\wt{t}),
  \end{split}
\end{equation}
Similarly, letting the exterior derivative act on the left gives $d(a
dw_{I} \otimes \p\wt{w}_{J}) = (\p_{\Mtan_{i}} a d\wt{w}_{i} \wedge
\wt{w}_{I} ) \otimes \p\wt{w}_{J}.$

To motivate our construction of the heat kernel further, in a
neighborhood of $\tf$ let $\gamma$ be a section of $\End$
with the property that $\gamma\rvert \diag_{M} = \id$ on the form
bundles, and consider the section of $\beta^{*}\End$ on $\Mheat$ of the form
\begin{equation}
  \label{eq:ansatzbegins}
  K(p, q, t) = \frac{1}{(2 \pi t)^{n/2}} e^{- G(p, q)/2t} \gamma,
\end{equation}
such that
$G(p, q)$ satisfies that $\beta^{*}(G(p,q)/t)
\rvert_{\tf} = \norm[g]{\Mtan}^{2},$ that is, that $G(p, q)/t
$ restricts to the metric function on $\tf$.  Such a form $\gamma$ and
function $G$ can be
constructed but we neither prove nor use this; we merely use it
as motivation. It is straightforward to check that for any smooth
compactly supported form $\alpha$
\begin{equation}
  \label{eq:9}
\lim_{t \to 0}  \int_{M} K(p, q, t) \alpha(q) \dVol_{q} =
(4\pi)^{-n/2} \int_{M} e^{-\norm[g(p)]{\Mtan}^{2}/4} \alpha(p) \sqrt{g(p)}
|d\Mtan| = \alpha(p),
\end{equation}
and in fact the convergence takes place in $L^{2}$.  
(In fact, such an endomorphism $\gamma$ can be constructed easily by
taking the identity map on $\Lice$ over $M$, pulling this back via
$\beta$ to $\beta^{*} \End \rvert_{\tf}$ and extending off smoothly
in a neighborhood.  On each exterior algebra $\Lambda_{p}^{*}M$, the identity
can be expressed in terms of a basis $e_{i}$ with dual basis $e^{*}_{i}$ as $\sum_{I} e_{I}
\otimes e^{*}_{I}$.  In a neighborhood of $\tf \cap \ff$ we can take
the basis $e_{i}$ to be $dx$, $dy_{i}$, $x^{k}dz_{\alpha}$, i.e. we
can take the $e_{i}$ to be a basis of forms for $\Lice$ all the way
down to $x = 0$.)

Working in the coordinates \eqref{eq:coordsfftf}, since $t^{-n/2} = \tau^{-n}\wt{x}^{-nk}$, the Taylor expansion of the
heat kernel at $\tf$ should take the form
\begin{equation}
  \label{eq:tf-expansion}
  \frac{1}{(4\pi)^{n/2} \wt{x}^{kn} } \tau^{-n} \sum_{j = 0}^{\infty}
  \tau^{j} b_{j},
\end{equation}
where the $b_{j} = b_{j}(\wt{x}, \wt{y}, \wt{z}, \xi, \eta, \zeta)$
are sections of $\beta^{*} \End$, which we again write in a
neighborhood of $\tf \cap \ff$ as sections of $\Endbd$ pulled back to
the fibers of $\tf$.  Writing each $b_{j}$ as a finite sum of terms
of the form
\begin{equation}
  \label{eq:14}
  \alpha \otimes g^{-1}\beta,
\end{equation}
where $\alpha$ and $\beta$ are sections of $\Lice$ and $g^{-1}$
indicates taking the dual vector field, we see that
by \eqref{eq:tffull} we have, 
\begin{equation}
  \label{eq:12}
  \lp (\frac n2 - \frac 12 R  +
  \sigma(\Delta))  \otimes \id  \rp b_0 =   \lp (\frac n2 - \frac 12 R  +
 \lp 
 \begin{array}{cc}
   \Delta_{\Mtan} & 0 \\
   0 & \Delta_{\Mtan}
 \end{array}
\rp )
 \otimes \id  \rp b_0.
\end{equation}
The only solution to this equation which gives the identity operator
at $t = 0$ is 
\begin{equation}
  \label{eq:15}
  b_{0} = e^{- ||\Mtan||^{2} / 4} \times \id.
\end{equation}

 The procedure of solving for the
remaining $b_{j}$ is standard \cite[Chapter 7]{tapsit}; letting the
Laplacian act on this expansion we show that on each term $a_j$ it
acts fiberwise like a constant coefficient, second order elliptic
differential operator plus the radial vector field plus a constant corresponding to the order of
the term in the expansion. We have the following
\begin{lemma}\label{thm:tf-expansion}
  There exist sections $b_j$ of $\mathcal{A}_{\phg}(\End \rvert_{\tf})$ satisfying 
$$
b_j =
  e^{- ||\Mtan||^{2} / 4} \wt{b}_{j}(\wt{x}, \wt{y}, \wt{z},\xi, \eta, \zeta)
$$
  where $\wt{b}_j$ is a polynomial in $\xi, \eta, \zeta$ and a  polyhomogeneous
  section of $\End$ over $\tf$, such that for any
  distribution  $H'$ in $\mathcal{A}^{\phg}(\End)$ with asymptotic
  expansion near $\tf$ given by \eqref{eq:tf-expansion} we have
$$
t(\p_t + \Delta) H' = O(\tau^\infty),
$$
i.e.\  $t(\p_t + \Delta) H'$ vanishes to infinite order at the blown up
$t = 0$ diagonal, and, moreover, the asymptotic sum of the $b_j$
exists and yields such an $H'$.
\end{lemma}
The existence of a distribution $H'$ as in Lemma \ref{thm:tf-expansion} is only a first step
in constructing a parametrix for the heat kernel.  We will discuss the
rest of the process in \textsection\ref{sec:parametrix-construction}.

A useful double check of the order of blow up of the heat kernel at $\ff$ is
        the following.  Near $\ff \cap \tf$ we have
		\begin{align*} 
			\delta(x - \wt{x})   \delta(z - \wt{z})
                        \delta(y - \wt{y}) &= \delta(\xi \tau
                        \wt{x}^{k}) \delta(\eta \tau \wt{x}^{k})
                        \delta(\zeta \tau)
                        \\
			&=  \frac{1}{\tau^{n} \wt{x}^{(n - f)k}}\delta(\xi ) \delta(\eta)\delta(\zeta).
		\end{align*}
Since $ \operatorname{Id} = \lim_{t \searrow 0}H \dVol_g \sim \lim_{t \searrow 0} H \wt{x}^{kf}
d\wt{x}d\wt{y}d\wt{z}$, we confirm that $H$ should have order $-nk$ at
$\ff$.  In fact, we can deduce more; considering $\wt{x}^{kn} H
\rvert_{\ff}$, on the interior of $\ff$ we can use coordinates in
\eqref{eq:projectivefff}, we get that 
\begin{equation}
  \label{eq:5}
 \delta(x - \wt{x})  \delta(y - \wt{y})  \delta(z - \wt{z}) =
 \wt{x}^{-(n - f)k} \delta(\sigma)  \delta(\wt{\eta})  \delta(z - \wt{z}),
\end{equation}
which means that, on the face $\ff$, we expect that the restriction
$\wt{x}^{nk} H \rvert_{\ff}$ will be given by $\delta(\sigma)
\delta(\wt{\eta})  \delta(z - \wt{z})$ \textit{at least as the time
  variable} $\wt{T} = t / \wt{x}^{k}$ \textit{goes to zero}, as that
is the region in which the action of $H$ is definitively approximated
by the identity.  On the other hand, $\wt{x}$ commutes with the heat
operator $\p_{t} + \Delta$.  As we will see in
\eqref{eq:modelheatsecond}, $t(\p_{t} + \Delta)$ restricts to an
operator on $\ff$ and defines a fiber-wise heat type operator on $\ff$, so we expect to have
\begin{equation}
  \label{eq:ff-reason}
  t(\p_{t} + \Delta) \rvert_{\ff} (\wt{x}^{nk} H) \rvert_{\ff} =  0.
\end{equation}
This, together with \eqref{eq:5}, implies that an ansatz for the heat
kernel should include that
\textbf{on each fiber of $\ff$,  $\wt{x}^{nk} H \rvert_{\ff} $ is the
  fundamental solution to the induced heat equation on the fiber},
more precisely, it is the solution which equals $\delta_{\sigma =
  0}\delta_{\wt{\eta} = 0} \operatorname{Id}_{Z}$ at time equals zero.  The induced
heat equations are translation invariant in $\sigma$ and $\wt{\eta}$,
thus induced by convolution operators, and the heat kernels we speak
of are the convolution kernels in $\sigma$ and $\wt{\eta} $.

As for the blow up at $\fff$, as we will see below, the
operator acts as a modified heat operator in $\p_{x}$ and
$Y$ on the bundle of
fiber harmonic forms, so in the coordinates in
\eqref{eq:frontfront} we will have
		\begin{align} \label{eq:fff-reason}
			\delta(x - \wt{x})   
                        \delta(y - \wt{y}) \delta(z - \wt{z}) 
                           =  \frac{1}{\wt{x}^{1 + \dimY}}\delta(s - 1) \delta(z - \wt{z})
                        \delta(\eta).
		\end{align}
In this case, $t(\p_{t} + \Delta)$ only admits a restriction to
$\fff$ on the fiber-harmonic forms $\Harmm$, on which $\delta(z -
\wt{z})$ becomes projections $\Pi_{Z, y}$ onto the kernel of
$\Delta^{Z, y}$.  Thus we expect that $\wt{x}^{1 + \dimY + kf} H
\rvert_{\fff}$ on fiber harmonic forms is given by the convolution kernel
for the heat kernel in $\eta$, times the dilation invariant kernel for
the heat kernel in $s$ with limit $\delta_{s = 1}$ at time $0$.

\medskip

We now compute the asymptotic behavior of $t(\p_t + \Delta)$ at the
faces $\ff$ and $\fff$.  First we will work at $\ff$.	
\begin{proposition}[The model problem on $\ff$]
The operator $ \Nff = t(\p_t + \Delta^g) \rvert_{\ff}$ acts fiberwise on $\ff$, and is 
expressed in the coordinates in \eqref{eq:projectivefff} by
\begin{equation}\label{eq:modelheatsecond}
  \begin{split}
    \Nff &= \wt{T} \lp \p_{\wt{T}} + \lp
    \begin{array}{cc}
      - \p_{\sigma}^{2} + \Delta_{\eta} +
      \Delta^{Z,y} & 0 \\
      0 &  - \p_{\sigma}^{2} + \Delta_{\eta} +
      \Delta^{Z,y}
    \end{array}
    \rp \rp
  \end{split}
\end{equation}
on the fiber above $y \in Y$.  Here $\Delta_{\eta}$ is the constant
coefficient Hodge-Laplacian on the tangent space $T_{y}Y$ with
translation invariant metric $h(y)$, and $\Delta^{Z, y}$ is the
Hodge-Laplacian on $(Z, \km_y)$
\end{proposition}

The situation is more delicate at $\fff$.    As we will see in
Section \ref{sec:parametrix-construction}, near $\fff$, it will suffice to consider
$t(\p_t + \Delta)$ restricted to approximately fiber harmonic forms, (see Section \ref{sec:fibharmforms}).  Thus let
$\gamma  \in x^s
\mathcal{H}_{\phg}(\Lice)$ with the space on the right defined in \eqref{eq:phg-appro-fib-harm}, and thus by
assumption $\delta^Z  \gamma, d^Z \gamma$ are both $O(x^{s + k})$.
From \eqref{eq:Hodge-Laplacian-decomp} it follows that for such fiber
harmonic forms,
\begin{equation}
  \label{eq:Hodge-Lapl-on-fiber-harmonic}
  \begin{split}
    \Delta \gamma &= \wt{\Delta}_0 \gamma + x^{-2k} \Delta_Z \gamma +
    x^{-k} (\delta_Z Q_1 + d_Z Q_2) \gamma
    + O(x^{s - 1}) \\
  \end{split}
\end{equation}
where $\wt{\Delta}_0$ acts on forms decomposed as in
\eqref{eq:ice-decomp} as
\begin{align*}
  \wt{\Delta}_0 &=  -
  \p_{x}^{2} - \frac{kf}{x} \p_{x} + \Delta_{Y} 
  + \lp
  \begin{array}{cc}
    k\NN(1 - k(f - \NN)) x^{-2} & - 2k x^{-k - 1} d^{Z} \\
    - 2 k x^{-k - 1} \delta^{Z} & k(f - \NN)(1 - k \NN) x^{-2}
  \end{array} \rp.
\end{align*}
By fiber harmonicity, $x^{-2k} \Delta_Z \gamma = O(x^{s-k})$. Thus the two terms $x^{-2k}\Delta_Z $ and
$x^{-k} \delta_Z Q_1$ act on fiber harmonic forms as operators of
order $x^{-k}$, and thus in the heat operator $t(\p_t + \Delta)$ there are term behaving
like $t x^{-k}$ (on fiber harmonic forms) and
\textit{$t/x^{-k}$ is not a bounded function at $\fff$}!  On the other
hand, if we project back to the fiber harmonic forms, then by
\eqref{eq:harmonic-projection}-\eqref{eq:harmonic-projection-2} we kill
these terms; concretely, with $\Pi_{\mathcal{H}}$ the fiber harmonic
projector in \eqref{eq:real-fiber-harmonic-projection}, we have
\begin{equation}
  \label{eq:Hodge-Lapl-on-fiber-harmonic-2}
    \Pi_{\mathcal{H}} \Delta  \Pi_{\mathcal{H}} = \wt{\Delta}_0 + x^{-1} \wt{E}'
\end{equation}
where $\wt{E}' \in \Diff^2_{\bl, \phg}$ (see \eqref{eq:errors-in-operator2}),
and thus does not decrease the order of vanishing of polyhomogeneous distributions.  Defining
\begin{equation}
  \label{eq:regsingheat}
 P_{A, B} :=  - \p_{s}^{2} - \frac{A}{s}\p_{s} + \frac{B}{s^{2}}.
\end{equation}
and
  \begin{equation}
    \label{eq:alpha-beta-gamma}
    \alpha(\NN) := kf, \quad \beta(\NN) := k\NN(1 - k(f - \NN)), \quad \gamma(\NN) = k(f - \NN)(1 - k
    \NN),
  \end{equation}
we have the following.
\begin{proposition}[Heat operator on fiber harmonic forms at $\fff$]\label{eq:model-first-heat}
Restricted to the fiber harmonic forms $\mathcal{H}$ as defined through
\eqref{eq:fiber-harmonic-space},
\begin{equation}\label{eq:normal-operator-ff}
\Nfff :=  \Pi_{\mathcal{H}} t (\p_{t} + \Delta)
\Pi_{\mathcal{H}}  \rvert_{\fff}
\end{equation}
restricts to
the face $\fff$ in the coordinates \eqref{eq:frontfront} as
\begin{equation}
  \label{eq:modelheatfirst}
  \begin{split}
  &  \Nfff =  t' \lp \p_{t'} +
    \lp
    \begin{array}{cc}
     P_{\alpha({\NN}), \beta({\NN})} + \Delta_{\eta} & 0 \\
      0 & P_{\alpha({\NN}), \gamma({\NN})} + \Delta_{\eta} 
    \end{array} \rp \rp.
  \end{split}
\end{equation}
\end{proposition}

\begin{remark}\label{thm:explanation}
  Analysis of the fiber harmonic forms is necessary in particular
  because the structure of the operator $\Delta^g$ is such that, off
  of the fiber harmonic forms, the leading order term is $x^{-2k}
  \Delta^{Z, y}$, while restricted to the fiber harmonic foms the
  leading order term drops in order.  Indeed, if it weren't for the presence
  of the term $x^{-k} \wt{P}$ in \eqref{eq:Hodge-Laplacian-decomp},
  which presents complications in the analysis, on
  fiber harmonic forms $\Delta^g$ would be given by to leading order
  by $\wt{\Delta}_0$.  
  Indeed, the need for the two different regimes represented by the
  boundary hypersurfaces $\ff$ and $\fff$ is exactly this change in
  asymptotic order of the operator on and off the fiber harmonic
  forms.  Correspondingly, we will see below in the proof of Lemma
  \ref{thm:pre-ansatz}   that the operator $t(\p_t
  + \Delta)$ restricted to $\ff$ has a fundamental solution which
  vanishes at $\fff$ to infinite order \emph{off the fiber harmonic forms}.
\end{remark}

The heat equation
for the regular
singular ODEs in \eqref{eq:regsingheat} has been studied in detail.  
To such an operator there corresponds a pair of indicial roots given
by the order of vanishing of homogeneous solutions, specifically
$P_{A,B} (s^{\ell}) = 0$ if and only if
\begin{equation}
  \label{eq:indicialroots}
  \ell = \frac{-(A - 1) \pm \sqrt{(A - 1)^{2} + 4B}}{2}.
\end{equation}
The numbers $\ell$ give important information about the operator
$P_{A, B}$, in particular they give the order of vanishing of the
Green's function at $s = 0$.  The operators that will arise in our work are those in the matrices
in \eqref{eq:modelheatfirst}.  We define the indicial set
\begin{equation}
  \label{eq:indicial-roots}
  \begin{split}
    \Lambda &= \bigcup_{{\NN} = 1}^{f} \set{ \frac{-(\alpha - 1) \pm
        \sqrt{(\alpha- 1)^{2} + 4\beta}}{2}, \frac{-(\alpha - 1) \pm
        \sqrt{(\alpha- 1)^{2} + 4\gamma}}{2}} \\
& = \bigcup_{{\NN} = 1}^{f}  \set{ -(kf - 1)/2 \pm
        |k({\NN} - f/2) + 1/2|, -(kf - 1)/2 \pm
        |k({\NN} - f/2) - 1/2| }
  \end{split}
\end{equation}
Letting
\begin{equation}\label{eq:nu}
\nu^{2} = B + (\frac{A - 1}{2})^{2} > 0
\end{equation}
where $\nu > 0$, from \cite[Vol.~2, Eqn.~8.45]{taylor:vol2} there is a
fundamental solution $H_{A,B}(s, \wt{s}, t)$  
\begin{equation}
  \label{eq:modelheatpotentialfirst}
  (\p_{t} + P_{A,B}) H_{A,B}(s, \wt{s},t) = 0,\mbox{ and }
 H \to \id \mbox{ as } t \to 0, \mbox{ on } L^{2}(s^{A}ds).
\end{equation}
Indeed, one has the explicit formula
\begin{equation}
  \label{eq:modelheatpotentialfirst2}
  H_{A,B}(s, \wt{s}, t) = (s \wt{s})^{-(A - 1)/2}\frac{1}{2t}
  e^{-(s^{2} + \wt{s}^{2})/4t} I_{\nu}\lp \frac{s \wt{s}}{2t} \rp
\end{equation}
where $I_{\nu}$ is the modified Bessel function of
order $\nu$ of the first kind \cite[Chap.\ 9]{AS1964}.

As discussed below \eqref{eq:fff-reason}, at the face $\fff$ we expect
the heat
kernel to be of order $\wt{x}^{-1 - \dimY - kf}$.  Thus we expect to have
\begin{equation}
  \label{eq:18}
  0 = t (\p_{t} + \Delta) H = \frac{1}{\wt{x}^{1 + \dimY + kf}}
(t (\p_{t} + \Delta))  (\wt{x}^{1 + \dimY + kf} H),
\end{equation}
and since $\Pi_{\mathcal{H}} t (\p_{t} + \Delta) \Pi_{\mathcal{H}} $
defines a differential operator on section of $\mathcal{H} \otimes \ov{\mathcal{H}}^*$ restricted
to $\fff$, 
we include in our ansatz for the   
fundamental solution \eqref{eq:fundamental-solution}, and indeed prove
in Theorem \ref{thm:heatkernel} below, that there is a fundamental
solution $H$ satisfying that $\wt{x}^{1 + \dimY + kf}
H$ has a smooth restriction to $\fff$, and writing
\begin{equation}
  \label{eq:27}
  \NHfff  :=  (\wt{x}^{1 + \dimY + kf}H)\rvert_{\fff},  \mbox{ we have
    } N_{\fff}(t (\p_{t}
  + \Delta)) \NHfff = 0.
\end{equation}
Furthermore, again as discussed below \eqref{eq:fff-reason}, it is sensible to include in the ansatz for $H$ that
$\NHfff$ is the fundamental solution for the model operator $N_{\fff}(t (\p_{t}
  + \Delta))$, meaning specifically that $\NHfff$ is a section of
  the restriction of the sub-bundle $\End(\Harmm)$ to $\fff$ and
  is given using the fundamental solutions to the model heat equations
$H_{A,B}$ from
\eqref{eq:modelheatpotentialfirst}-\eqref{eq:modelheatpotentialfirst2}.
Specifically, we will have as an ansatz that $ \NHfff = \kappa_{\fff}
$, where
  \begin{equation}
    \label{eq:fffimposed}
\kappa_{\fff,y} := \lp
    \begin{array}{cc}
     H_{\alpha, \beta}(s, 1, t') & 0 \\
      0 & H_{\alpha, \gamma}(s, 1, t') 
    \end{array} \rp    (4\pi t')^{-b/2}e^{- |\eta|_{y}^{2}/4t'},
  \end{equation}
where $\alpha, \beta, \gamma$ are as in \eqref{eq:alpha-beta-gamma}, and in
particular continue to be \textit{operators} depending on the fiber
form degree $\NN$.
The distribution $\NHfff$ is polyhomogeneous on $\fff$, and the
leading order behavior at $s = 0$ satisfies that for $0< c \le t' \le
C < \infty$, for some smooth $a(t'), b(t')$,
\begin{equation}
  \label{eq:model-heat}
  \begin{split}
    H_{\alpha, \beta}(s, 1, t') \sim s^{-(kf - 1)/2} a(t')
      s^{\nu(\alpha, \beta)},
\quad H_{\alpha, \gamma}(s, 1, t') \sim s^{-(kf - 1)/2} b(t')
      s^{\nu(\alpha, \gamma)}
  \end{split}
\end{equation}
with $\nu$ as in \eqref{eq:nu}
\begin{equation}
  \begin{split}
    \label{eq:nu-in-terms-of-alphabeta}
    \nu(\alpha, \beta) = \left\{
      \begin{array}{cc}
        k(f/2 - {\NN})  - 1/2 & \mbox{ if }  {\NN} < f/2, \\
        k({\NN} - f/2)  + 1/2 & \mbox{ if }  {\NN} \ge f/2,
      \end{array}
    \right.  \\
 \quad \nu(\alpha, \gamma) = \left\{
      \begin{array}{cc}
        k(f/2 - {\NN})  + 1/2 & \mbox{ if }  {\NN} \le f/2, \\
        k({\NN} - f/2)  - 1/2 & \mbox{ if }  {\NN} > f/2,
      \end{array}
    \right.
  \end{split}
\end{equation}
and thus by \eqref{eq:model-heat} on $\fff$ in the region $0< c \le t' \le
C < \infty$, 
\begin{equation}
  \label{eq:first-model-asymptotics}
  \kappa_{\fff} = O(s^{\ov{\nu}}) \mbox{ where } \ov{\nu}(\NN) = 
\left\{ 
  \begin{array}{cc}
 -k{\NN} & \mbox{ if } {\NN} < f/2, \\
 - k{\NN}+ 1 & \mbox{ if } {\NN} =  f/2,  \\
 {- k(f - {\NN})} & \mbox{ if } {\NN} >  f/2. 
  \end{array}
\right. 
\end{equation}
In words, each $P_{\alpha, \beta}$ has two indicial roots, the 
order of $H_{\alpha,
  \beta}$ for fixed $\wt{s}, t > 0$ is the larger of these two, and $p$ is the \textit{smaller}
of the leading orders of $H_{\alpha, \beta}$ and $H_{\alpha, \gamma}$.

The behavior of the heat kernel at $\fff$ also shows what to expect at
the left face, the lift of $x = 0$.  There we should just have the
projection onto the fiber harmonic forms times the leading order
behavior of the $H_{\alpha, \beta}$ and $H_{\alpha, \gamma}$, acting
appropriately on $\Lice$, times the lifted heat kernel of the base
$Y$.  Indeed, we expect
\begin{equation}
  \label{eq:heat-kernel-left-face}
\Pi_\mathcal{H}  H \Pi_\mathcal{H} \simeq \kappa := \lp
    \begin{array}{cc}
     H_{\alpha(\NN), \beta(\NN)}(x, \wt{x}, t) & 0 \\
      0 & H_{\alpha(\NN), \gamma(\NN)}(x, \wt{x}, t) 
    \end{array} \rp    H_Y 
\end{equation}
where $H_Y$ is the heat kernel on $(Y, h)$ lifted to the tubular
neighborhood $\mathcal{U}$ in \eqref{eq:tubular} via the projection
and $\kappa$ acts on sections of the bundle of fiber harmonic forms
$\mathcal{H}$ with its grading by fiber form degree $\NN$.  (See
Section \ref{sec:fibharmforms}.)  In fact, with $\ov{\nu} = \ov{\nu}(\NN)$
the fiber degree dependent weight in \eqref{eq:first-model-asymptotics},
\begin{equation}
  \label{eq:kappa-section-of}
\kappa \in x^{\wt{\nu}(\NN)} C^\infty(M \times M ; \oplus_{\NN= 0}^f\mathcal{H}^{\NN} \otimes (\wt{\mathcal{H}}^{\NN})^*)
\end{equation}
where $\mathcal{H}$ and $\wt{\mathcal{H}}$, respectively, the pullbacks
of the fiber harmonic form bundle (defined on a neighborhood
$\mathcal{U}$ of the boundary) via the left and right projections of
$M \times M$.

As discussed below \eqref{eq:ff-reason}, on the face $\ff$, we expect that the heat kernel will have leading
asymptotic $\wt{x}^{-nk}$, so we expect and prove that
\begin{equation}
  \label{eq:27-ff}
  \NHff := (\wt{x}^{nk}H)\rvert_{\ff},  \implies N_{\ff}(t (\p_{t}
  + \Delta)) \NHff = 0.
\end{equation}
Again, we will set $\NHff$ equal to a fundamental solution to the heat
equation, namely, using the decomposition in
\eqref{eq:modelheatfirst}, we expect to have $\NHff = \kappa_{\ff}$ where
\begin{equation}
  \label{eq:ffimposed}
  \kappa_{\ff, y}(\sigma, \eta, z, z', \wt{T})
  :=  \operatorname{Id}_{2 \times 2} (4 \pi \wt{T})^{-(b +1)/2}e^{-(\sigma^{2} +
        |\eta|_{h_{1}}^{2})/4\wt{T}} H_{Z,y},
\end{equation}
where $H_{Z,y}  = H_{Z,y}(z, z', \wt{T})$ is the heat kernel for $\Delta^{Z, y}$.
\medskip\\
Before stating the full structure theorem for the heat kernel let us briefly recall the notion of an {\bf{index set}}, which by definition is a set of exponents $\mathcal E(\bullet)=\{(\gamma,p)\}\subset \mathbb C\times\mathbb N$ associated with each face $\bullet\in\{\lf,\rf,\tb,\ff_1,\ff,\tf\}$ such that
\begin{itemize}
\item[(i)] each half-plane $\re \gamma<C$ contains only finitely many $\gamma$;
\item[(ii)] for each $\gamma$, there is a number $P(\gamma)\in\mathbb
  N_0$ such that $(\gamma,p)\in \mathcal E(\bullet)$ for every $0\leq
  p\leq P(\gamma)$ and $(\gamma,p)\notin \mathcal E(\bullet)$ if
  $p>P(\gamma)$; 
\item[(iii)]
If $(\gamma,p)\in\mathcal E(\bullet)$, then $(\gamma+j,p)\in \mathcal E(\bullet)$ for all $j\in\mathbb N$.
\end{itemize}
We give a full rigorous definition of polyhomogeneity in Section
\ref{sec:mwc}, but roughly speaking, we call a differential form
$\alpha$ {\bf{polyhomogeneous with index
    family}} $\mathcal E=\{\mathcal E(\bullet)\mid\bullet\in
\{\lf,\rf,\tb,\ff_1,\ff,\tf\}\}$ if it has an expansion at each
boundary hypersurface $\bullet$ with exponents determined by the
corresponding index set $\mathcal E(\bullet)$ and coefficient
functions which are themselves polyhomogeneous (with exponents
determined by $\mathcal E$).  For example, smooth functions on
$\Mheat$ are polyhomogeneous with indicial set satisfying
$\mathcal{E}(\bullet) = \mathbb{Z} \times \{ 0 \}$ for all $\bullet$,
and if a polyhomogeneous function vanishes to infinite order at a
particular boundary hypersurface $\bullet$, then it is polyhomogeneous
with an index set $\mathcal{E}$ satisfying $\mathcal{E}(\bullet) =
\varnothing$.  We define
$$
\inf \mathcal{E}(\bullet) = \inf \{ \gamma \mid (\gamma, p) \in
\mathcal{E}(\bullet) \}.
$$

\begin{theorem}\label{thm:heatkernel}
  There exists a section $H$ of $\beta^{*} \End$ over $\Mheat$ which is
  polyhomogeneous, i.e.\ $H \in \mathcal{A}_{\phg}^{\mathcal{E}}(\Mheat;
  \beta^{*}\End)$ for some index set $\mathcal{E}$ and satisfying the
  following properties.
  \begin{enumerate}
  \item In the interior of $\Mheat$, $(\p_{t} + \Delta) H =
    0$.
  \item The operator $H_{t}$ defined initially on forms
    $\alpha \in C^{\infty}_{c}(M; \Omega^{*}(M))$ by
    \begin{equation}
      \label{eq:heatflow}
      H_{t} \alpha(w) = \int_{M} H(w, \wt{w}, t) \alpha(\wt{w}) \dVol_{\wt{w}}
    \end{equation}
is symmetric on $L^2(d\Vol_g)$, and for such $\alpha$
\begin{equation}
  \label{eq:39}
  H_{t}\alpha \to \alpha \mbox{ as } t \to 0
\end{equation}
in $L^2$.
  \item The index set $\mathcal{E}$ of $H$ satisfies that
    $\mathcal{E}(\tb) = \varnothing$, while
  \begin{equation}
    \label{eq:indexsets}
    \inf \mathcal{E}(\fff) \ge -1 - \dimY - kf , \quad
    \inf \mathcal{E}(\ff) \ge -kn, 
    \quad
    \mathcal{E}(\tf) = \mathbb{N} - \dim(M),
  \end{equation}
where $\mathcal{E}(\bullet) \ge c$ means that $\inf_{\zeta, p \in
  \mathcal{E}(\bullet)} \mbox{Re} \ \zeta \ge c$.
\item Moreover, at the faces $\ff$ and $\fff$, 
  \begin{equation}\label{eq:ansatzfaces}
(\wt{x}^{1+ \dimY + kf}H) \rvert_{\fff} = \kappa_{\fff}, \quad
  (\wt{x}^{kn}H) \rvert_{\ff} = \kappa_{\ff},
\end{equation}
with $\kappa_{\fff}$ and $\kappa_{\ff}$ the model heat kernels defined
in \eqref{eq:fffimposed} and \eqref{eq:ffimposed}.  
\item At the left 
face $\lf$, with $\kappa$ as in \eqref{eq:heat-kernel-left-face}
  \begin{equation}\label{eq:leftfacebehaviour}
H_{t} = \kappa(1 + O(\rho_{\lf}^k)).
\end{equation}
In particular, for $t > 0$ fixed,  $H_t$ is approximately
  fiber harmonic in the sense of \eqref{eq:fiber-harmonic-space}, and
$\rho_{\lf}$ is a boundary defining function for $\lf$ (see \eqref{eq:a-bdf}).  Thus 
\begin{equation}\label{eq:left-face-index}
\inf \mathcal{E}(\lf) \ge - \frac{kf}{2} + 1.
\end{equation}
The behavior of $H_{t}$ at the right face $\rf$ can be deduced from
symmetry.
Moreover for the behavior at the codimension $2$ face $\lf
\cap \rf$, the leading order behavior is the product of that at $\lf$
and $\rf$, i.e.  $H_t = O((\rho_{\lf} \rho_{\rf})^{-kf/2 + 1})$.
\end{enumerate}
\end{theorem}

The proof of Theorem \ref{thm:heatkernel} is at the end of \textsection\ref{sec:parametrix-construction}.

\subsection{Parametrix construction}\label{sec:parametrix-construction}

We will establish the following
\begin{proposition}\label{thm:ansatz}
  There exists a polyhomogeneous section
  \begin{equation}
K \in
  \mathcal{A}_{\phg}(\Mheat; \beta^{*}(\End))\label{eq:13}
\end{equation}
such that $K$ satisfies \eqref{eq:39}--\eqref{eq:leftfacebehaviour} above,
and such that $Q := t(\p_{t} + \Delta) K$, which is polyhomogeneous,
has index set $\mathcal{E}'$ satisfying
  \begin{equation}
    \label{eq:indexsetsprime}
    \mathcal{E}'(\fff) \ge -1 - \dimY - kf + 1, \quad
    \mathcal{E}'(\ff) \ge -kn + 1, 
    \quad
    \mathcal{E}'(\lf) = \mathcal{E}'(\tf) = \mathcal{E}'(\tb) =  \varnothing.
  \end{equation}
\end{proposition}

Our work above allows us to break the proof of the proposition
into two main steps; first we will prove the following lemma:
\begin{lemma}\label{thm:pre-ansatz}
  There exists a polyhomogeneous $K_1$ satisfying
  \eqref{eq:39}--\eqref{eq:leftfacebehaviour} above, together with 
  \eqref{eq:tf-expansion} for
the indicated $b_j$.  Furthermore, $K_1$ can be taken fiber harmonic
in a neighborhood of $\fff$.  In fact we can assume that $K_1 \equiv
\kappa$  in a neighborhood of $\lf$ and also in a neighborhood
of $\rf$, where $\kappa$ is as in \eqref{eq:leftfacebehaviour}.  
\end{lemma}

\begin{proof}[Proof of Proposition \ref{thm:ansatz} assuming Lemma \ref{thm:pre-ansatz}]
  Assuming that we have such a distribution $K_1$, we study $t(\p_t +
  \Delta) K_1$.  Automatically we have that $t(\p_t +  \Delta) K_1$ vanishes to infinite
  order at $\tf$ and $\tb$, as follows from Lemma
  \ref{thm:tf-expansion}.  Furthermore, $t(\p_t +  \Delta) K_1$ vanishes to order $-kn +
  1$ at $\ff$ by \eqref{eq:modelheatsecond} and the fact that the
  leading order term $\kappa_{\ff}$ there solves the model problem.  

At $\fff$ things are again more delicate.  Recall that $K_1 = O(\rho_{\ff}^{-1
  - \dimY - kf})$ at $\fff$, where $\rho_{\fff}$ is the boundary
defining function for $\fff$ in, e.g.\ $\rho_{\fff} = \rho /
\ov{\rho}$ with $\rho$ as in \eqref{eq:polarfirstmodel} and
$\ov{\rho}$ as in \eqref{eq:polarsecondnmodel}.  Since $K_1$ is
fiber harmonic near $\fff$, by \eqref{eq:real-fiber-harmonic-projection} and \eqref{eq:Hodge-Lapl-on-fiber-harmonic} we have
\begin{equation*}
  \begin{split}
  \Delta K_1 &= \wt{\Delta}_0 K_1 + x^{-2k} \Delta^{Z, y} K_1 +
    x^{-k} (\delta_Z Q_1 + d_Z Q_2) K_1 + x^{-1} \wt{E} K_1 
\\
&= \wt{\Delta}_0 K_1 + x^{-k} (\delta_Z Q_1  + d_Z Q_2)K_1 + x^{-1} \wt{E} K_1\\
&\qquad + x^{-k} (d_Z \wt{Q}_1 + \delta_Z \wt{Q}_2) K'
+ O(\rho_{\fff}^{-1 - \dimY - kf + 2k}). 
  \end{split}
\end{equation*}
Furthermore, by \eqref{eq:Hodge-Lapl-on-fiber-harmonic-2} we have that
$\Pi_{\mathcal{H}}t(\p_t + \Delta) \Pi_{\mathcal{H}} K_1$ is order $
-1 - \dimY - kf + 1$ since its leading order term solves the model
problem.

We assert the existence of a polyhomogeneous distribution $A$ of order
$-1 - \dimY - kf + k$ such that $t(\p_t + \Delta) (K_1 - A)$ itself vanishes
to order $-1 - \dimY - kf + 1$ at $\fff$.  Indeed, since the leading
order term in $t(\p_t + \Delta)$ is $tx^{-2k} \Delta^{Z, y}$, and since by
\eqref{eq:solvability} we
can solve $\Delta^{Z, y} B = (d_Z \wt{Q}_1 + \delta_Z \wt{Q}_2) K' +
    \delta_Z Q_1 K_1 + O(\rho_{\fff}^{-1 -
  \dimY - kf + k})$ where $B$ is polyhomogeneous with asymptotic
expansion determined by the expansion of the right hand side, in
particular $B = O(\rho_{\fff}^{-1 - \dimY - kf})$.  We take $A = x^{k} B$ and
thus obtain, with $\wt{P}$ as in \eqref{eq:Hodge-Laplacian-decomp}, 
\begin{equation*}
  \begin{split}
    t(\p_t + \Delta)(K_1 - x^k B) &= t(\p_t + \wt{\Delta}_0) (K_1 -
    x^k B) + t x^{-1} \wt{E} (K_1 -
    x^k B) 
\\
    &\qquad 
    - t x^{-k} \wt{P} x^k B  + t O(\rho_{\fff}^{-1 -
  \dimY - kf + 2k})\\
 &= t(\p_t + \wt{\Delta}_0) (K_1 -
    x^k B) +
    t  O(\rho_{\fff}^{-1 - \dimY - kf}) \\
    &\qquad + t x^{-1} O(\rho_{\fff}^{-1 -  \dimY - kf})
    + t O(\rho_{\fff}^{-1 - \dimY - kf })  \\
&= O(\rho_{\fff}^{-1 - \dimY - kf + 1}).
  \end{split}
\end{equation*}
Since the expansion of $B$ at $\ff$
has the same order as $K_1$, the distribution
$$
K_2 = K_1 - x^k B
$$
has all of the desired properties of $K$ in the statement of the
proposition except that $(\p_t + \Delta) K_2$ is not rapidly
decreasing at $\lf$.  Note that, since $\rho_{\fff}^{1 + b + kf} K_1 = O(s^{\ov{\nu}(\NN)})$
where $\ov{\nu}$ is the (fiber degree dependent) order of $\kappa$
computed in \eqref{eq:first-model-asymptotics}, by well-posedness $B$
also satisfies $B= O(\rho_{\lf}^{\ov{\nu}})$.

To deal with the expansion at $\lf$ we argue along similar lines, but
there we iterate the argument to get a parametrix
$K$ with $(\p_t + \Delta) K$ vanishes to infinite order at $\lf$.  (We
work in the interior of $\lf$ though the arguments at the
intersection of $\lf$ and $\fff$ are the same in the projective
coordinates $s' = x/\wt{x}, \eta' = (y - \wt{y})/\wt{x}, \tau' = t/
\wt{x}^2$ together with $z, \wt{x}, \wt{y}, \wt{z}$.)  Recall that $K_1
\equiv \kappa$ near $\lf$ and thus $K_2 = \kappa - x^k B$ near
$\ff_1$. Again with $\wt{P}$ as in \eqref{eq:Hodge-Laplacian-decomp},
we have
\begin{equation}
  \label{eq:left-face-1}
  \begin{split}
    (\p_t + \Delta) K_2 &= x^{-k} \wt{P} \kappa + x^{k - 2} \wt{E}
    \kappa - x^{-k } \Delta^{Z, y} B - x^{-k}\wt{P}x^{k}B + O(x^{\ov{\nu} + k})\\
&= x^{-k} (d_Z Q_1 + \delta_Z Q_2) \kappa - x^{-k} \Delta^{Z, y} B+ O(x^{\ov{\nu} + k -
      2}),
  \end{split}
\end{equation}
where $\ov{\nu}$ is the leading order power of $\kappa$ computed in
\eqref{eq:first-model-asymptotics}.
As in the argument at $\fff$, since the RHS of \eqref{eq:left-face-1}
manifestly gives that $\Pi_{\mathcal{H}}( (\p_t + \Delta) K_2) =
O(x^{\ov{\nu} + k - 2})$, by Section \ref{sec:fibharmforms} there is distribution $A_0$ such that
$x^{\ov{\nu}} \Delta^{Z, y} A_0 = (d_Z Q_1 +
  \delta_Z Q_2) \kappa - \Delta^{Z, y} B + O(x^{\ov{\nu} + k})$.  Here the $x^{\ov{\nu}}$ in front
  makes it so that $A_0$ is $O(1)$.  Thus
$$
(\p_t + \Delta) (K_2  - x^{\ov{\nu} + k} A_0)  = O(x^{\ov{\nu} + k - 2}) - x^{-k}
\wt{P} x^{\ov{\nu} + k} A_0 = O(x^{\ov{\nu}}).
$$

We will now solve away iteratively to decrease the order of the
error.  For this we assume for the moment that we are given, for some $q > \ov{\nu} + \epsilon$,
\textit{any} distribution $A_1 = x^q \wt{A}_1 + O(x^{q + \epsilon})$
with $\wt{A}_1$ smooth and non-vanishing up to the boundary as in $\ice-$form.
First, we find a distribution $B_1$ so that $x^qA_2 := (\p_t + \Delta) (x^{q + 2k}B_1 ) - A_1$ is fiber harmonic.  We
can do this by solving $(I - \Pi_{\mathcal{H}})A_1 = \Delta^{Z, y} B_1 +
O(x^k)$ as in Section \ref{sec:fibharmforms}, where $\Pi_\mathcal{H}$
is the projection onto the fiber harmonic forms, since then $(\p_t +
\Delta) x^{q + 2k}B_1 = x^q \Delta^{Z, y} B_1 + O(x^{q + k})$.  We then construct a
term $C_1$ with $(\p_t + \Delta) x^{q + 2} C_1 \approx  A_2$, as
follows.  Decomposing $A_2 = (A_2^1, A_2^2)$ as in
\eqref{eq:ice-decomp} in the left varibles, and noting that if $C_1 =
((-(q + 2)^2 - (\alpha - 1)(q + 2) + \beta)^1 A_2^{-1} , (-(q + 2)^2 - (\alpha - 1)(q + 2) +
\gamma)^{-1} A_2^2)$, then
$$
\lp
\begin{array}{cc}
  P_{\alpha, \beta} & 0 \\ 0 & P_{\alpha, \gamma}
\end{array}
\rp x^{q + 2} C_1 = x^q  A_2.
$$
(The numbers we divided by above are non-zero, since the indicial roots
of $P_{\alpha, \beta}$ and $P_{\alpha, \gamma}$ are bounded above by
$\ov{\nu} - \epsilon$, as explained below \eqref{eq:first-model-asymptotics}.)
For this $C_1$ we have
\begin{equation*}
  \begin{split}
    x^q A_2 - (\p_t + \Delta) x^{q + 2}C_1 &= x^{q} A_2 -
    \wt{\Delta}_0 x^{q + 2} C_1 + x^{- k} \wt{P}' x^{q + 2}C_1 + O(x^{q + 2}) \\
&= O(x^{q + \delta}) + x^{- k} \wt{P}' x^{q + 2}C_1 + O(x^{q + 2 + k
      - 2}) ,
  \end{split}
\end{equation*}
where $q + \delta$ can be taken to be the order of the subsequent term
in the expansion of $A_2$
where $\wt{\Delta}_0$ is in \eqref{eq:Hodge-Lapl-on-fiber-harmonic}
and $\wt{P}$ is as in \eqref{eq:Hodge-Laplacian-decomp}, and thus
by Section \ref{sec:fibharmforms} we see that the left hand side lies
in the image of $\Delta^{Z, y}$ to order
$x^{k}$. We can thus find a distribution $D_1$ such that
\begin{equation*}
  \begin{split}
    x^q A_2 - (\p_t + \Delta)( x^{q + 2}C_1  - x^{q + 2 + k}D_1) &=
    O(x^{q + 1}) - x^{q + 2 - k} \Delta^{Z, y} D_1 + x^{k} \wt{P}' x^{q +
      2} C_1 \\
&= O(x^{q + 1}) ,
  \end{split}
\end{equation*}
which gives 
\begin{equation}
  \label{eq:improvement}
  (\p_t + \Delta)(x^q(x^{2k} B_1 - x^{2}C_1  + x^{2 + k}D_1)) =
   x^qA_1 + O(x^{q + \delta}).
\end{equation}
It is straightforward to check that the added terms do not increase
the order of blowup at $\fff$.  Thus we can kill off the leading order
term of $x^q A$, and in fact
can kill off all terms iteratively by this process.  (If there are log
terms present the arguement is analogous and left to the reader.)

From the previous two paragraphs, it follows that we can find a
distribution $K'$ such that $K := K_2 - K'$ satisfies the requirements
of the lemma, specifically such that $t(\p_t + \Delta)K$, in addition
to having the same leading order asymptotics at $\tf$ and $\ff$ and
$\fff$ that $t(\p_t + \Delta) K_2$ has, also vanishes to infinite
order at $\lf$.  Indeed, since we can solve away terms to obtain
errors of succesively decreasing order, taking the Borel sum \cite{damwc}
of these distributions gives $K'$.
\end{proof}

\begin{proof}[Proof of Lemma \ref{thm:pre-ansatz}]

We seek a distribution $K'$ with the stated asymptotic properties;
namely \eqref{eq:39}-\eqref{eq:ansatzfaces} and  \eqref{eq:tf-expansion} for
the indicated $b_j$.  Such a $K'$ will exist by  Lemma
\ref{thm:matching} in Appendix \ref{sec:mwc}
provided the hypotheses are satisfied, meaning that the following
matching conditions hold.  We must find a set $\{ \rho_\bullet \}$ of
boundary defining functions for the boundary hypersurfaces, $\bullet
= \lf, \rf,
\tf, \tb, \ff, \fff$ of $\Mheat$ such that
\begin{equation}\label{eq:practical-matching}
  \begin{split}
   \kappa_{\ff} &=
    \frac{1}{(4\pi)^{n/2} } \tau^{-n} \sum_{j \in \mathbb{N}} \tau^{j}
    \wt{b}_{j} \rvert_{\ff},\\
    \wt{x}^{kn}\kappa_{\fff}
   &= \wt{x}^{1 + \dimY + kf}\kappa_{\ff}  \mbox{ on } \ff \cap \fff,
  \end{split}
\end{equation}
and that $\kappa_{\ff}, \kappa_{\fff}$ and the $b_j$ vanish to
infinite order at $\tb$.  Indeed, in the notation of Lemma
\ref{thm:matching} we have $\kappa_1 =     (\rho_{\fff} / \wt{x})^{1 +
  \dimY + kf} \kappa_{\fff}$ and $\kappa_2 = (\rho_{\ff} /
\wt{x})^{kn} \kappa_{\ff}$, and the matching conditions in terms of
$\kappa_1$ and $\kappa_2$ in Lemma \ref{thm:matching} are exactly \eqref{eq:practical-matching}.
We use boundary
defining functions $\rho_{\ff} = \overline{\rho}, \rho_{\fff} = \rho /
\overline{\rho}$ for the faces $\ff$ and $\fff$ defined in
\eqref{eq:polarsecondnmodel} and  \eqref{eq:polarfirstmodel}.  We
define boundary defining functions $\rho_{\lf}, \rho_{\rf}$ for the
faces $\lf$ and $\rf$ by the equations
\begin{equation}
  \label{eq:bdfsatfffff}
  \rho_{\lf}
\rho_{\ff} \rho_{\fff} = \wt{x},  \quad 
\rho_{\rf} \rho_{\ff} \rho_{\fff} = x.
\end{equation}
Finally, we use $\tau$ in \eqref{eq:coordsfftf} as $\rho_{\tf}$;
though it is not valid at $\tb \cap \tf$, all the distributions in
question will vanish to infinite order there and there will be no
conditions to check.

The first matching condition in \eqref{eq:practical-matching} follows
easily since the coefficients of the expansion of $\kappa_{\fff}$ are
determined by the same differential equation which determines the
$b_j$, and the coefficients in both expansions are uniquely determined
by their being equal to polynomials times Gaussians on the fibers of
$\tf \cap \ff$.

  Finally we check that the second condition in
  \eqref{eq:practical-matching} holds.  
First we consider $\kappa_{\fff} = \kappa_{\fff, \wt{y}}(s, 1, \eta, t')$
above the point $\wt{y} \in Y$ (i.e.\ restricted to $\fff_{\wt{y}}$).
In the polar coordinates in \eqref{eq:secondmodel} and using the boundary
defining functions above \eqref{eq:bdfsatfffff}, we have
\begin{equation}
  \label{eq:32}
  s = \frac{\overline{\psi}_{x} \rho_{\ff}^{k-1}}{\rho_{\lf}}  + 1,
  \quad t' =  \frac{\overline{\phi}_{t}
    \rho_{\ff}^{2(k-1)}}{\rho^{2}_{\lf}} , \quad \eta = \frac{\overline{\psi}_{y}
    \rho_{\ff}^{k-1}}{\rho_{\lf}}.
\end{equation}
Using \cite[Eqn.\ 9.7.1]{AS1964}, we have that the modified Bessel function
satisfies $$I_{\nu}(z) = (e^{-z}/\sqrt{2 \pi z}) (1 + O(1/z)),$$  and thus
\begin{equation}
  \label{eq:40}
  \rho_{\ff}^{kn}(\rho_{\lf} \rho_{\ff})^{-1 - \dimY - kf}
    \kappa_{\fff, \wt{y}} = \frac{\rho_{\lf}^{-kf}}{(4 \pi
      \overline{\phi}_{t})^{(b + 1)/2}} e^{-(|\overline{\psi}_{x}^{2} +
    |\overline{\psi}_{y}|^{2}_{\wt{y}}) / 4 \overline{\phi}_{t}} ( 1 + O(\rho_{\ff})).
\end{equation}
On the other hand, above each base
  point $\wt{y} \in Y$, $\kappa_{\ff, \wt{y}}(\sigma, \eta', z,
  z', \wt{T})$ can be written using separation of variables with
  respect to the spectrum of $\Delta^{Z, y}$.  Indeed, since $H_{Z,y}$
  has discrete spectrum, it is standard that $H_{Z,y}(z, \wt{z}, t) =
  \Pi_{0} + E$ where $\Pi_{0}$ is projection onto the kernel of
  $\Delta^{Z,y}$ and $|E| < e^{-\lambda_{0} t}$ as $t \to \infty$,
  $\lambda_{0}$ being the smallest non-zero eigenvalue of $\Delta_{Z,y}$. Thus
  \begin{equation}
    \label{eq:seporvars}
    \kappa_{\ff, y} = (2 \pi \wt{T})^{-(b +1)/2}e^{-(\sigma^{2} +
        |\eta'|_{h_{1}}^{2})/2} \Pi_{0} + E',
  \end{equation}
where $E'$ is exponentially decaying.  Now we have
\begin{equation}
  \label{eq:42}
  \wt{T} = \overline{\phi}_{t} \rho_{\fff}^{-2(k - 1)}
  \rho_{\lf}^{-2k}, \quad \eta' = \overline{\phi}_{t} \rho_{\fff}^{-(k
    - 1)} \rho_{\lf}^{-k}, \quad \sigma = \overline{\psi}_{x}
  \rho_{\fff}^{-(k - 1)} \rho_{\lf}^{-k},
\end{equation}
and thus
\begin{equation}
  \label{eq:43}
  \rho_{\fff}^{1 + \dimY +
      kf}(\rho_{\lf} \rho_{\fff})^{-kn} \kappa_{\ff}  =
    \frac{\rho_{\lf}^{-kf}}{(4 \pi \overline{\phi}_{t})^{(b + 1)/2}}
    e^{-(|\overline{\psi}_{x}^{2} +
    |\overline{\psi}_{y}|^{2}_{\wt{y}}) / 4 \overline{\phi}_{t}},
\end{equation}
so the matching condition at $\ff \cap \fff$ holds.

For the behavior of $K_1$ at $\lf$, since $\kappa_{\fff}$ is given by $\wt{x}^{- 1 - \dimY - kf} \kappa$
at the boundary and we can take $K_1$ to be equal to $\kappa$ at $\lf$
and at $\rf$.\end{proof}

We can now conclude the proof of Theorem \ref{thm:heatkernel} modulo
arguments in Appendix \ref{sec:triple-space}.
\begin{proof}[Proof of Theorem \ref{thm:heatkernel}]
  Having established the existence of a parametrix $K$ as in
  Proposition \ref{thm:ansatz}, we will now prove Theorem
  \ref{thm:heatkernel}.  We do so by inverting the error $ Q =
  t(\p_t + \Delta) K$ from Proposition \ref{thm:ansatz} via a Neumann
  series.  To be precise, it will be convenient to think of distributional
  kernels $A(p, p', t)$ on $M \times M \times \mathbb{R}^+$ acting on
  $C_{c}^{\infty}(M^{\circ} \times (0, \infty))$ by operating as
  convolution kernels in the time variable, so for $\phi \in
  C_{c}^{\infty}(M^{\circ} \times (0, \infty))$ by
  \begin{equation}\label{eq:convolution-operator}
    (A \star \phi)(p, t)  := \int_{M} \int_{0}^{t} A(p, p', t - s)
    \phi(p', s) ds \dVol_{p'}. 
  \end{equation}
  Then
  \begin{equation}\label{eq:real-error-to-iterate}
(\p_t + \Delta) \wt{K} = I + t^{-1}Q,
\end{equation}
and the right hand side
  can be inverted via a Neumann series, i.e.\ $(\operatorname{Id} +
  t^{-1}Q)(I + Q') = \operatorname{Id}$ where $Q' = \sum_{j =
    1}^{\infty}(-1)^j (t^{-1} Q)^{j}$ and $(t^{-1} Q)^j =  t^{-1} Q
  \star \cdots \star t^{-1} Q$, $j$-times.
 In Proposition \ref{thm:composition} below, we show that the summands
 $(t^{-1} Q)^j$ vanishes at successively faster orders at $\ff$ and
 $\fff$.  Moreover, as discussed in \cite{MV2012,BGV2004},
  this series is convergent in $C^\infty$, and the infinite order of
  vanishing of $t^{-1} Q$ at $\lf$ is preserved in the sum, i.e. $Q'$
  vanishes also to infinite order there.  In fact, one sees as in
  \cite{MV2012} that $\wt{K} (I + Q')$ is polyhomogeneous with the
  index set $\mathcal{E}$ statisfying the properties of Theorem \ref{thm:heatkernel}.
\end{proof}

\section{Spectral and Hodge theoretic properties of the Hodge-Laplacian}\label{sec:proofs}

In this section we deduce the main theorems from the
introduction.  We begin with a detailed analysis of the
polyhomogeneous forms in the maximal domain.

\subsection{Polyhomogeneous forms in $\mathcal{D}_{\max}$ and
   $\mathcal{D}_{\min}$}

Recall the definition of $\mathcal{D}_{max}$ and $\mathcal{D}_{min}$
from the introduction, and the space $\mathcal{A}_{\phg}(\Lice)$ of
polyhomogeneous $\ice$-forms (below denoted simply by
$\mathcal{A}_{\phg}$) discussed in Section \ref{sec:mwc}.
We   determine   conditions which assure that  a given polyhomogeneous
differential form $\gamma\in \mathcal{A}_{\phg}$ is contained in the
maximal domain $\mathcal{D}_{\max}$ of $\Delta^g$. This will be used
to show, with an additional assumption on the index set of a
$\phg$ form, that
\begin{equation}\label{eq:7}
  \gamma \in \mathcal{D}_{\max} \cap \mathcal{A}_{\phg} \implies \gamma
  \in \mathcal{D}_{\min} \cap \mathcal{A}_{\phg}.
\end{equation}

Let $\gamma\in \mathcal{A}_{\phg}$ be contained in the maximal domain,
i.e.~we assume that $\gamma\in L^2$ and $\Delta^g\gamma\in L^2$. Let
$\gamma=x^s\tilde\gamma$ where $\tilde\gamma= \tilde{\gamma}_0(y, z) + \mathcal O(x^\epsilon)$. Here notation such as $\mathcal O(x^{\epsilon})$ indicates that the differential form $\gamma$ is locally a combination of basis forms
\begin{eqnarray*}
 \quad dy_I \wedge x^{k\NN}dz_A\qquad \textrm{and}\qquad dx\wedge x^{k\NN}dy_I \wedge x^{k\NN}dz_A,
\end{eqnarray*}
where $I$ and $A$ are multi-indices on the base and fiber, respectively.
with coefficient functions which are bounded by $c x^\epsilon$
pointwise in norm when
$x\searrow 0$, and $\tilde{\gamma}_0$ is a form on $M$ whose
coefficient functions are independent of $x$. Let us determine the
possible range of values $s$. From \eqref{eq:cuspedgemetric} it
follows that in a neighborhood of the boundary, the volume form of the cuspedge metric $g$ is 
\begin{equation*}
\dVol_g=x^{kf}\rho\,dx\wedge dy \wedge dz,
\end{equation*}
where $\rho=  a(y,z) + \mathcal{O}(x^k)$ and $a$ is a
non-vanishing positive function.  It follows that
\begin{eqnarray}\label{eq:condL2}
x^s\tilde\gamma\in L^2(M,g)  \iff s>-\frac{1}{2}(kf+1).
\end{eqnarray}
We begin by analyzing the indicial roots of $\Delta^g$, specifically
we find the order of vanishing of approximately fiber harmonic homogeneous forms in the kernel of
$\Delta^g$.  By Proposition \ref{thm:Hodge-Laplacian}, the leading
order part of $\Delta^g$ restricted to approximately fiber harmonic
forms (see Section \ref{sec:fibharmforms}) is
\begin{eqnarray*}
\Pi_{\mathcal{H}} \Delta^g_0 \Pi_{\mathcal{H}} \sim := \lp
\begin{array}{cc}
 P_{\alpha({\NN}), \beta({\NN})} & 0 \\ 0 &  P_{\alpha({\NN}), \gamma({\NN})} 
\end{array}
\rp,
\end{eqnarray*}
with $P_{\alpha({\NN}), \beta({\NN})},  P_{\alpha({\NN}),  \gamma({\NN})}$ the operators, depending on fiber degree, defined in
\eqref{eq:regsingheat}--\eqref{eq:alpha-beta-gamma}
We note that
\begin{equation}\label{eq:symindroots}
 P_{\alpha({f - \NN}), \beta({f - \NN})} =   P_{\alpha({\NN}),  \gamma({\NN})} \qquad (\NN=0,\ldots,f).
\end{equation}
By \eqref{eq:indicial-roots}, the values $s$ for which
$P_{\alpha(\NN), \beta(\NN)} x^s = 0$ and for which elements of size $x^s$ are in addition contained in $L^2(M,g)$ is
\begin{equation}\label{eq:L2indroots}
s=\begin{cases} -k\NN&\quad\textrm{if}\; \NN<\frac{1}{2}(f+\frac{1}{k}),\\1-k(f-\NN)&\quad\textrm{if}\; \NN>\frac{1}{2}(f-\frac{3}{2k}). \end{cases}
\end{equation}

\begin{proposition}\label{prop:criterionDmax}
Suppose the differential form
$\gamma=(\gamma^1,\gamma^2)=(x^{s_1}\tilde \gamma^1,x^{s_2}\tilde
\gamma^2)\in \mathcal{A}_{\phg}$ and that $\tilde\gamma^j=
\tilde{\gamma}_0^j(y, z) + \mathcal O(x^\epsilon)$ is contained in the
maximal domain $\mathcal{D}_{\max}$.  (Thus the leading order term is
assumed not to have a logarithm, as is a priori allowed for
$\phg$-distributions.) Then  each  $s_j$ is an indicial
root of $P_{\alpha(\NN), \beta(\NN)}$ for some $0\leq \NN_j\leq f$ or $s_j>\frac{1}{2}(-kf+3)$. In either case, $s_j\geq \frac{1}{2}(-kf+3)$.
\end{proposition}

\begin{proof}
Recall form Proposition \ref{thm:Hodge-Laplacian} that 
\begin{equation*}
\Delta^g = \Delta_0 + x^{-k}\wt{P} + x^{-1} \wt{E},
\end{equation*}
where 
\begin{equation}\label{eq:leadingtermLapl}
\Delta_0 =  \lp
\begin{array}{cc}
 P_{\alpha({\NN}), \beta({\NN})} & 0 \\ 0 &  P_{\alpha({\NN}), \gamma({\NN})} 
\end{array}
\rp
+ \lp
\begin{array}{cc}
  \frac{1}{x^{2k}} \Delta^{Z,y} + \Delta_{Y}  & - 2k x^{-k - 1} d^{Z} \\
  - 2 k x^{-k - 1} \delta^{Z} &  \frac{1}{x^{2k}} \Delta^{Z,y} + \Delta_{Y} 
\end{array} \rp.
\end{equation} 
In view of the symmetry  \eqref{eq:symindroots}  it suffices to consider the image of the  the component  $\gamma^1=x^{s_1}\tilde\gamma^1$ under  $\Delta^g$.  The discussion naturally falls into several  cases. 

\begin{itemize}
\item[(1)]   The form $\tilde\gamma_0^1$ is not approximately fiber harmonic in the sense of \textsection\ref{sec:fibharmforms}. Then the lowest nonvanishing  term in \eqref{eq:leadingtermLapl}  is
  $x^{-2k+s_1}\Delta^Z\tilde\gamma_0^1$, which is
  contained in $L^2$ if and only if
\begin{eqnarray*}
s_1>\frac{1}{2}(3kf-1).
\end{eqnarray*}

\item[(2)] The form $\tilde\gamma_0^1$ is approximately fiber harmonic harmonic. We then consider the following subcases.
\begin{itemize}
\item[(2.a)] $s_1$ is an   indicial root of $P_{\alpha({\NN_1}), \beta({\NN_1})}$ and hence equals the number in \eqref{eq:L2indroots}.
\item[(2.b)] $s_1$ is not an indicial root of $P_{\alpha({\NN_1}), \beta({\NN_1})}$, i.e.~$P_{\alpha({\NN_1}), \beta({\NN_1})}(x^{s_1}\tilde\gamma^1)\neq0$. We claim that at least one of the following two statements  holds true:
\begin{itemize}
\item 
The polyhomogeneous expansion of $\tilde\gamma^1$ contains a term $\tilde\gamma_{\ell}^1$ of order $\mathcal O(x^{\delta})$ where $\delta-2k<s_1-2$ and  $\tilde\gamma_{\ell}^1$ is not approximately fibre harmonic.
\item
The  lowest nonvanishing term in the first component of $\Delta^g\gamma$ is  of order $x^{s_1-2}$. 
\end{itemize}
If this claim holds true we conclude that the lowest nonvanishing term in the first component of $\Delta^g\gamma$ is  of order at most $x^{s_1-2}$. To prove the claim, assume that the first statement is false. Then the second one must hold true as  is clear from the form of the  Laplacian $\Delta_0$   in \eqref{eq:leadingtermLapl}. To be specific, collecting the terms of order $x^{s_1-2}$ in the first component of $\Delta_0\gamma$  we obtain 
\begin{multline}\label{eq:termsorders2}
P_{\alpha({\NN_1}), \beta({\NN_1})}(x^{s_1}\tilde\gamma^1)+x^{-2k}\Delta^Z\tau^1 +x^{-k-1}d^Z\tau^2+x^{-k}d_ZQ_2\tau^3+x^{-k}Q_4d_Z\tau^4\\
+x^{-k}\delta_ZQ_1\tau^5+x^{-k}Q_3\delta_Z\tau^6
\end{multline} 
for suitable differential forms $\tau^1,\ldots,\tau^6$ of orders
$$
\tau^1= O(x^{s_1+2k-2}),\quad \tau^2=O(x^{s_1+k-1}),\mbox{ and } \tau^j=O(x^{s_1+k-2})\quad (j=3,4,5,6).
$$
By Hodge theory, the terms $d_Z\tau^4$ and $\delta_Z\tau^6$
both vanish approximately in the sense of \textsection\ref{sec:fibharmforms}, since otherwise a nonvanishing term $x^{-2k}\Delta^Z\tau^4$, respectively $x^{-2k}\Delta^Z\tau^5$ would occur. These are  both  of order strictly less than   $s_1-2$, contradicting our initial assumption. Considering the remaining five terms in \eqref{eq:termsorders2} it follows from Hodge theory and the assumption that  $\tilde\gamma_0^1$ is approximately fibre harmonic  that the sum 
\begin{eqnarray}\label{eq:restterms}
x^{-2k}\Delta^Z\tau^1 +x^{-k-1}d^Z\tau^2+x^{-k}d_ZQ_2\tau^3 +x^{-k}\delta_ZQ_1\tau^5
\end{eqnarray}
is approximately orthogonal to $P_{\alpha({\NN_1}), \beta({\NN_1})}(x^{s_1}\tilde\gamma^1)$ in the sense of \textsection\ref{sec:fibharmforms}. Hence we conclude that the nonzero term $P_{\alpha({\NN_1}), \beta({\NN_1})}(x^{s_1}\tilde\gamma^1)$ cannot cancel with the sum \eqref{eq:restterms}. It follows that the second statement is true, whence the claim.
\end{itemize}
\end{itemize}
The asserted statement follows by inspection of each of the above cases. In case (1) it follows from
\begin{equation*}
s_1> \frac{1}{2}(3kf-1) >\frac{1}{2}(-kf+3),
\end{equation*} 
using that $k\geq3$. In case (2.b) the lowest nonvanishing term in $\Delta^g\gamma$  is of order at most $s_1-2$. Since $\gamma\in\mathcal D_{\max}$ it follows from \eqref{eq:condL2} that
\begin{equation*}
s_1-2>-\frac{1}{2}(kf+1)\Longleftrightarrow s_1>-\frac{1}{2}(kf+3).
\end{equation*}
In case (2.a), the form $\tilde\gamma_0^1$ is approximately fibre harmonic and therefore by the Witt condition $\NN\neq\frac{f}{2}$. The exponent $s_1$ is given by \eqref{eq:L2indroots} from which it follows that if $f$ is even that
\begin{equation*}
s_1\geq-k\big(\frac{f}{2}-1\big)
\end{equation*}
(here we use the Witt condition) and if $f$ is odd that
\begin{equation*}
s_1\geq-k\big(\frac{f}{2}-\frac{1}{2}\big)\geq  -\frac{kf}{2}+\frac{3}{2} ,
\end{equation*}
where the last inequality follows from the assumption $k\geq3$.
\end{proof}

\begin{lemma}\label{lem:AmaxequalsAmin}
Assume  $k\geq3$. Then  $\mathcal D_{\min}\cap\mathcal A_{\phg}=\mathcal D_{\max}\cap\mathcal A_{\phg}$.
\end{lemma}

\begin{proof}
It suffices to prove the inclusion $\mathcal D_{\max}\cap\mathcal A_{\phg} \subseteq\mathcal D_{\min}\cap\mathcal A_{\phg}$. For $\varepsilon>0$ we define the logarithmic cutoff function $\chi_{\varepsilon}\colon[0,\infty)\to[0,1]$ by
\begin{eqnarray*}
\chi_{\varepsilon}(x):=\begin{cases}0,&x\leq\varepsilon^2,\\-\frac{\log(x/\varepsilon^2)}{\log(\varepsilon)},&\varepsilon^2<x<\varepsilon,\\1,&x\geq\varepsilon.\end{cases}
\end{eqnarray*}
For $\varepsilon^2<x<\varepsilon$ it satisfies
\begin{eqnarray}\label{eq:dercutoff}
\chi_{\varepsilon}'(x)=-\frac{1}{\log(\varepsilon)x}\qquad\textrm{and}\qquad\chi_{\varepsilon}''(x)=\frac{1}{\log(\varepsilon)x^2}.
\end{eqnarray}
Let $\gamma \in\mathcal D_{\max}\cap\mathcal A_{\phg}$ and set $\gamma_{\varepsilon}=\chi_{\epsilon}\gamma$. Then
\begin{multline}\label{eq:Laplacecutoff}
\Delta^g\gamma_{\varepsilon}=\chi_{\varepsilon}\Delta^g\gamma-(\partial_x^2\chi_{\varepsilon})\gamma-(\partial_x\chi_{\varepsilon})(\partial_x\gamma)-\frac{kf}{x}(\partial_x\chi_{\varepsilon})\gamma\\
+A^j(\partial_x\chi_{\varepsilon})(\partial_{z_j}\gamma)+B^i(\partial_x\chi_{\varepsilon})(\partial_{y_i}\gamma),
\end{multline}
where $A^j= \mathcal O(x^k)$ and $B^i= \mathcal O(x^{2k})$ are bounded functions with that order of decay in $x$. We show that
\begin{eqnarray}\label{eq:condDmin}
\|\Delta^g\gamma_{\varepsilon}-\Delta^g\gamma\|_{L^2(M,g)}\to0\qquad\textrm{as}\quad\varepsilon\to0,
\end{eqnarray}
hence establishing that $\gamma\in\mathcal D_{\min}$. It is clear that
\begin{equation*}
\|\chi_{\varepsilon}\Delta^g\gamma-\Delta^g\gamma\|_{L^2(M,g)}\to0\qquad\textrm{as}\quad\varepsilon\to0,
\end{equation*}
and thus it  suffices to consider the next three terms in \eqref{eq:Laplacecutoff} and to show that
\begin{eqnarray}\label{eq:errorterms}
\frac{1}{\log^2(\varepsilon)}\int_{\varepsilon^2}^{\varepsilon}\frac{1}{x^4}|\gamma|^2x^{kf}\,dx+\frac{1}{\log^2(\varepsilon)}\int_{\varepsilon^2}^{\varepsilon}\frac{1}{x^2}|\partial_x\gamma|^2x^{kf}\,dx +\frac{k^2f^2}{\log^2(\varepsilon)}\int_{\varepsilon^2}^{\varepsilon}\frac{1}{x^4}|\gamma|^2x^{kf}\,dx
\end{eqnarray}
converges to $0$ as $\varepsilon\to0$. Let $\gamma=x^s\tilde\gamma$ for some   $\tilde\gamma=\mathcal O(1)$. A short calculation   shows that each  integrand in   \eqref{eq:errorterms} is of order $x^{-1+\delta}$ for  some $\delta>0$ and hence converges   to $0$ as $\varepsilon\to0$  if
\begin{eqnarray}
s>-\frac{kf}{2}+\frac{3}{2}.
\end{eqnarray}
In the borderline case $s=-\frac{kf}{2}+\frac{3}{2}$ we still get convergence since then the first integral in \eqref{eq:errorterms} becomes
\begin{eqnarray*}
\frac{1}{\log^2(\varepsilon)}\int_{\varepsilon^2}^{\varepsilon}\frac{1}{x}\,dx=\frac{1}{\log^2(\varepsilon)}(\log(\varepsilon)-\log(\varepsilon^2))=-\frac{1}{\log(\varepsilon)}\to0\quad\textrm{as}\;\varepsilon\to0,
\end{eqnarray*}
and analogously for the second and third integral. Hence 
\begin{equation*}
s\geq-\frac{kf}{2}+\frac{3}{2}\quad \Longrightarrow\quad  \gamma \in\mathcal D_{\min} 
\end{equation*}
for any  $\gamma=x^s\tilde\gamma \in\mathcal D_{\max}\cap\mathcal A_{\phg}$. On the other hand, Proposition \ref{prop:criterionDmax} shows that
\begin{equation*}
\gamma=x^s\tilde\gamma \in\mathcal D_{\max}\cap\mathcal A_{\phg}\quad \Longrightarrow\quad s\geq-\frac{kf}{2}+\frac{3}{2},
\end{equation*}
and hence the claim follows.
\end{proof}

\subsection{Spectral theory}\label{sec:spectral-theory}

\begin{proof}
[Proof of Theorem  \ref{thm:heatkernelmap}]
The existence of a fundamental solution $H_t$ is the content of
Theorem \ref{thm:heatkernel}, so it suffices to show that $H_t$ has
the properties stated in Theorem \ref{thm:heatkernelmap}.  Since $H_t$
and $\p_t H_t$ are formally self-adjoint (i.e.\ symmetric), to show
that they are self-adjoint it suffices to show that they are compact
operators.  But indeed they are, as follows from \cite[Thm
VI.23-24]{RSI} together with
$$
H_t, \p_t H_t \in L^2(\End; M \times M),
$$
where, given a smooth section $A$ of $\End$, then $A \in L^2(\End; M
\times M)$ if
$$
\int \| A(p, q) \|^2_{\End} \dVol_{M}(p) \dVol_M(q) < \infty.
$$
 For $t>0$, $H_{t}$ is given by an $L^{2}$ integral kernel, so is a
 compact operator; indeed, by \eqref{eq:leftfacebehaviour}, the index
 set $\mathcal{F}$ of $H_t \in \mathcal{A}_{\phg}(M \times M)$
 restricted to $t > 0$ constant is $\mathcal{F}(\lf) =
 \mathcal{E}(\lf)$ and $\mathcal{F}(\rf) = \mathcal{E}(\rf)$, for
 $\mathcal{E}$ the index family of $H$.  From \eqref{eq:left-face-index}, these satisfy the lower bound
\begin{equation}\label{eq:lowerbindsetlf}
\inf  \mathcal{F}(\lf), \inf  \mathcal{F}(\rf)\geq -\frac{kf}{2}+1
\end{equation} 
(meaning $H_t$ is a bounded endomorphism) and
\begin{equation}\label{eq:expvolform}
\dVol_{M}(p) \dVol_M(q) \simeq x^{kf} \wt{x}^{kf}\, dx\,  d\wt{x}\, dy\,
 d\wt{y}\, dz\, d\wt{z},
\end{equation}
so the kernel of $H_t$ is square
 integrable.  Since the restriction of  $\p_{t} H_{t}$ to a fixed time
 $t$ has the same index set on $M \times M$ as $H_{t}$, it is also
 compact.  
\medskip\\
It remains to establish \eqref{thm:heatkernelmap}, i.e.~that $H_t(\alpha) \in \mathcal{D}_{\min}$ for every $\alpha\in L^2$. In fact, $H_t(\alpha)$ is a polyhomogeneous distribution with index
set $\mathcal{E}(\lf)$.  This is straightforward: writing the expansion
of $H_t$ at $x = 0$ up to some order $N$ we have
$$
H_t = \sum_{\substack{(s,p) \in \mathcal{E}(\lf) \\ |s| \le N}} x^s
\log^p(x) a_{s, p}(y, z, \wt{w})  + E_N
$$
where $\wt{w} = (\wt{x}, \wt{y}, \wt{z})$, and the coefficients
$a_{s,p}$ are polyhomogeneous endomorphisms on the manifold with
boundary $\p M \times M$ and $E_N$ is a polyhomogeneous endomorphism
on $M \times M$ with $E_N = o(x^N)$.  Thus
\begin{equation}\label{eq:image-is-phg}
H_t(\alpha) = \int_M \Big(\sum_{\substack{(s,p) \in \mathcal{E}(\lf) \\
    \Re s \le N}} x^s \log^p(x) a_{s, p}(z,
y, \wt{w}) \alpha(\wt{w})  + E_N \alpha(\wt{w}) ) \Big) \dVol_g(\wt{w}).
\end{equation}
For example by  \cite[Proposition 3.20]{Ma1991}, since the $x^{-N}E_N$
are given by a polyhomogeneous integral kernel, they define 
bounded maps of $L^2$, and the conormality estimates
(see \eqref{eq:simbound}--\eqref{eq:conormal-estimates}) follow by
differentiating $x^{-N} E_N$. The integrals coming from the partial expansion terms are finite and
give the expansion coefficients of $H_t(\alpha)$. This shows that
$H_t(\alpha)\in\mathcal A_{\phg}$, and moreover that the leading order
term has no logarithmic factor. Thus, In view of Lemma \ref{lem:AmaxequalsAmin} it suffices to prove $H_t(\alpha)\in\mathcal D_{\max}$ in order to conclude that $H_t(\alpha)\in\mathcal D_{\min}$. But indeed, $\inf  \mathcal{E}(\lf)$ satisfies the lower bound \eqref{eq:lowerbindsetlf}, hence it follows that the lowest order term in the polyhomogeneous expansion \eqref{eq:image-is-phg} is of order at least $-\frac{kf}{2}+1$ which by \eqref{eq:condL2} is sufficient to conclude $H_t(\alpha)\in L^2$. Because $H_t$ is a fundamental solution of the heat equation, it follows that $\Delta^gH_t(\alpha)=-\partial_tH_t(\alpha)$ which by the same argument is contained in $L^2$ since  $\partial_tH_t$  has the same index set as $H_t$ for $t>0$.
\end{proof}

It now follows that the fundamental solution $H_t$ from Theorem
\ref{thm:heatkernel} is in fact the heat kernel in the following
sense.
\begin{proposition}\label{thm:fund-soln-is-heat-kernel}
  \label{eq:spectral-heat-kernel}
  The heat kernel $\exp(-t \Delta^g)$ defined by applying the spectral
  theorem to the self-adjoint operator $(\Delta^g, \mathcal{D})$ has
  Schwartz kernel equal to the fundamental solution $H_t$ in Theorem
  \ref{thm:heatkernel}, meaning
$$
(e^{-t \Delta^g} \alpha)(w)  = \int_M H_t(w, \wt{w}, t) \alpha(\wt{w}) \dVol_g(\wt{w}).
$$
\end{proposition}
Using this we may finish the proof of Theorem \ref{thm:essentiallyselfadjoint}.
 
\begin{proof}[Proof of Theorem \ref{thm:essentiallyselfadjoint}]
As discussed below the statement of Theorem \ref{thm:heatkernelmap},
Theorem \ref{thm:heatkernelmap} itself establishes essential
self-adjointness of $\Delta^g$. It remains to prove that the
spectrum is discrete, but this follows immediately from the spectral
theorem and the fact that $H_t$ is a compact operator (hence has
discrete spectrum.)

Moreover, the Weyl asymptotic formula in \eqref{eq:weyl-asymptotics} follows
from the standard heat kernel argument in \cite[\textsection8.3]{taylor:vol2} together with the heat trace asymptotics in
Corollary \ref{thm:heat-trace}.
\end{proof}

\begin{corollary}[Heat trace asymptotics]\label{thm:heat-trace}
  For each $t > 0$, the fundamental solution $H_t$ in Theorem \ref{thm:heatkernel} is trace
  class and satisfies that $F(t) := \Tr H_t$ is a polyhomogenous
  conormal distribution on $\mathbb{R}^+$ satisfying
  \begin{equation}
    \begin{split}
      \label{eq:heat-trace}
      F(t) &= t^{-n/2} \Vol(M,g) + (\sum_{j = 1}^{\infty} a_j t^{-n/2 + j} )+
      c_0 t^{-(\dimY + 1)/2 + 1/2k}  \\
      &\qquad + O(t^{-(\dimY + 1)/2 + 1/2k
        + \epsilon}).
    \end{split}
  \end{equation}
\end{corollary}
The proof of Corollary \ref{thm:heat-trace}, which uses Theorem
\ref{thm:heatkernel} and Melrose's pushforward theorem, is deferred to  Section
\ref{sec:heat-trace-asymptotics} below.

\subsection{Harmonic forms and Hodge theory}\label{sec:hodge-theory}

We begin our discussion of Hodge theory by pointing out that elements
$\gamma \in L^2$
satisfying $\Delta^g\gamma = 0$, admit asymptotic expansions at the
boundary of $M$.  Indeed, for such forms $\gamma$, by the spectral
theorem and the fact that $H_t$ is the heat kernel (Corollary
\ref{thm:fund-soln-is-heat-kernel}), we see that 
\begin{equation}
  \label{eq:phg-kernel}
  \begin{split}
    \gamma = H_t \gamma &  = \int_M H_t(w, \wt{w}, t) \gamma(\wt{w})
    \dVol_g(\wt{w}).
  \end{split}
\end{equation}
By the proof of Theorem \ref{thm:heatkernelmap}, specifically
\eqref{eq:image-is-phg}, we have the following.
\begin{lemma}\label{thm:phg-kernel}
  Assume that $\gamma \in \ker(\Delta^g \colon L^2 \lra
  L^2)$.  Then $\gamma$ is polyhomogeneous conormal and $\gamma =
  \mathcal{O}(1)$, i.e.\ is bounded in norm.
\end{lemma}

Lemma \ref{thm:phg-kernel} allows us to conclude that the $L^2$ kernel of
$\Delta^g$ is equal to the Hodge cohomology in
\eqref{eq:hodge-cohomology} .
\begin{lemma}
  Notation as above, $\mathcal{H}_{L^2}(M, g) = \ker(\Delta^g \colon L^2 \lra
  L^2)$.
\end{lemma}
\begin{proof}
  If $\gamma \in \mathcal{H}_{L^2}(M, g)$ then $\gamma$ is in the
  maximal domains of both $d$ and $\delta$, and so for smooth
  compactly supported $\beta$,
  \begin{align*}
   \la \Delta^g \gamma , \beta \ra_{L^2} &:=  \la \gamma ,
   \Delta^g\beta \ra_{L^2} 
   =  \la \gamma , d\delta \beta \ra_{L^2}  + \la \gamma ,  \delta d
   \beta \ra_{L^2} = 0 + 0 = 0,
  \end{align*}
so $\gamma \in \ker(\Delta^g \colon L^2 \lra
  L^2)$

On the other hand, if  $\gamma \in \ker(\Delta^g \colon L^2 \lra
  L^2)$, then by Lemma \ref{thm:phg-kernel} we can integrate by parts  
  to obtain
  \begin{equation*}
  0 =   \la \Delta^g \gamma , \gamma \ra_{L^2} = \| d \gamma
  \|^2_{L^2} + \| \delta \gamma \|^2_{L^2},
  \end{equation*}
so $\gamma \in \mathcal{H}^*_{L^2}(M, g)$. \end{proof}

We can now follow the arguments in \cite{HHM2004, HR2012} to prove
Theorem \ref{thm:hodge} above.  Before we begin we recall some facts
about intersection cohomology, a cohomology theory that applies to
stratified spaces.  We do not attempt to make a full explanation of it
here, but mention only that there is in fact a family of intersection
cohomology groups for our stratified space $X$ defined in
\eqref{eq:collapsed-manifold} (obtained by collapsing the boundary of
$\p M$ over the base $Y$) depending on a function
$\pp \colon \mathbb{N} \lra \mathbb{N}$ called the `perversity,' which
is non-decreasing and whose values matter only on the codimensions of
the strata of $X$.  Here we have only one singular stratum, $Y \subset
X$, the image of the boundary $\p M$ via the projection onto $X$, and
its codimension is $f + 1$, where $\dim Z = f$.  The
`upper middle degree' perversity $\ov{\mm}$ is a special example of a
perversity, which satisfies    
\begin{equation}
  \label{eq:upper-middle}
      \ov{\mm}(f + 1)  = \left\{
    \begin{array}{cc}
(f - 1)/2 & \mbox{ if } f \mbox{ is odd,} \\
f /2  - 1& \mbox{ if } f \mbox{ is even.}
    \end{array}
\right.
\end{equation}
The `lower middle perversity' $\underline{\mm}$ differs from
$\ov{\mm}$ only when $f$ is even, in which case $\underline{\mm}(f +
1) = f/2$.
As we will rely on the spectral sequence arguments from \cite{HHM2004,
  HR2012} during the proof, we will only need to study the
intersection cohomology locally, specifically on a basis of open sets
of $X$.  Concretely, from \cite{HHM2004}, for canonical neighborhoods $U
= V \times C_1(Z)$ as in \eqref{eq:canon-neighb} with
contractible $V$, we have
\begin{equation*}
  \label{eq:local-intersection-cohomology*}
  I\!H^p_{\pp}(U) = \left\{
    \begin{array}{cc}
     H^p(Z) & \mbox{ if }  p < f - 1 -  \pp(f + 1), \\
      \set{0} & \mbox{ if }  p \ge f - 1 -  \pp(f + 1).
    \end{array}
\right.
\end{equation*}
From the Witt condition \eqref{eq:witt}, we see that
\begin{equation}
  \label{eq:local-intersection-cohomology}
  I\!H^p_{\ov{\mm}}(U) =   I\!H^p_{\underline{\mm}}(U) = \left\{
    \begin{array}{cc}
     H^p(Z) & \mbox{ if }  p < f /2, \\
      \set{0} & \mbox{ if }  p \ge f / 2,
    \end{array}
\right.
\end{equation}
regardless of the parity of $f$.
\begin{proof}[Proof of Theorem \ref{thm:hodge}]
Although Theorem \ref{thm:hodge} describes a relationship between the
Hodge cohomology and the intersection cohomology, to prove it we go
through the standard route and use the intermediary of $L^2$-cohomology.  Thus consider the chain complex
\begin{equation}
  \label{eq:L2-coho-complex}
\dots \lra L^2_d\Omega^{p-1}(M, g) \lra L^2_d\Omega^{p}(M, g) \lra
L^2_d\Omega^{p+1}(M, g) \lra \dots.
\end{equation}
where $L^2_d\Omega^{p}(M, g)$ is the maximal domain of the exterior
derivative $d$, specifically 
$$
L^2_d\Omega^{p}(M, g) = \{ \alpha \in L^2\Omega^p(M, g) : d\alpha \in
L^2\Omega^{p + 1}(M, g) \}.
$$
Then the $L^2$-cohomology is the quotient
$$
L^2H^p (M, g) = \frac{\set{\alpha \in L_d^2\Omega^{p}(M, g) : d\alpha = 0}}
{\set{d\eta : \eta \in L^2_d\Omega^{p-1}(M, g)}}.
$$
As explained in \cite[page 6]{HR2012}, it suffices to show that 
\begin{equation}
  \label{eq:L2-equals-singular}
  L^2H^p(M, g) \simeq I\!H^p_{\ov{\mm}}(X; \mathbb{R}),
\end{equation}
for then the $L^2$-cohomology is finite dimensional, which implies
that the range of $d$ (and thus its adjoint $\delta$) is closed.  From
\cite[\textsection2.1]{HHM2004} it then follows using the Kodaira
decomposition theorem that $\mathcal{H}^p_{L^2}(M, g)$ is
isomorphic to $L^2H^p(M, g)$ and thus by
\eqref{eq:L2-equals-singular} Theorem \ref{thm:hodge} holds.

Thus it suffices to prove \eqref{eq:L2-equals-singular}, and for this
we also follow the arguments in \cite[pp. 5-6]{HR2012}, where it is
explained that it suffices to show that for canonical neighborhoods $U
= V \times C_1(Z)$ as in \eqref{eq:canon-neighb} with
contractible $V$, the local
chain complex 
\begin{equation}
  \label{eq:L2-coho-complex-local}
\dots \lra L^2_d\Omega^{p-1}(U, g) \lra L^2_d\Omega^{p}(U, g) \lra
L^2_d\Omega^{p+1}(U, g) \lra \dots,
\end{equation}
satisfies
\begin{equation}
  \label{eq:L2-equals-singular-local}
  L^2H^p(U, g) \simeq I\!H^p_{\ov{\mm}}(U)
\end{equation}
Here $L^2H^p(U, g)$ is defined as above with $U$ replacing $M$.  The
intersection cohomology groups for $\ov{\mm}$ are computed in
\eqref{eq:local-intersection-cohomology}, and thus we need only to
analyze the groups on the left.  To see \eqref{eq:L2-equals-singular-local} we use the
K\"unneth formula of Zucker, \cite[Corollary 2.34]{Z1982}, whose
assumptions are satisfied here by the fact that the exterior
derivative on $Z$ is closed on its maximal domain.  Thus,
in the notation of \cite[Page 5]{HR2012}, we have 
\begin{equation}\label{eq:local-L2-coho-check}
  \begin{split}
    L^2H^p(U, g)  &= \bigoplus_{i = 0}^1  W\!H^i((0, 1), dx^2, k(p - i -
    f/2))) \otimes H^{p -i}(Z ; \mathbb{R}) ,  \end{split}
\end{equation}
where $W\!H^i((0, 1), dx^2, a)$ is the
cohomology of the complex
\begin{equation}
  \label{eq:L2-coho-complex-local-2}
0  \lra (x^a L^2\Omega^{0}((0, 1), dx^2))\stackrel{d}{\lra} x^a L^2\Omega^{1}((0, 1), dx^2)\lra
0,
\end{equation}
where the space on the left is the maximal domain of $d$ on $x^a
L^2\Omega^{0}((0, 1), dx^2)$.  Again from \cite{HR2012} (via
\cite{HHM2004}), $W\!H^1((0, 1), dx^2, a) = 0$ if $a \neq 1/2$ and
$W\!H^0((0, 1), dx^2, a) = \mathbb{R}$ if $a < 1/2$ and $\{ 0 \}$ if $a
\ge 1/2$.  When $i = 1$, $k(p - i -
    f/2) \neq 1/2$ since $k > 1$, so the $i = 1$ terms do not
    contribute.  When $i = 0$, we have $k(p - i -
    f/2) = k(p - f/2)$ which satisfies
$$
k(p - f/2) < 1/2 \mbox{ if } p \le f/2 \quad \mbox{ and } \quad k(p - f/2) > 1/2
\mbox{ if } p > f/2.
$$
Using the Witt condition then gives
\begin{equation}
  \label{eq:2}
  \begin{split}
    L^2H^p(U, g)  &= 
\left\{
    \begin{array}{cc}
      H^p(Z) &\mbox{ if } p < f/2, \\
      \{ 0 \} &\mbox{ if } p \ge f/2,
    \end{array} \right.
  \end{split}
\end{equation}
matching \eqref{eq:local-intersection-cohomology} and completing the proof.
\end{proof}

We now discuss the proof of Theorem \ref{thm:moduli-space}.  As the
spaces in the theorem are incomplete cusp edge spaces in a
neighborhood of the divisor by \cite{MZ2015}, our results would apply
to these spaces, if not for the fact that moduli spaces such as these
have interior orbifold points.  This is not a problem, since, as in
\cite{JMMV2014} we may lift to a finite cover with no such points.
One can then work on the space $C^\infty_{c, orb}(\mathcal{M}_{1,1})$
of functions which near each orbifold point are smooth when lifted to
a local finite cover resolving the singularity.  Constructing a heat
kernel on the lift and averaging over the group action then gives  
a fundamental solution to the heat kernel downstairs which has all the
desired properties.  We leave the details of this simple extension to
the reader.

\begin{appendix}

\section{Manifolds with corners}\label{sec:mwc}

In this section we recall some of the facts about distributions on
manifolds with corners (mwc's) used in this paper.  This material is
due largely to Melrose, and the reader is referred to his book
\cite{tapsit} for more details.  See also \cite{Hvol1}.

The objects considered here, for example the ice-metrics
\eqref{eq:cuspedgemetric}, have polyhomogeneous regularity, which we
define now.  The sheaf of \textbf{polyhomogeneous conormal} (or
polyhomogeneous, or simply phg) functions $\mathcal{A}_{\phg}(X)$ is
defined as follows.  First, an index set $\mathcal{E}$ on a manifold with corners
$X$ is an association to each boundary hypersurface $H$ of $X$ a set
\begin{equation}
  \label{eq:indexset}
  \begin{split}
    \mathcal{E}(H) \subset \mathbb{C} \times \mathbb{N} \mbox{
      satisfying that the subset} \qquad \\
\set{(z, p) \in \mathcal{E}(H) : \mbox{ Re } z < c} \mbox{ is finite for
  all } c \in \mathbb{R}. 
  \end{split}
\end{equation}
Given an index set $\mathcal{E}$, for a boundary  face
$F = \cap_{i = 1}^{\delta} H_{i}$ for boundary hypersurfaces $H_{i}$,
define the subset $\mathcal{E}(F) \subset \mathbb{C}^{p} \times
\mathbb{N}^{p}$ by $(z, p) = (z_{1}, \dots, z_{\delta}, p_{1}, \dots,
p_{\delta}) \in \mathcal{E}(F)$ if and only if $(z_{i},
p_{i}) \in \mathcal{E}(H_{i})$.  We define the Frechet space
$\mathcal{A}^{\mathcal{E}}_{\phg}(X)$ as follows.  We write $u \in
\mathcal{A}^{\mathcal{E}}_{\phg}(X)$ if and only if for each boundary face $F
= \cap_{i = 1}^{\delta} H_{i}$, writing $\rho_{i}$ for a boundary
defining function of $H_{i}$, $u$ satisfies
\begin{equation}
  \label{eq:phgexpansion}
  \begin{split}
    u &\sim  \sum_{(z, p) \in
      \mathcal{E}(F)} a_{z,p} \rho^{z} \log^{p} \rho \mbox{ where} \\
   \rho^{z} &= \prod_{i = 1}^{\delta} \rho_{i}^{z_{i}}, \qquad \log^{p} \rho =
   \prod_{i  = 1}^{\delta} \log^{p_{i}} \rho_{i} ,
  \end{split}
\end{equation}
and the symbol $\sim$ means that
\begin{equation}
  \label{eq:simbound}
    E_N = u  -   \sum_{\substack{(z, p) \in
      \mathcal{E}(F) \\  \mbox{ Re } z_i < N \ \forall i }} a_{z,p} \rho^{z} \log^{p}
    \rho,
\end{equation}
where $E_N$ is a smooth function on the interior of $X$ which is
$O(|\rho|^N)$, where $|\rho| = (\rho_{1}^{2} + \dots +
\rho_{\delta}^{2})^{1/2}$.  Moreover, $E_N$ is conormal, meaning that if
$\mathcal{V}_b = \mathcal{V}_b(X)$ denotes the set of smooth vector fields on $X$ that are
tangent to all boundary hypersurfaces, then
\begin{equation}\label{eq:conormal-estimates}
|\rho|^{-N} \mathcal{V}_b^k E  \subset L^\infty.
\end{equation}

Note that if a phg function $u$ vanishes to infinite order at $H$,
then $u$ is polyhomogeneous with index set $\mathcal{E}$ satisfying
$\mathcal{E}(H) = \varnothing$.

\begin{lemma}\label{thm:matching}
  Let $X$ denote a mwc, $\mathcal{M}(X) = \set{H_{i}}_{i \in
    \mathcal{I}}$ its boundary hypersurfaces,
  and for each $i \in \mathcal{I}$, let $\rho_{i}$ denote a boundary
  defining function of $H_i$.  Given a smooth vector bundle $E \lra X$, if $\kappa_{i}$ are
  polyhomogeneous sections on
  $H_{i}$, then provided
  \begin{equation}\label{eq:matchingcondition}
  \rho_{i}^{c_{i}} \kappa_{j} \rvert_{H_{i} \cap H_{j}} =
  \rho_{j}^{c_{j}} \kappa_{i} \rvert_{H_{i} \cap H_{j}}
  \end{equation}
there exists a polyhomogeneous conormal distribution $K$ on $X$
satsifying
\begin{equation}\label{eq:desired_restriction}
\rho_{i}^{c_{i}}K \rvert_{H_{i} } = \kappa_{i}
\end{equation}

Assume moreover that at a particular boundary hypersurface which we
take to be $H_1$, that we are given an index set $F_1 \subset
\mathbb{C} \times \mathbb{N}$ and polyhomogeneous sections $b_{j, p}
\in \mathcal{A}_{\phg}(E\rvert_{H^1}; H_1)$.  Then given functions
$\kappa_i$ on $H_i$, $i \neq 1$, there exists a distribution $K$
satisfying \eqref{eq:desired_restriction} for $i \neq 1$ and such that
\begin{equation}\label{eq:desired-expansion}
K \simeq \sum_{s,p \in F_1} \rho_1^{s} \log^p (\rho_1) b_{s, p}
\end{equation}
provided \eqref{eq:matchingcondition} holds for $i, j \neq 1$ and
furthermore for $i \neq 1$
\begin{equation}\label{eq:second-matching-condition}
  \kappa_{i} \simeq
 \rho_i^{c_i} \sum_{s,p \in F_1} \rho_1^{s} \log^p (\rho_1) b_{s, p} \rvert_{H_i}.
\end{equation}

\end{lemma}

\begin{remark}
(1) Note the converse; if $K = \rho_i^{-c_i} \rho_j^{-c_j} a$ for some
positive function $a$
near $H_i \cap H_j$ then setting $\rho_l^{c_l} K \rvert_{H_l}= \kappa_l$ for $l = i,
j$, we have $\rho_j^{c_j} \kappa_i = \rho_i^{c_i} \kappa_j$ on $H_i
\cap H_j$.

(2) The matching condition \eqref{eq:matchingcondition} implies
further matching conditions on multifold intersections, e.g.\ it
implies that
$$
\rho_i^{c_i} \rho_j^{c_j} \kappa_l =  \rho_i^{c_i} \rho_l^{c_l} \kappa_j =
\rho_l^{c_l} \rho_j^{c_j} \kappa_i \mbox{ on } H_i \cap H_j \cap H_l.
$$

(3) The second matching condition \eqref{eq:second-matching-condition}
merely says that the desired data on a bhs $H_i$ has the same
asymptotic expansion at $H_1$ as the the desired distribution
restricted to $H_i$.
\end{remark}

\begin{proof}
  
Denote the number of boundary hypersurfaces of $X$ by
$m = |\mathcal{M}|$. There is a number $\delta$ and boundary defining
functions $\rho_i$ such that the set
$\{ \rho_i < \delta\}$ is diffeomorphic as mwc's to $H_i \times [0,
\delta)$.  Without loss of generality we take $\delta = 1$.
Following the remark, for a collection of bhs' $H_{i_1}, \dots, H_{i_p}$, the distrubution
$$
\kappa_{i_1 \dots i_p} = ( \prod_{i \neq  i_k} \rho^{c_i} )
\kappa_{i_k} \rvert_{\rho_{i_1} = \dots = \rho_{i_p} = 0}
$$
is well-defined independently of the choice of $i_k \in \{ 1, \dots, m
\}$.

Let $\chi(x)$ be a
cutoff function with $\chi \equiv 1$ for $x \le 1/3$ and $\chi
\equiv 0$ for $x \ge 2\epsilon/3$.
For the distribution $K$ we may take
$$
K = \sum_{p = 1}^m (-1)^{p - 1}  \sum_{1 \le i_1 < \dots < i_p \le m}
\kappa_{i_1 \dots i_p}
 (\prod_{j \in  \{ i_1, \dots , i_p \}} \chi(\rho_{j}) \rho_j^{-c_j}).
$$
For example if $m = 2$ then
$$
K = \chi(\rho_1) \rho_1^{- c_1} \kappa_1  + \chi(\rho_2) \rho_2^{-
  c_2} \kappa_2 - \chi(\rho_1) \chi(\rho_2)\rho_1^{- c_1} \rho_2^{- c_2} \kappa_{12}.
$$
Note that each term in the sum defining $K$ defines a polyhomogeneous
conormal distribution on all of $X$, as the
distribution $\kappa_{i_1 \dots i_p} $ is defined on a neighborhood
of $H_{i_1} \cap \dots \cap H_{i_p}$ off which the product $\prod_{j
  \in  \{ i_1, \dots , i_p \}} \chi(\rho_{j})$ vanishes.

Letting $A_{i_1 \dots i_p}$ be the term corresponding term in the
definition of $K$, note that if $i \not \in \{ i_1, \dots, i_p \}$
then $\rho_i^{c_i} A_{i_1 \dots i_p} = \rho_i^{c_i} A_{i_1 \dots i
  \dots i_p} \rvert_{\rho_i = 0}$.
Fixing $i$, multiplying by $\rho^{c_i} K$ and restricting to $\rho_i =
0$ gives
\begin{align*}
  \rho_i^{c_i} K \rvert_{\rho_i = 0} &= \kappa_i + \sum_p^{m - 1}
  (-1)^{p - 1} \rho_i^{c_i} ( \sum_{\substack{
   1 \le i_1 < \dots < i_p \le m \\
    i     \not \in \{i_1 \dots i_p\}
  }} A_{i_1 \dots i_p} - \sum_{\substack{
   1 \le i_1 < \dots < i_{p + 1} \le m \\
    i \in \{i_1 \dots i_{p + 1}\}
  }} A_{i_1
    \dots i_{p + 1}}) \rvert_{\rho_i = 0} \\
  &= \kappa_i,
\end{align*}
which establishes \eqref{eq:desired_restriction}.

We now prove the final statement of the lemma.  Let $\chi$ be the
cutoff function defined above.  First, we claim that
under the stated assumptions there exists a distribution $K'$
supported in $\{ \rho_1 \leq 1 \}$ satisfying both
\eqref{eq:desired-expansion} (with $K$ replaced by $K'$) and that
\begin{equation}\label{eq:desired-expansion-restriction}
\rho_i^{c_i} K'  \rvert_{H_i}= \chi(\rho_1) \kappa_i
\end{equation}
for each $i \neq 1$.  To see
this, take any distribution $K''$ supported in $\{ \rho_1 \leq 1 \}$ satisfying
\eqref{eq:desired-expansion}, and note that $
a_i := \rho_i^{c_i}  K''\rvert_{H_i} - \chi(\rho_1)\kappa_i = O(\rho^\infty_1)$.  By the support
condition, the distribution $K' = K'' - \sum_{i \neq 1} \chi(\rho_i)
a_i$ is defined globally, has the same asymptotic expansion at $H_1$
as $K''$, and satisfies \eqref{eq:desired-expansion-restriction}.
This $K'$ will play the role of $\chi(\rho_1)\rho^{-c_1}\kappa_1$ from
the previous paragraph.  Concretely, for $1 < i_1 < \dots < i_p \le
m$, let $a_{i_1 \dots i_p} = (\Pi_{j \in \{ i_1, \dots, i_p}
\rho_j^{c_j} K') \rvert_{H_{i_1} \cap \dots \cap H_{i_p}}$.  Then we
may take
\begin{align*}
  K &= \sum_{p = 1}^m (-1)^{p - 1} \sum_{1 < i_1 < \dots < i_p \le m}
  \kappa_{i_1 \dots i_p} (\prod_{j \in \{ i_1, \dots , i_p \}}
  \chi(\rho_{j}) \rho_j^{-c_j}) \\
&\quad + K' + \sum_{p = 1}^m (-1)^{p - 1} \sum_{1 < i_1 < \dots < i_p \le m}
  a_{i_1 \dots i_p} (\prod_{j \in \{ i_1, \dots , i_p \}}
  \chi(\rho_{j}) \rho_j^{-c_j}).
\end{align*}
Again, for example if $m = 2$ then
$$
K = K'  + \chi(\rho_2) \rho_2^{-
  c_2} \kappa_2 - (\rho_2^{c_1}K')\rvert_{H_2} \rho_2^{- c_2} \chi(\rho_2).
$$
The given expression for $K$ can be directly checked to satisfy \eqref{eq:desired_restriction}--~\eqref{eq:desired-expansion}.
\end{proof}

\subsection{Melrose's pushforward theorem}\label{sec:pushforward}

Given a map $\beta \colon X \lra Y$ between manifolds with corners, if
$\mathcal{M}(\bullet)$ with $\bullet = X, Y$ denotes the space of
boundary hypersurfaces, then $\beta$
is a \textbf{b-map} if it is smooth and if for each $H \in \mathcal{M}(Y)$ with
$\rho_H$ a
boundary defining function for $H$ then 
$$
\beta^* \rho_H
= a \Pi_{H'_j \in \mathcal{M}(X)} \rho_{H'_1}^{e(H'_1, H)}
  \rho_{H'_1}^{e(H'_2, H)} \dots \rho_{H'_1}^{e(H'_N, H)}
$$ 
where $a
      \in C^\infty(X)$ is non-vanishing and $N$ is the number of
      boundary hypersurfaces
      of $Y$ and the $e(H', H)$ are non-negative integers.    This
      means foremost that $\rho_H$ pulls back to a smooth function,
      and the numbers $e(H', H)$ simply keep track of the order of
      vanishing of $\beta^* \rho_H$ at each face of $X$.
      The function
\begin{equation}\label{eq:exponent-matrix}
e \colon \mathcal{M}(X) \times \mathcal{M}(Y) \lra \mathbb{N}_0
\end{equation}
is
      the \textbf{exponent matrix} of $\beta$, and $e(H', H) > 0$ means $H'$ maps into
      $H$ via $\beta$.

If a b-map has a few additional properties then it pushes forward
polyhomogeneous distributions (more accurately, densities) to polyhomogeneous distributions and
their index sets change in a way dictated by the exponent matrix.
Note that it follows from the definition of a b-map that every boundary face
$F$ of $X$ (meaning an intersection of boundary hypersurfaces),
can be associated to a face  $\overline{\beta}(F)$ of $Y$ defined to be
  the unique face with $\beta(x) \in
  \overline{\beta}(F)^\circ$ for every $x \in F^\circ$.
A b-map $\beta \colon X \lra Y$ is a \textbf{b-fibration} if:
\begin{itemize}
\item $\beta$ does not increase the codimension of faces, i.e.\ for
  each boundary face $F$ of $X$, the associated face
  $\overline{\beta}(F)$ in $Y$ satisfies that $\codim(F) \le
  \codim(\overline{\beta}(F))$.
\item  Restricted to the interior of any face $F^\circ$, $\beta \colon
  F^\circ \lra (\overline{\beta}(F))^\circ$ is a fibration of open
  manifolds in the standard sense. 
\end{itemize}

According to a theorem of Melrose \cite{Melrose1992} which we state
below, a b-fibration pushes forward phg densities to phg densities in
a manner we describe now.
First, on a manifold with corners we choose a non-vanishing b-density $\mu$, meaning
a section of $|\Lambda|^n(\Tbc X)$, the density bundle of the
b-cotangent bundle.  The b-tangent bundle $\Tb X$ is the bundle
whose smooth sections are $\mathcal{V}_b$, the vector fields tangent
to the boundary.  The bundle $\Tbc X$ is the dual bundle of $\Tb X$,
and near a face $F = \cap_{i = 1}^{\delta} H_i$ where $\rho_i$ are
bdf's and $y$ are coordinates on $F$ then, the sections of $\Tbc X$ take the form
$$
\sum_i \xi_i \frac{d \rho_i}{\rho_i}  + \eta\ dy.
$$
It follows that near any intersection $F = \cap_{j \in J} H_{j_q}$ of
boundary hypersurfaces for $J \subset \mathcal{I}$ where $\mathcal{I}$ indexes $\mathcal{M}(X)$  (i.e.\ any face of $X$) that 
a non-vanishing b-density takes the form
\begin{equation}\label{eq:b-density}
\mu =\left|a \frac{ dy \prod_{j \in J} d\rho_{j}}{\prod_{j \in J} \rho_{j}}\right|
\end{equation}
for some smooth non-vanishing function $a$ on $X$.  A polyhomogeneous
b-density $u \in \mathcal{A}^{\mathcal{E}}_{\phg}(X) \otimes |\Lambda|^n(\Tbc
X)$ can be written as $f \mu$ for a phg function $f$ and the index
set of $u$ is by definition the index set of $f$.
\begin{theorem}[Melrose \cite{Melrose1992}]\label{thm:pushforward}
  Let $u \in \mathcal{A}^{\mathcal{E}}_{\phg}(X) \otimes |\Lambda|^n(\Tbc X)$
  be a polyhomogeneous b-density on $X$ with index set
  $\mathcal{E}$, let $f \colon X \lra Y$ be a b-fibration with
  exponent matrix $e$, and
  define the pushforward $f_* u$ to be the distribution on smooth
  functions $v \in C_{comp}^\infty(Y)$ acting by $\la f_* u, v \ra_Y = \la u,
  f^* v \ra_X$.  Then provided for each $H \in \mathcal{M}(X)$ we have 
  \begin{equation}\label{eq:integrability}
e(H, H') = 0  \  \forall H' \in
  \mathcal{M}(Y) \quad\implies\quad \mathcal{E}(H) > 0,
\end{equation}
then $f_* u \in \mathcal{A}^{\mathcal{E}'}_{\phg}(Y) \otimes
  |\Lambda|^n(\Tbc Y)$ where
$$
\mathcal{E}'(H) = \overline{\bigcup}_{H'} \{ (\frac{z}{e(H', H)}, p) :
(z, p) \in \mathcal{E}(H) \},
$$
with the (extended) union taken over $H'$ with $e(H', H) > 0$.  
\end{theorem}
The extended union, defined in \cite{tapsit}, contains the standard
union and possibly more log terms.

\subsection{Heat trace asymptotics}\label{sec:heat-trace-asymptotics}

We now use Theorem \ref{thm:pushforward} to prove the heat trace
formula in Corollary \ref{thm:heat-trace} above.  The heat trace is
equal to
\begin{equation}
  \label{eq:heat-trace-formula}
  \Tr (e^{-t \Delta}) = \int_M H_t(w, w) \dVol_g = \push_* ((\iota^* H_t) \dVol),
\end{equation}
where $\iota \colon M \times [0, \infty) \lra M
\times M \times [0, \infty)$ is the diagonal inclusion and $\push \colon M \times [0, \infty) \lra [0, \infty)$ is the
projection onto the right factor.  The natural space here on which to
consider $H_t$ is $\Mheat$, and thus to evalute this pushforward we
must see how $\push$ and $\iota$ act on the natural blown up
spaces. The following may be easily verified.
\begin{enumerate}
\item The closure $(\Mheat)_\Delta := \cl (\iota(M^\circ \times (0, \infty)))$ is a
  manifold with corners with $4$ boundary hypersurfaces, $\sff, \fff^d,
  \ff^d, \tf^d$, equal to the intersection of $\cl (\iota(M^\circ
  \times (0, \infty)))$ with $\rf \cap \lf, \fff, \ff,$ and $\tf$,
  respectively
\item The map $\push$ extends form the interior $M^\circ \times (0,
  \infty)$ to a b-fibration $\push \colon (\Mheat)_\Delta \lra [0,
  \infty)$ with exponent matrix
$$
e_\push(\sff) = 0 , e_\push(\fff^d) = 2,
  e_\push(\ff^d) = 2k , e_\push(\tf^d) = 2.
$$
\end{enumerate}

To apply the pushforward theorem, we note that the volume density 
$$
\mu = |\dVol_g \frac{dt}{t} |
= x^{kf + 1} \frac{|dx dy dz dt|}{xt}
$$
is equal on $(\Mheat)_\Delta$ to
$$
\mu = a \ (\rho_{\sff} \rho_{\fff^d}
  \rho_{\ff^d})^{kf + 1} \mu_0,
$$
where $\mu_0$ is a non-vanishing b-density on $(\Mheat)_\Delta$.  
 Thus
$(\iota^*H)  \mu$ is phg on $(\Mheat)_\Delta$ with index family
$\mathcal{E}^d$ satisfying
\begin{gather*}
  \inf \mathcal{E}^d(\sff) = 3, \quad
 \inf \mathcal{E}^d(\fff^d) = -  \dimY, \\
   \inf \mathcal{E}^d(\ff^d) = k(f - n) + 1, \quad \mathcal{E}^d(\tf^d) = \{-n, -
  n + 1, \dots \} .
  \end{gather*}
Note that $\Tr e^{- t\Delta} \frac{dt}{t}= \push_* ((\iota^* H_t)
\mu)$. The integrability condition \eqref{eq:integrability} must be
checked
only for $\sff$ and thus holds by Theorem \ref{thm:heatkernel}, and we apply the pushforward theorem
to obtain that $\Tr e^{- t\Delta}$ is
polyhomogeneous with index set
\begin{gather*}
  \{ (\zeta_1/2, p_1) \colon (\zeta_1, p_1) \in \mathcal{E}^d(\fff^d) \}
  \ \ov{\cup} \ \{ (\zeta_2/(2k), p_2) \colon (\zeta_2, p_2) \in
  \mathcal{E}^d(\ff^d) \}  \\ \ \ov{\cup} \ \{ (\zeta_3/2, p_3) \colon
  (\zeta_3, p_3) \in \mathcal{E}^d(\tf^d) \}.
\end{gather*}
In particular, 
\begin{equation*}
  \begin{split}
    F(t) = (\sum_{j = 0}^f a_j t^{-n/2 + j/2} )+ c_0 t^{-(\dimY +
      1)/2 + 1/(2k)} + O(t^{-(\dimY + 1)/2 + 1/(2k) + \epsilon}),
  \end{split}
\end{equation*}
for some $\epsilon > 0$.  As discussed in \cite[Section 3.3]{MV2012},
the heat kernel in fact lies in an even calculus and thus the terms
for odd $j$ in this sum are equal to $0$, giving the trace formula
\eqref{eq:heat-trace}.  The fact that the leading order term is the
volume is standard.

\section{Triple Space}\label{sec:triple-space}

We will now analyze composition properties for ``Volterra'' type
convolution operators as described in \eqref{eq:convolution-operator}. To do so,
following \cite{tapsit,Grieser-Hunsicker}, we construct a ``triple space,''
which we denote by $\Mtrip$, which is designed specifically to
accomodate the process of composing operators which have the structure
of the error terms in \eqref{eq:real-error-to-iterate}.  The structure of  
our triple space is analogous to that constructed by Grieser and
Hunsicker in \cite{Grieser-Hunsicker}, with slightly different
homogeneities and with the added complication that there are time
variables involved.

Note that, given $A_i$, $i = 1, 2$, we want is to make sense of the integral
	\begin{equation}\label{eq:voltera-composition}
		\int_M\! \int_0^{t'} A_1(w, w', t') A_2(w', \wt{w}, t - t') \dVol_g(w') dt'.
	\end{equation}
Define the wedge 
\begin{equation}\label{eq:wedge}
W := \set{ t - t' \geq 0} \subset \mathbb{R}^+_{t}
\times \mathbb{R}^+_{t'},
\end{equation}
and define the left, center, and right projections
	\begin{equation}
          \begin{split}
            \pi_L : M \times M \times M \times W &\lra M \times M \times [0, \infty)_{\wtf} \\
            (w, w', \wt{w}, t, t') &\longmapsto (w, w', t') \\
            \pi_C : M \times M \times M \times W &\lra M \times M \times [0, \infty)_{\wtf} \\
            (w, w', \wt{w}, t, t') &\longmapsto (w, \wt{w}, t) \\
            \pi_R : M \times M \times M \times W &\lra M \times M \times [0, \infty)_{\wtf}\\
            (w, w', \wt{w}, t, t') &\longmapsto (w', \wt{w}, t - t').
          \end{split}
	\end{equation}
Then, formally, the integral in \eqref{eq:voltera-composition} says
that the integral kernel of $A_1 A_2$ (as an operator acting by
convolution in time) is
\begin{equation}
  \label{eq:fake-pushforward}
  (A_1 A_2)(w, \wt{w}, t) = (\pi_C)_*(\pi^*_L A_1)(\pi^*_R A_2),
\end{equation}
where $(\pi_C)_*$ denotes the pushforward, i.e.\ the integral along
the fibers of $\pi_C$ (which, by the way we have set up the problem,
requires the choice of a metric on the fibers which we come to
shortly.)  Analysis of \eqref{eq:fake-pushforward} 
becomes tractable if the space $M^3 \times W$ is blown up so that the
pushforward theorem described in Section \ref{sec:pushforward} applies.

Note that $M^3 \times W$ is a manifold with corners with $5$ boundary
hypersurfaces
\begin{gather*}
  \label{eq:triple-initial-bhs}
  \Lfacet = \{x = 0 \}, \quad   \Cfacet = \{x' = 0 \}, \quad
  \Rfacet = \{\wt{x} = 0 \}\\
  \Tfacet_1 = \{ t' = 0 \}, \quad \Tfacet_2 = \{ t - t' = 0 \}
\end{gather*}
It is easy to check that, in the language of Appendix \ref{sec:mwc}, the maps
$\pi_\bullet$ with $\bullet \in \{ L, C, R \}$ are b-maps from $M^3
\times W$ to $M^2 \times [0, \infty)_t$ and the exponent matrices
are also easy to compute, 
\begin{gather}
  \label{eq:initial-exponent-matrices}
e_{\pi_L}( \bullet , \bullet') = \left\{
    \begin{array}{cc}
      1 &    \bullet = \Lfacet, \bullet' = \lf \\
      1 &    \bullet = \Cfacet, \bullet' = \rf \\
      1 &    \bullet = \tb'_1, \bullet' = \tb \\
      0  &  \mbox{ otherwise }
    \end{array}\right., \ 
e_{\pi_C}( \bullet , \bullet') = \left\{
    \begin{array}{cc}
      1 &    \bullet = \Lfacet, \bullet' = \lf \\
      1 &    \bullet = \Rfacet, \bullet' = \rf \\
      0  &  \mbox{ otherwise }
    \end{array}\right., \\
e_{\pi_R}( \bullet , \bullet') = \left\{
    \begin{array}{cc}
      1 &    \bullet = \Cfacet, \bullet' = \lf \\
      1 &    \bullet = \Rfacet, \bullet' = \rf \\
      1 &    \bullet = \tb'_2, \bullet' = \tb \\
      0  &  \mbox{ otherwise }
    \end{array}\right..
\end{gather}
We blow up $M^3 \times W$ to form a space $\wt{\beta} \colon \Mtrip \lra M^3 \times
W$ in a sequence of steps as follows.

First, consider the three pullbacks of the submanifold $\mathcal{B}_0 = \{ x = \wt{x} ,
y = \wt{y} , \wt{t} = 0\}\subset M^2 \times
[0, \infty)_{\wt{t}}$ defined in \eqref{eq:first-blowdown}
\begin{equation}
  \label{eq:triple-first-blowdown}
  \pi_L^{-1}(\mathcal{B}_0), \quad \pi_C^{-1}(\mathcal{B}_0), \quad \pi_R^{-1}(\mathcal{B}_0).
\end{equation}
These three sets intersect pair-wise in the triple intersection:
\begin{equation}
  \label{eq:triple-first-blowdown-intersection}
  \pi^{-1}_L(\mathcal{B}_0) \cap \pi^{-1}_C(\mathcal{B}_0) =
  \pi^{-1}_C(\mathcal{B}_0) \cap \pi^{-1}_R(\mathcal{B}_0)  =
  \pi^{-1}_L(\mathcal{B}_0) \cap  \pi^{-1}_R(\mathcal{B}_0) = \mathcal{S} ,
\end{equation}
where 
\begin{equation}
  \label{eq:51}
  \mathcal{S} = \set{x  = x' =\wt{x} = t = t' = y - y' = y' - \wt{y} =
    0}.
\end{equation}
We blow up the set $\mathcal{S}$, with appropriate homogeneities,
specifically letting
\begin{equation}
  \label{eq:52}
  \Mtripa = [M^3 \times W ; \mathcal{S} ]_{inhom},
\end{equation}
with $t \sim
x^{2} \sim (x')^2 \sim \wt{x}^{2} \sim |y - y' |^{2} \sim  |y' - \wt{y}|^2$,
and let $\wt{\beta}_0 \colon \Mtripa \lra M^3 \times W$ denote the
blowdown map.  Call the introduced boundary hypersurface
$\fff^\cap$. Near  to $\fff^\cap$, we have polar coordinates
\begin{equation}
  \label{eq:triple-polar-first-intersection}
  \begin{split}
    \rho_{\cap} &= \lp t  + x^{2} + (x')^2 + \wt{x}^{2} + |y -
    y'|^{2} + |y' - \wt{y}|^2 \rp^{1/2}, \\
    \phi^\cap &= \lp \frac{t'}{\rho_{\cap}^{2}} , \frac{t - t'}{\rho_{\cap}^{2}},
    \frac{x}{\rho_{\cap}}, \frac{x'}{\rho_{\cap}}
    \frac{\wt{x}}{\rho_{\cap}}, \frac{y - y'}{\rho_{\cap}}, \frac{y' - \wt{y}}{\rho_{\cap}}
  \rp \\
    &=: (\phi^\cap_{t'}, \phi^\cap_{t - t'},  \phi^\cap_{x}, \phi^\cap_{x'},
    \phi^\cap_{\wt{x}}, \phi^\cap_{y - y'}, \phi^\cap_{y' - \wt{y}}),
    \mbox{ along with } y',  z, z',  \wt{z}.
  \end{split}
\end{equation}
The asymmetry of the $y, y', \wt{y}$ in the coordinates is spurious in
the sense that if one defines $\phi^\cap_{y - \wt{y}} = (y -
  \wt{y})/\rho_{\cap}$, then any two of the $\phi^\cap_{y - y'}, \phi^\cap_{y' - \wt{y}}$
can be used in $\phi^\cap$ by redefining $\rho_{\cap}$
using e.g.\ $|y - y'|^2$ and $|y - \wt{y}|^2$  (and then using
$\phi^\cap_{y - y'}, \phi^\cap_{y' - \wt{y}}$).  Either set of coordinates is defined in
a colar neighborhood of $\fff^\cap$.

We then blow up the closures of the lifts 
$$\mathcal{S}_{\bullet} := \cl((\pi_{\bullet} \circ
\wt{\beta}_0)^{-1}(\mathcal{B}_0) \setminus \fff^\cap),
$$
i.e.\ the rest of the lifts
of the $\mathcal{B}_0$ via the three projections, 
where $\bullet \in \set{L, C, R}$.
These are disjoint subsets and we blow them up in any order, setting
\begin{equation}
  \label{eq:triple-first-blowup}
  \Mtripb = [\Mtripa; \cup_{\bullet = L, C, R} \mathcal{S}_\bullet ]_{inhom},
\end{equation}
with the appropriate homogeneties, e.g.\ for $\mathcal{S}_L$ we have
$t' \sim
x^{2} \sim (x')^2 \sim |y - y'|^2$.  Again, we have a
blowdown map
\begin{equation}
  \label{eq:Mheat1-blowdown}
  \wt{\beta}_1 \colon \Mtripb \lra M \times M \times M \times W.
\end{equation}
The new faces we call
$\fffb$ with $\bullet \in \{ L, C, R \}$.  Coordinates at
$\fffL$ can be determined as follows.  Note that $\mathcal{S}_L$ is given in the coordinates
\eqref{eq:triple-polar-first-intersection} by $\phi^\cap_{t'} = \phi^\cap_x =
\phi^\cap_{x'} = \phi^\cap_{y - y'} = 0$, and that in a neighborhood
of $\mathcal{S}_L$ away from $\fff^\cap$, $\phi^\cap_{t'} \sim t'$. 
Thus, to match homogeneities with the blowups
  of the double space, we want to blow this up so that the following
  give polar coordinates near 
the intersection of $\fffL$ and $\fff^\cap$:
\begin{equation}
  \label{eq:triple-polar-first-Left}
  \begin{split}
    \rho^L &= \lp \phi^\cap_{t'} + (\phi^\cap)^2_x +
(\phi^\cap)^2_{x'} + |(\phi^\cap)_{y - y'}|^2 \rp^{1/2}, \\
    \phi^L &= \lp \frac{\phi^\cap_{t'}}{(\rho^L)^{2}} ,
    \frac{\phi^\cap_{x}}{\rho^L}, \frac{\phi^\cap_{x'}}{\rho^L},
    \frac{\phi^\cap_{y - y'}}{\rho^L} \rp \\
    &=: (\phi^L_{t'},  \phi^L_{x}, \phi^L_{x'},
    \phi^L_{y - y'}),
    \mbox{ along with } y',  z, z', \wt{z}, \rho_{\cap}, \phi^\cap_{\wt{x}},
    \phi^\cap_{y' - \wt{y}}, \phi^\cap_{t - t'}
  \end{split}
\end{equation}
with functions as in \eqref{eq:triple-polar-first-intersection}.  It
is also possible to use simpler projective coordinates, as we will see
below.  Coordinates near
$\fffR$ can be derived similarly by switching $\phi^\cap_{t'}$ with
$\phi^\cap_{t - t'}$  and $\phi^\cap_x$ with $\phi^\cap_{\wt{x}}$.
The situation at $\fffC$ is slightly different since, writing
$\phi^\cap_t = \phi^\cap_{t'} + \phi^\cap_{t - t'}$,  the
pullback of $\phi^\cap_t$ on $\Mheata$ via $\pi_C$ vanishes at $\phi^\cap_{t'} =
0 = \phi^\cap_{t - t'}$, and thus $\mathcal{S}_C$ is codimension $1$
higher than $\mathcal{S}_\bullet$ for $\bullet = L, R$.  

Here we
blow up so that the following give coordinates
\begin{equation} 
  \label{eq:triple-polar-first-center}
  \begin{split}
    \rho^C &= \lp \phi^\cap_{t} + (\phi^\cap)^2_x +
(\phi^\cap)^2_{\wt{x}} + |\phi^\cap_{y - \wt{y}}|^2 \rp^{1/2}, \\
    (\phi^\cap)^C &= \lp \frac{\phi^\cap_{t'}}{(\rho^C)^{2}},
    \frac{\phi^\cap_{t - t'}}{(\rho^C)^{2}} ,
    \frac{\phi^\cap_{x}}{\rho^C}, \frac{\phi^\cap_{\wt{x}}}{\rho^C},
    \frac{\phi^\cap_{y - \wt{y}}}{\rho^C} \rp \\
    &=: (\phi^C_{t'}, \phi^C_{t - t'},  \phi^C_{x}, \phi^C_{\wt{x}},
    \phi^C_{y - \wt{y}}),
    \mbox{ along with } y',  z, z', \wt{z}, \rho_{\cap}, \phi^\cap_{x'},
    \phi^\cap_{y' - \wt{y}}.
  \end{split}
\end{equation}
\begin{lemma}
  With terminology as in Appendix \ref{sec:pushforward}, the maps
  $\pi_\bullet$ extend from the interior to b-maps 
  \begin{equation}
    \label{eq:cone-triple-space-projections}
    \wt{\pi}_{\bullet} \colon \Mtripb \lra \Mheata
  \end{equation}
for $\bullet \in \{L, C, R\}$ with exponent matrices $e_{\wt{\pi}_{\bullet}}$ satisfying
\begin{equation}
  \label{eq:first-exponent-matrix}
  \begin{split}
    e_{\wt{\pi}_{\bullet}}(\ff_1^\cap, \fff) &= 1,
    e_{\wt{\pi}_{\bullet}}(\ff_1^{\bullet'}, \fff) = \delta_{\bullet,
      \bullet'}, e_{\wt{\pi}_C}(\fffL, \lf) = 1,
    e_{\pi_C}(\fffR, \rf) = 1, \\
e_{\wt{\pi}_R}(\fffL, \lf) &= 1, e_{\wt{\pi}_L}(\fffR, \rf) = 1,
e_{\wt{\pi}_R}(\fffC, \tb) > 0 , e_{\wt{\pi}_L}(\fffC, \tb) > 0, 
  \end{split}
\end{equation}
where $\delta_{\bullet, \bullet'} = 1$ if $\bullet = \bullet'$ and
zero otherwise.  When $\bullet \in \{\Lfacet, \Cfacet, \Rfacet, \tb'_1,
\tb'_2 \} $, i.e.\ when it is the pullback of a boundary hypersurface
of $M \times M \times M \times W$ via the blowdown map, then the
exponent matrix satisfies \eqref{eq:initial-exponent-matrices} with
$\wt{\pi}$ replacing $\pi$.

Moreover, $\wt{\pi}_C$ is a b-fibration in the sense of Appendix \ref{sec:pushforward}.
\end{lemma}
\begin{remark}
  The significance of the inequalites in
  \eqref{eq:first-exponent-matrix} involving $\tb$ is that all the
 distributions under consideration vanish to infinite order at $\tb$, and thus the
  pullbacks of these distributions via $\pi_R$ will vanish to infinite
  order at $\fffC$, and the same for $\pi_L$.
\end{remark}
\begin{proof}
  We verify the lemma for for $\wt{\pi}_C$ and leave the other nearly
  identical calculations to the reader.  That $\wt{\pi}_C$ extends to a
  b-map follows easily by writing the pulling back the coordinates in
 \eqref{eq:polarfirstmodel} and writing them in terms of those in
  \eqref{eq:triple-polar-first-intersection}.  In particular, note that the pullback
  \begin{equation}
    \label{eq:38}
    \wt{\pi}_C^* \rho = \rho_{\cap} \rho^C,
  \end{equation}
so the exponent matrix claim holds.  The rest of the definitions of
b-fibration are easy to check.
\end{proof}
\begin{remark}
The extended map $\wt{\pi}_{L}$ is \emph{not} a b-fibration as it maps
the interior of $\ff_1^C$ to the interior of the face $\tb \cap
\lf$ due to the fact that $t = 0$ on $W$ implies that $t' = 0$
also, thus the map increases the codimension of a face.  The same
holds for $\wt{\pi}_R$, i.e.\ $\wt{\pi}_R(\fffC) \subset \tb
\cap \rf$.
\end{remark}

Next we must blow up the lifts of $\mathcal{B}_1$ in
\eqref{eq:second-blowdown}.  Since by
\eqref{eq:first-exponent-matrix}, $\wt{\pi}_\bullet$
only maps $\ff_1^{\bullet'}$ to $\ff_1$ if $\bullet = \bullet'$, any
of the pair-wise intersections is again equal to the triple intersection
\begin{gather*}
  \mathcal{S}' =  \wt{\pi}_L^{-1}(\mathcal{B}_1) \cap \wt{\pi}_C 
    ^{-1}(\mathcal{B}_1) = \wt{\pi}_C 
    ^{-1}(\mathcal{B}_1) \cap \wt{\pi}_R 
    ^{-1}(\mathcal{B}_1) 
=  \wt{\pi}_L 
    ^{-1}(\mathcal{B}_1) \cap \wt{\pi}_R 
    ^{-1}(\mathcal{B}_1) .
\end{gather*}
Indeed, each is a subset of $\ff_1^\cap$, and in the polar coordinates  defined on the
interior of $\ff_1^\cap$, using the definition of $\mathcal{B}_1$ in
\eqref{eq:second-blowdown}  
\begin{equation}
  \label{eq:triple-intersection-second-round}
  \mathcal{S}'  = \set{ \rho = \phi^\cap_{t'} = \phi^\cap_{t - t'} = 0, \phi^\cap_x = \phi^\cap_{x'} =
    \phi^\cap_{\wt{x}},   \phi^\cap_{y - y'} =  \phi^\cap_{y - \wt{y}} = 0},
\end{equation}
with no restrictions on $y', z, z', \wt{z} $.
We form a space $[\Mtripb ; \mathcal{S}']_{inhom}$ with appropriate homogeneities.
To understand this space, note first that near $\mathcal{S}'$ we can use
projective coordinates on $\fff^\cap$, concretely we can take for
example $\wt{x}$ to be a boundary defining function of $\fff\cap$ and
coordinates $\wt{x}, t'/\wt{x}^2, (t - t')/\wt{x}^2, x/\wt{x},
x'/\wt{x}, (y - y')/\wt{x}, (y' - \wt{y})/\wt{x}$ to replace the polar
coordinates in \eqref{eq:triple-polar-first-intersection}.  
Then the homogeneities are determined by those in the $\ff$ blowdown
of the double space, and one has coordinates 
\begin{equation}
  \label{eq:polarsecondnmodel-triple}
  \begin{split}
    \overline{\rho}_{\cap} &= \lp  \wt{x}^{2(k - 1)} + \frac{t}{\wt{x}^2} +
    (\frac{x - \wt{x}}{\wt{x}})^2 + (\frac{x' - \wt{x}}{\wt{x}})^2
    + (\frac{|y - \wt{y}|}{\wt{x}})^{2} +  (\frac{|y' -
      \wt{y}|}{\wt{x}})^{2} \rp^{1/2(k-1)}, \\
    \overline{\phi} &:= (\overline{\phi}_{\wt{x}}, \overline{\phi}_{t'},
    \overline{\phi}_{t - t'}, 
    \overline{\phi}_{x - \wt{x}}, \overline{\phi}_{x' - \wt{x}}, 
    \overline{\phi}_{y - \wt{y}},
    \overline{\phi}_{y' - \wt{y}}) \\
    & = \lp \frac{\wt{x}}{\overline{\rho}_\cap} ,
    \frac{t'}{\wt{x}^2\overline{\rho}_\cap^{2(k-1)}}, \frac{t -
      t'}{\wt{x}^2\overline{\rho}_\cap^{2(k-1)}}
    \frac{x - \wt{x}}{\wt{x}\overline{\rho}_\cap^{(k-1)}}, \frac{x' -
      \wt{x}}{\wt{x}\overline{\rho}_\cap^{(k-1)}}
     \frac{y - \wt{y}}{\wt{x}\overline{\rho}_\cap^{(k-1)}}, \frac{y' - \wt{y}}{\wt{x}\overline{\rho}_\cap^{(k-1)}}
    \rp  \mbox{ along with }
    \wt{y}, z, z', \wt{z}.
  \end{split}
\end{equation}
One can also take coordinates in which $x, x', \wt{x}$ are permuted,
and the same with $y, y', \wt{y}$.
$$
\mbox{ We let $\ff^\cap$ denote the introduced boundary hypersurface.}
$$

The lifts of the $\wt{\pi}_\bullet^{-1}(\mathcal{B}_1)$ minus their intersections now
have disjoint closures.  For example, we have
$$
\wt{\pi}_{\Lfacet}^{-1}(\mathcal{B}_1) \cap \fff^\cap  \setminus \ff^\cap =
\{ \overline{\phi}_{t'} = \overline{\phi}_{x - x'} =
\overline{\phi}_{y - y'} = 0 \}
$$
while 
$$
\wt{\pi}_{\Cfacet}^{-1}(\mathcal{B}_1) \cap \fff^\cap \setminus \ff^\cap =
\{
  \overline{\phi}_{t} = \overline{\phi}_{x -\wt{x}} = 
  \overline{\phi}_{y - \wt{y}} = 0 \},
$$
where $  \overline{\phi}_{t} = \overline{\phi}_{t'} +
\overline{\phi}_{t - t'}$ and $\overline{\phi}_{x -\wt{x}} =
\overline{\phi}_{x -x'} + \overline{\phi}_{x' -\wt{x}}$ and for
$\wt{\pi}_{\Rfacet}$ we have $\overline{\psi}_{x'} = \psi_{t  - t'}
= 0, \psi_{y'} = 0$;  since $|\overline{\phi}| = 1$, these sets
are disjoint. 
Furthermore, the pullbacks satisfy that 
$$
\wt{\pi}_{\bullet}^{-1}(\mathcal{B}_1) \cap \ff_1^{\bullet'} =
\delta_{\bullet, \bullet'},
$$
for $\bullet, \bullet' \in \{ R, C, L \}$, and each intersection is
straightforward to write down, e.g.\ with coordinates as in \eqref{eq:triple-polar-first-center},
$$
\wt{\pi}_{\Cfacet}^{-1}(\mathcal{B}_1) \cap \fffC =
\{ \rho^C =  \phi^C_{t'} = \phi^C_{t - t'} =
\phi^C_{x} - \phi^C_{\wt{x}} = 0, 
    \phi^C_{y - \wt{y}} = 0\} .
$$
We will blow up first the  $\wt{\pi}_\bullet^{-1}(\mathcal{B}_1) \cap  \ff_1^\cap$ and then the $\wt{\pi}_\bullet^{-1}(\mathcal{B}_1) \cap \ff_1^\bullet$  with for $\bullet \in
\{L, C, R\}$.

In the interior of $\ff_1^\bullet$ with $\bullet \in \{ L, R\}$ the
blow ups of the pullbacks of $\mathcal{B}_1$ are particularly easy
to understand as there we can just pullback the projective coordinates
in \eqref{eq:projectivesecondmodle} and use these together with the
other uneffected coordinates to obtain projective coordinates e.g.\
near $\pi_{\Lfacet} \circ \wt{\beta}_0^{-1}(\mathcal{B}_1) \cap \fffL$
valid near the interior of the introduced boundary hypersurface.  
            \begin{equation}
              \label{eq:projectivesecondmodle-interior-side-triple-right}
              x', \quad \sigma = \frac{s - 1}{(x')^{k-1}} = \frac{x -
                x'}{(x')^{k}}, \quad \eta' = \frac{y -
                y'}{(x')^{k}}, \quad T' =  \frac{t'}{(x')^{2k}},
            \end{equation}
together with $\wt{w}, t$ on the introduced boundary hypersurface.  In
the interior of $\ff_1^C$, one needs only to remember that the vanishing
of the pullback of the $\phi_t$ coordinate implies the vanishing of
both $\phi_{t'}$ and $\phi_{t - t'}$.  One can use $\wt{x}$ as a
boundary defining function and then two projective time coordinates
$T' = t'/\wt{x}^{2k}$ and $\wt{T} = (t - t')/\wt{x}^{2k}$.  
In the
interior of $\ff_1^{\cap}$ but away from $\ff^\cap$, we want the same
homogeneities, but now the pullback of $\wt{x'}$ in the interior of
$\ff^\cap$ is proportional to $\rho^\cap$ and in the interior
of the introduced blow up we will have coordinates as in
\eqref{eq:projectivesecondmodle-interior-side-triple-right} with all
the functions replaced by their $\phi$ couterparts, e.g.\ $x'$ replaced
by $\phi_{x'}$ and $\frac{y -  \wt{y}}{(x')^{k}}$ replaced by
$\psi_{y - y' }/\phi_{x'}$.

We focus at the intersection $\fff^\cap \cap \fff^\bullet$, first with
$\bullet = C$.  Near $\mathcal{S}_C$, we can simplify things slightly
by using projective coordinates, derived from
\eqref{eq:triple-polar-first-center} by noting that
$\phi^\cap_{\wt{x}}$ is non-zero at $\fff^\cap \cap \fff^\bullet \cap
\ \mathcal{S}_C$ and can thus be used as a boundary defining function.
Specifically, take
$$
\wt{\XX} = \phi^\cap_{\wt{x}},\  \XX =
\frac{x}{\wt{x}},\  \TT =
\frac{t}{\wt{x}^2},\ \YY = \frac{y - \wt{y}}{\wt{x}},
$$
together with the other (non-polar) coordinates in
\eqref{eq:triple-polar-first-center}.  Blowing up to introduce a face
$\ff^{\cap, C}$, have
\begin{equation*}
  \begin{split}
    \PP &= (\TT + \rho_\cap^{2(k - 1)} + (\XX - 1)^2 + |\YY|^2
    )^{1/(2(k -1))},  \\
    \PS &=  (\TT/\PP^{2(k-1)} , \rho_\cap/\PP , (\XX - 1)/\PP^{k-1} , \YY/\PP^{k-1}
    ),
  \end{split}
\end{equation*}
but it follows that $\mathcal{S}_1 \cap \fff^C$ intersects $\ff^{\cap,
  C}$ at $\PS = (0, 1, 0, 0)$ and thus $\rho^\cap$ can be used as a
boundary defining function. Again working near $\mathcal{S}_C$ we can take
$\rho_\cap$ as a boundary defining function for $\ff^{\cap, C}$ and
use projective coordinates $ \rho_\cap , \TT/\rho_\cap^{2(k-1)}, (\XX
- 1)/\rho_{\cap}^{k-1} , \YY/\rho_{\cap}^{k-1} $.  Using these we blow
up $\mathcal{S}_1 \cap \fff^C$ with
\begin{equation*}
  \begin{split}
    \NPP &= (\frac{\TT}{\rho_{\cap}^{2(k-1)}} +
    (\phi^{\cap}_{\wt{x}})^{2(k - 1)} + \frac{(\XX -
      1)^2}{\rho_{\cap}^{k-1}}+ \frac{|\YY|^2}{\rho_{\cap}^{k-1}}
    )^{1/(2(k -1))}, \\
    \NPS &=   (\frac{\TT}{(\NPP\rho_\cap)^{2(k-1)}} , \frac{\phi^{\cap}_{\wt{x}}}{\NPP} , \frac{\XX - 1}{(\NPP\rho_\cap)^{k-1}} , \frac{\YY}{(\NPP\rho_\cap)^{k-1}}
    ) \\
&=   (\frac{t}{\wt{x}^2(\NPP\rho_\cap)^{2(k-1)}} ,
\frac{\wt{x}}{\NPP\rho_\cap} , \frac{x- \wt{x}}{\wt{x}(\NPP\rho_\cap)^{k-1}} ,
\frac{y - \wt{y}}{\wt{x}(\NPP\rho_\cap)^{k-1}}
    ), 
  \end{split}
\end{equation*}
and this is the final blow up of $\mathcal{S}_C$.  The blowups for
$\mathcal{S}_{L}, \mathcal{S}_R$ are similar and left to the reader.

\begin{proposition}[Incomplete cusp edge heat triple space]
The above construction yields a space and blowdown map
\begin{equation}\label{eq:real-triple-blowdown}
\wt{\beta} \colon \Mtrip \lra M \times M \times M \times W,
\end{equation}
such that the maps $\wt{\pi}_\bullet$ from \eqref{eq:cone-triple-space-projections} extend to b-maps
$\dpi_\bullet \colon \Mtrip \lra \Mheatb$ with exponent matrix satisfying \eqref{eq:initial-exponent-matrices},
\eqref{eq:first-exponent-matrix} (with $\pi$ and $\wt{\pi}$ replaced
by $\dpi$), and
\begin{equation}  \label{eq:second-exponent-matrix}
  \begin{split}
    e_{\dpi_{\bullet}}(\ff^\cap, \ff) = 1, \quad  
    e_{\dpi_{\bullet}}(\ff^{\cap, \bullet'}, \ff)  =
    e_{\dpi_{\bullet}}(\ff^{\bullet'}, \ff) &= \delta_{\bullet,
      \bullet'} \\
    e_{\dpi_{L}}(\ff^{\cap, R}, \fff) = e_{\dpi_{R}}(\ff^{\cap, L},
    \fff) = e_{\dpi_{C}}(\ff^{\cap, R}, \fff) =
    e_{\dpi_{C}}(\ff^{\cap, L}, \fff) &= 1 \\
    e_{\dpi_{L}}(\ff^{R}, \rf) = e_{\dpi_{R}}(\ff^{L},
    \lf) = e_{\dpi_{C}}(\ff^{R}, \rf) =
    e_{\dpi_{C}}(\ff^{L}, \lf) &= 1 \\
    e_{\dpi_{L}}(\ff^{\cap, C}, \tb),    e_{\dpi_{L}}(\ff^{C}, \tb),
    e_{\dpi_{R}}(\ff^{\cap, C}, \tb),    e_{\dpi_{R}}(\ff^{C}, \tb)
    &\ge 1.
  \end{split}
\end{equation}
Moreover, apart from components of $e_{\dpi_{\bullet}}(\ff^{\cap, C},
\bullet')$ and $e_{\dpi_{\bullet}}(\ff^{C}, \bullet')$ with $\bullet
\in \{L, R \}$, all other components are zero.

  Again, $\dpi_C$ is a b-fibration.
\end{proposition}
\begin{proof}
Again, we focus on $\dpi_C$.  To check that $\wt{\pi}_C$ extends to a
$b-map$, we pull back the polar coordinates $\overline{\rho},
\overline{\phi}, \wt{y}, z, \wt{z},$ from
\eqref{eq:polarsecondnmodel} defined at $\ff$ in $\Mheatb$.  First, we
compute
\begin{equation*}
  \begin{split}
    \wt{\pi}_C^* \ov{\rho} &=  \wt{\pi}_C^*(\lp  (t/\wt{x}^2) + \wt{x}^{2(k - 1)}  + (s - 1)^2
    + (|y - \wt{y}|/\wt{x})^{2} \rp^{1/2(k-1)}) \\
&= \lp  \TT + \wt{x}^{2(k - 1)}  + (\XX - 1)^2
    + |\YY|^{2} \rp^{1/2(k-1)} \\
&= \NPP\rho_\cap,
  \end{split}
\end{equation*}
and then note that
\begin{equation*}
\begin{split}
  \wt{\pi}_C^* \ov{\phi} &= \wt{\pi}_C^*\lp
  \frac{t}{\wt{x}^2\overline{\rho}^{2(k-1)}},
  \frac{\wt{x}}{\overline{\rho}} , \frac{x -
    \wt{x}}{\wt{x}\overline{\rho}^{(k-1)}}, \frac{y -
    \wt{y}}{\wt{x}\overline{\rho}^{(k-1)}} \rp \\
&= \lp
  \frac{t}{\wt{x}^2(\NPP \rho_{\cap})^{2(k-1)}},
  \frac{\wt{x}}{\NPP \rho_{\cap}} , \frac{x -
    \wt{x}}{\wt{x}(\NPP \rho_{\cap})^{(k-1)}}, \frac{y -
    \wt{y}}{\wt{x}(\NPP \rho_{\cap})^{(k-1)}} \rp \\
&= \NPS.
\end{split}
\end{equation*}
This establishes both claims for $\dpi_C$.  The $R, L$ case are left
to the reader.
\end{proof}

\begin{proposition}\label{thm:composition}
  For $i = 1, 2$, let $A_i \in \mathcal{A}_{\phg}^{\mathcal{E}_i}(\Mheatb)$
  with the index sets $\mathcal{E}_i$ satisfying $\mathcal{E}_i(\fff)
  \ge - 3 - \dimY - kf$, $\mathcal{E}_i(\ff) \ge - kn - 2k$,
  $\mathcal{E}_i(\lf) = \mathcal{E}_i(\tb) = \varnothing$ and
  $\mathcal{E}_i(\rf)$ satisfying \eqref{eq:leftfacebehaviour}.
  Then $A_3 := \iint_0^t A_1(w, w', t') A_2(w', \wt{w}, t - t')
  \dVol_{w'} dt'$ lies in $\mathcal{A}_{\phg}^{\mathcal{E}_3}(\Mheatb)$ where
  for any $\epsilon > 0$,
  \begin{equation}
    \label{eq:composition-index-sets}
    \begin{split}
      \inf \mathcal{E}_3(\fff) &\ge \inf \mathcal{E}_1(\fff) +
      \inf \mathcal{E}_2(\fff) + 3 + \dimY + kf - \epsilon, \\
      \inf \mathcal{E}_3(\ff) &\ge \inf \mathcal{E}_1(\ff) +
      \inf \mathcal{E}_2(\ff) + kn + 2k - \epsilon.
    \end{split}
  \end{equation}
\end{proposition}
\begin{remark}
  The constants $kn + 2k$ and $3 + \dimY + kf$ in
  \eqref{eq:composition-index-sets} should be interpreted, for
  instance in the case of $\ff$, as saying that the (Volterra type)
  composition of two operators given by Schwartz kernels as in the
  theorem has Schwartz kernel whose leading order asymptotic behavior
  at $\ff$ increases \emph{relative to the rate} $-kn - 2k$, in particular
  if both the composed operators grow like $-kn - 2k$ then so does the
  composition.  These are, incidentally, the exact rates of blowup of
  the heat kernel times $t^{-1}$ at the faces $\ff$ and $\fff$, and
  furthermore the fact that the errors $t^{-1}Q$ in
  \eqref{eq:real-error-to-iterate} vanish one order faster than
  $t^{-1} H$ means, as described above, that taking powers makes them
  vanish at increasing rates at both $\ff$ and $\fff$.
\end{remark}

\begin{proof}
  We write $A_3$ as the pushforward of a b-density and then apply the
 Pushforward Theorem from Section \ref{sec:pushforward}.  First we
 define a non-vanishing b-density $\mu_0$ on $M \times M \times M
 \times W$ as follows.  We let $\nu$ be a
 non-vanishing b-density on $M$ satisfying $\nu = a|\frac{dx dy
   dz}{x}|$ for a smooth nonvanishing function $a$ near the boundary, and consider
$$
\mu_0 = \nu \  \nu' \ \wt{\nu} \  | \frac{dt' dt}{t' (t - t')} |
$$
where $\nu', \wt{\nu}$ are equal to $\nu$ in
the primed and tilded coordinates, respectively.  Since the blowdown
map $\wt{\beta}$ from \eqref{eq:real-triple-blowdown} is a
b-map, $\wt{\beta}^* \mu_0$ is a b-density on $\Mtrip$, and one checks
that 
\begin{equation}
\wt{\beta}^* \mu_0 = G \ov{\mu}_0,\label{eq:G}
\end{equation}
for a non-vanishing b-density $\ov{\mu}_0$ on $\Mtrip$ and $G \in
C^\infty(\Mtrip)$ satisfying that for some non-vanishing smooth
function $G'$,
\begin{equation*}
  \begin{split}
    G &= G' (\rho_{\fff^L}\rho_{\fff^C}\rho_{\fff^R})^{\dimY}
    \rho_{\fff^\cap}^{2 \dimY}
    (\rho_{\ff^L}\rho_{\ff^C}\rho_{\ff^R})^{k\dimY + k - 1} \\
    &\qquad \times (\rho_{\ff^{\cap,L}}\rho_{\ff^{\cap, C}}\rho_{\ff^{\cap, R}})^{(k
      +1) \dimY + k - 1} \rho_{\ff^{\cap}}^{2k \dimY + 2(k-1)}.
  \end{split}
\end{equation*}
  Then we can write the desired pushforward as a pushforward of a
  b-density, specifically
\begin{equation}\label{eq:thing-you-want-to-compute}
  \begin{split}
    A_3 \ (\nu  \ \wt{\nu} \ |\frac{dt}{t}|) &= (\pi_C)_* \lp \pi_L^* A_1
    \pi_R^*A_2 \cdot ((t'/t) (t - t')) F(w') \mu_0 \rp \\
    &= (\dpi_C)_* \lp \dpi_L^* A_1
    \dpi_R^*A_2 \cdot \wt{\beta}^* ( (t'/t) (t - t') F(w') \mu_0 ) \rp
  \end{split}
\end{equation}
where $F$ is the
function defined by $\dVol_g= F \nu$ and in particular
$$
F = a \ x^{kf + 1} 
$$
where $a$ is a non-vanishing polyhomogeneous function on $M$, and
$\nu, \wt{\nu}$ are the pullbacks of the density $\nu$ above to the
left and right spacial factors of $M \times M \times \mathbb{R}^+$.
To find the asymptotics of $A_3$ itself we must compute the
asymptotics of the densities on the left hand side of
\eqref{eq:thing-you-want-to-compute};  Letting $\beta_2$ again denote the blowdown map
$\Mheatb \lra M \times M \times [0, \infty)$ in
\eqref{eq:intermediate-blowdown-double}, we check that
$$
\beta_2^*((\pi_L')^*\nu \ (\pi_R')^* \nu \frac{dt}{t})=
\rho_{\fff}^{\dimY} \rho_{\ff}^{\dimY k + k - 1 } \mu_2,
$$
where $\mu_2$ is a non-vanishing b-density on $\Mheatb$.  Thus from
\eqref{eq:thing-you-want-to-compute}, if the distribution $(\dpi_C)_* ( \dpi_L^* A_1
    \dpi_R^*A_2 \cdot \wt{\beta}^* ( (t'/t) (t - t') F(w') \mu_0 ) )$ is
    polyhomogeneous with index set $\mathcal{E}_3'$ then $A_3$ is phg
    with index set $\mathcal{E}_3$ satisfying
    \begin{equation}
      \label{eq:real-index-set}
      \mathcal{E}_3(\fff) = \mathcal{E}_3'(\fff) - \dimY, \quad
      \mathcal{E}_3(\ff) = \mathcal{E}_3'(\ff) - (k \dimY + k - 1),
    \end{equation}
and $\mathcal{E}_3(\bullet) = \mathcal{E}_3'(\bullet)$ otherwise.

Thus the index family of $\dpi_L^* A_1
    \dpi_R^*A_2 \cdot \wt{\beta}^* ( (t'/t) (t - t') F(w') )$ must be determined.
To determine $\dpi_L^* A_1$, we see that, at a bhs $H$ of $\Mtrip$,
the index set of $\dpi_L^* A_1$ is simply the index set of $A_1$ at
the bhs $H'$ of $\Mheatb$ at which $H$ is incident.  Thus from our
work above we see that $\dpi_L^* A_1$ has index set $\ov{\mathcal{E}}_1$ satisfying 
\begin{equation*}
  \begin{split}
    \ov{\mathcal{E}}_1(\Lfacet) &= {\mathcal{E}}_1(\lf) = \varnothing
    \\
    \ov{\mathcal{E}}_1(\Cfacet) =     \ov{\mathcal{E}}_1(\fff^R) =
    \ov{\mathcal{E}}_1(\ff^R)
    &=     \mathcal{E}_1(\rf) \\
    \ov{\mathcal{E}}_1(\tb'_1) =     \ov{\mathcal{E}}_1(\fff^C) =
    \ov{\mathcal{E}}_1(\ff^{\cap, C}) =
    \ov{\mathcal{E}}_1(\ff^{C}) &=     \mathcal{E}_1(\tb) = \varnothing
    \\
    \ov{\mathcal{E}}_1(\fff^\cap) = \ov{\mathcal{E}}_1(\fff^L) =
    \ov{\mathcal{E}}_1(\ff^{\cap, R})
    &=     \mathcal{E}_1(\fff) \\
    \ov{\mathcal{E}}_1(\ff^\cap) = \ov{\mathcal{E}}_1(\ff^{\cap, L}) =
    \ov{\mathcal{E}}_1(\ff^L) &=     \mathcal{E}_1(\ff) \\
    \ov{\mathcal{E}}_1(\Rfacet) & = \mathbb{Z},
  \end{split}
\end{equation*}
the last line coming from the fact that $\dpi_L^* A_1$ is independent
of $\wt{x}$, in particular is smooth up to $\Rfacet$.
The index set $\ov{\mathcal{E}}_2$ of $\dpi_R^* A_2$ has the same
expression in terms of $\mathcal{E}_2$ but with
all `R's switched with `L's, all $\lf$'s with $\rf$'s, and all $1$'s
with $2$'s (except of course for the $1$ in the subscript of $\fff$).
For example, (c.f.\ the second line above) $\ov{\mathcal{E}}_2(\Cfacet), \ov{\mathcal{E}}_2(\fff^L),
\ov{\mathcal{E}}_2(\ff^L)$ are all equal to $\mathcal{E}_2(\lf)$,
which is assumed to be $\varnothing$.  If we define the operation
$\mathcal{E}_1 \oplus \mathcal{E}_2$ on index sets to denote the index
set whose elements are sums of elements of the two index sets, 
It follows that $\dpi_L^* A_1 \
\dpi_R^* A_2$ is polyhomogeneous with index set $\mathcal{F}$
satisfying
\begin{equation}
  \label{eq:4}
  \begin{split}
   \{      \mathcal{F}(\Cfacet) , \mathcal{F}(\Lfacet)  ,     \mathcal{F}(\fff^L) ,
    \mathcal{F}(\ff^L) ,
     \mathcal{F}(\tb'_1) ,     \mathcal{F}(\tb'_2) ,     \mathcal{F}(\fff^C) ,
    \mathcal{F}(\ff^{\cap, C}) ,
    \mathcal{F}(\ff^{C}) \} & = \varnothing
    \\
    \mathcal{F}(\fff^\cap) = \mathcal{E}_1(\fff) \oplus
    \mathcal{E}_2(\fff), \quad     \mathcal{F}(\ff^\cap) = \mathcal{E}_1(\ff) \oplus
    \mathcal{E}_2(\ff), \quad  \mathcal{F}(\Rfacet)
    =     \mathcal{E}_2(\rf) &\\
    \mathcal{F}(\fff^R) = \mathcal{E}_1(\rf) \oplus
    \mathcal{E}_2(\fff), \quad \mathcal{F}(\ff^R) = \mathcal{E}_1(\rf)
    \oplus \mathcal{E}_2(\ff) & \\
     \mathcal{F}(\ff^{\cap, R}) = \mathcal{E}_1(\fff) \oplus
   \mathcal{E}_2(\ff), \quad    \mathcal{F}(\ff^{\cap, L}) = \mathcal{E}_1(\ff) \oplus \mathcal{E}_2(\fff)&.
  \end{split}
\end{equation}

Now we compute the asymptotics of the term $\wt{\beta}^* ( ((t'/t) (t - t')) F(w') \mu_0 )
= \wt{\beta}^* ( ((t'/t) (t - t')) F(w')) G \ov{\mu}_0 $ with $G$ in
\eqref{eq:G}.  First, write $\wt{\beta}^* ( (t
(t - t')) F(w')) = \ov{\pi}_L^*(\wtf) \ov{\pi}_R^*(\wtf)
\ov{\pi}^*_R(F)$ where $F$ is thought of as a function of the left
factor of $M \times M \times [0, \infty)$.  Recalling $\rho$,
$\ov{\rho}$ from \eqref{eq:polarfirstmodel} and
\eqref{eq:polarsecondnmodel}, respectively, and letting $a$ denote a polyhomogeneous 
function which  is smooth and non-vanishing up to boundary hypersurfaces
$\bullet$ for which $\mathcal{F}(\bullet) \neq \varnothing$ (and whose
value will change from line to line), we compute
\begin{equation*}
  \begin{split}
    \wt{\beta}^* ( \frac{t' (t - t')}{t} F(w')) 
&= a \ \frac{\ov{\pi}_L^*(\rho^2 \ov{\rho}^{2k}) \ov{\pi}_R^*(\rho^2
  \ov{\rho}^{2k})}{\ov{\pi}^*_C(\rho^2 \ov{\rho}^{2k})}
\ov{\pi}^*_R((\rho_{\lf} \rho \ov{\rho})^{kf + 1}) \\
&= a \ \frac{(\rho_{\fff^\cap} \rho_{\fff^L} \rho_{\ff^{\cap, R}})^2
(\rho_{\ff^\cap} \rho_{\fff^{\cap, L}} \rho_{\ff^{L}})^{2k}}{(\rho_{\fff^\cap} \rho_{\fff^C} \rho_{\ff^{\cap, R}}\rho_{\ff^{\cap, L}})^2
(\rho_{\ff^\cap} \rho_{\fff^{\cap, C}} \rho_{\ff^{C}})^{2k}}
\\
&\quad \times (\rho_{\fff^\cap} \rho_{\fff^R} \rho_{\ff^{\cap, L}})^2
(\rho_{\ff^\cap} \rho_{\fff^{\cap, R}} \rho_{\ff^{R}})^{2k}
\\
&\quad \times (\rho_C \rho_{\fff^L} \rho_{\ff^L} \rho_{\fff^{\cap}}
\rho_{\fff^R} \rho_{\ff^{\cap, L}} \rho_{\ff^\cap}
\rho_{\ff^{\cap, R}} \rho_{\ff^{R}})^{kf + 1} \\
&= a \ ( \rho_{\fff^L})^2
(\rho_{\fff^{\cap, L}} \rho_{\ff^{L}})^{2k}
(\rho_{\fff^\cap} \rho_{\fff^R})^2
(\rho_{\ff^\cap} \rho_{\fff^{\cap, R}} \rho_{\ff^{R}})^{2k}
\\
&\quad \times (\rho_C \rho_{\fff^L} \rho_{\ff^L} \rho_{\fff^{\cap}}
\rho_{\fff^R} \rho_{\ff^{\cap, L}} \rho_{\ff^\cap}
\rho_{\ff^{\cap, R}} \rho_{\ff^{R}})^{kf + 1} \\
&= a \   (\rho_{\fff^\cap}  \rho_{\fff^R})^2(\rho_{\ff^\cap} \rho_{\ff^{\cap, L}} \rho_{\ff^{\cap, R}}
    \rho_{\ff^R})^{2k} \\
&\quad \times ( \rho_{\fff^R}  \rho_{\fff^{\cap}}
\rho_{\ff^{\cap}} \rho_{\ff^{\cap, R}}
\rho_{\ff^{\cap, L}} \rho_{\ff^{R}})^{kf + 1}.  
  \end{split}
  \end{equation*}
Putting this all together, we see that $\dpi_L^* A_1
    \dpi_R^*A_2 \cdot \wt{\beta}^* ( ((t'/t) (t - t')) F(w') \mu_0 )$
    is polyhomogeneous with index set $\ov{\mathcal{F}}$ 
    \begin{equation}\label{eq:high-quality-mathematics}
      \begin{split}
    \ov{\mathcal{F}}(\fff^\cap) &= \mathcal{E}_1(\fff) \oplus
    \mathcal{E}_2(\fff) + (3 + kf + 2b), \\     \ov{\mathcal{F}}(\ff^\cap) &= \mathcal{E}_1(\ff) \oplus
    \mathcal{E}_2(\ff) + (1 + 2k + kf + 2kb), \\  \ov{\mathcal{F}}(\Rfacet)
    &=     \mathcal{E}_2(\rf) \\
    \ov{\mathcal{F}}(\fff^R) &= \mathcal{E}_1(\rf) \oplus
    \mathcal{E}_2(\fff) + (3 + kf + b), \\ \ov{\mathcal{F}}(\ff^R) &= \mathcal{E}_1(\rf)
    \oplus \mathcal{E}_2(\ff) + (1 + 2k + kf +kb)  \\
     \ov{\mathcal{F}}(\ff^{\cap, R}) &= \mathcal{E}_1(\fff) \oplus
   \mathcal{E}_2(\ff) + (1 + 2k + kf + (k + 1)b), \\
   \ov{\mathcal{F}}(\ff^{\cap, L}) &= \mathcal{E}_1(\ff) \oplus
   \mathcal{E}_2(\fff) + (1 + 2k + kf + (k + 1)b),
      \end{split}
    \end{equation}
and $\ov{\mathcal{F}}(\bullet) = \varnothing$ for all other values of $\bullet$.

Now we apply Theorem \ref{thm:pushforward} to analyze $ (\ov{\pi}_C)_* \lp
\pi_L^* A_1  \pi_R^*A_2 \cdot ((t'/t) (t - t')) F(w') \mu_0 \rp$ 
from
\eqref{eq:thing-you-want-to-compute}.  To check
that the conditions of the theorem hold, we first recall that
$\ov{\pi}_C$ is a b-fibration.  Also, note that
$e_{\ov{\pi}_C}(\Cfacet, H') =  e_{\ov{\pi}_C}(\tb_1', H') =
e_{\ov{\pi}_C}(\tb_2', H') = 0$ for all $H' \in
\mathcal{M}(\Mheatb)$, and so we must check the inte
grability condition there, but by below \eqref{eq:high-quality-mathematics}
we have $\ov{\mathcal{F}}(\Cfacet) = \ov{\mathcal{F}}(\tb_1') =
\ov{\mathcal{F}}(\tb_2') = \varnothing$, so the
integrability condition holds.  Thus $    A_3 (\pi_L')^*\nu \ (\pi_R')^* \nu$
 is phg on
$\Mheatb$ with index set $\mathcal{E}_3'$ satisfying 
\begin{equation}
  \label{eq:6}
  \begin{split}
    \mathcal{E}_3'(\lf)&= \ov{\mathcal{F}}(\Lfacet) \ \ov{\cup} \ 
    \ov{\mathcal{F}}(\fff^L) \ \ov{\cup} \  \ov{\mathcal{F}}(\ff^L) = \varnothing\\
    \mathcal{E}_3'(\rf)&= \ov{\mathcal{F}}(\Rfacet) \ \ov{\cup} \ 
    \ov{\mathcal{F}}(\fff^R) \ \ov{\cup} \  \ov{\mathcal{F}}(\ff^R)\\
    \mathcal{E}_3'(\fff)&= \ov{\mathcal{F}}(\fff^\cap) \ \ov{\cup} \ 
    \ov{\mathcal{F}}(\fff^\cap) \ \ov{\cup} \  \ov{\mathcal{F}}(\ff^{\cap,
      L}) \ \ov{\cup} \  \ov{\mathcal{F}}(\ff^{\cap, R})\\
    \mathcal{E}_3'(\ff)&= \ov{\mathcal{F}}(\ff^\cap) \ \ov{\cup} \ 
    \ov{\mathcal{F}}(\ff^{\cap, C}) \ \ov{\cup} \
    \ov{\mathcal{F}}(\ff^{C}) = \ov{\mathcal{F}}(\ff^\cap)\\
    \mathcal{E}_3'(\tb)&= \varnothing,
  \end{split}
\end{equation}
where we used from below \eqref{eq:high-quality-mathematics} that
various bhs's have infinite order vanishing.  From this we see that
the bounds in Proposition \ref{thm:composition} hold, in particular
that for any $\epsilon > 0$,
\begin{equation*}
  \begin{split}
    \inf \mathcal{E}_3'(\fff) \ge \inf \mathcal{E}_1(\fff) + \inf
    \mathcal{E}_2(\fff) + 3 + kf + 2b - \epsilon,\\
    \mathcal{E}_3'(\ff) =
    \inf \mathcal{E}_1(\ff) + \inf \mathcal{E}_2(\ff) + 1 + k + kn - \epsilon,
  \end{split}\end{equation*}
and thus by \eqref{eq:real-index-set} the acutal index set
$\mathcal{E}_3$ of $A_3$ satisfies \eqref{eq:composition-index-sets},
and the proof is complete.
\end{proof}

 \end{appendix}

\end{document}